\documentclass[11pt]{amsart}

\usepackage{amsmath,amsthm,amssymb,graphics,color}
\usepackage{hyperref} 
\usepackage{xcolor}
\usepackage[english]{babel}
\usepackage{tikz}
\usepackage[]{geometry}
\usetikzlibrary{quotes,angles}
\usepackage{subcaption}
\usepackage{paralist}
\usepackage{graphicx}
\usepackage{esvect}
\usepackage{float}
\usepackage{array}
\usepackage{esvect}
\usepackage{float}
\usepackage{array}
\usepackage{tensor}
\usepackage{amsthm}
\usepackage{amsfonts}
\usepackage{bbm}

\newcommand{\C}{\mathbb{C}}

\newcommand{\Z}{\mathbb{Z}}
\newcommand{\N}{\mathbb{N}}

\newcommand{\R}{\mathbb{R}}

\newcommand{\dist}{\text{dist}}
\renewcommand{\phi}{\varphi}
\renewcommand{\epsilon}{\varepsilon}

\newtheorem{theo}{Theorem}[section]
\newtheorem{prop}[theo]{Proposition}
\newtheorem{coro}[theo]{Corollary}
\newtheorem{lemm}[theo]{Lemma}

\theoremstyle{definition}
\newtheorem{def1}[theo]{Definition}
\theoremstyle{remark}
\newtheorem{rema}[theo]{Remark}

\newcommand{\nwc}{\newcommand}
\nwc{\Oph}{\operatorname{Op}_\hbar}
\nwc{\la}{\langle}
\nwc{\ra}{\rangle}

\nwc{\mf}{\mathbf} %Latex (as in \bf not tilted math letters)
\nwc{\blds}{\boldsymbol} %Latex 
\nwc{\ml}{\mathcal} %Latex

\renewcommand{\Re}{\operatorname{Re}}

\newcommand{\wt}{\widetilde}

\renewcommand{\d}{\partial}

\renewcommand{\phi}{\varphi}

\newcommand{\red}[1]{{\color{red}{#1}}}

\newcommand{\black}[1]{\color{black}}

\title{The wave trace and Birkhoff billiards}
\date{}

\author{Amir Vig}
\address{Department of Mathematics, UC Irvine, Irvine, CA 92697, USA} \email{bvig@uci.edu}

\begin{document}
\maketitle

\begin{abstract}
The purpose of this article is to develop a Hadamard-Riesz type parametrix for the wave propagator in bounded planar domains with smooth, strictly convex boundary. This parametrix then allows us to rederive an oscillatory integral representation for the wave trace appearing in \cite{MaMe82} and compute its principal symbol explicitly in terms of geometric data associated to the billiard map. This results in new formulas for the wave invariants. The order of the principal symbol, which appears to be inconsistent in the works of \cite{MaMe82} and \cite{Popov1994}, is also corrected. In those papers, the principal symbol was never actually computed and to our knowledge, this paper contains the first explicit formulas for the principal symbol of the wave trace. The wave trace formulas we provide are localized near both simple lengths corresponding to nondegenerate periodic orbits and degenerate lengths associated to one parameter families of periodic orbits tangent to a single rational caustic. Existence of a Hadamard-Riesz type parametrix with explicit symbol and phase calculations in the interior appears to be new in the literature, with the exception of the author's previous work \cite{Vig18} in the special case of elliptical domains. This allows us to circumvent the symbol calculus in \cite{DuGu75} and \cite{HeZe12} when computing trace formulas, which are instead derived from integrating our explicit parametrix over the diagonal.
%The wave trace is a distribution on $\mathbb{R}$ given by $\sum_{j = 1}^\infty e^{it \lambda_j}$, where $\lambda_j^2$ are the (positive) eigenvalues of the Laplacian on a compact domain. In general, two linear waves can be superimposed to give another solution to the wave equation. When we add up a bunch of waves at different frequencies, the peak singularities appear at points with substantial constructive interference. On a manifold, the famous ``propagation of singularities” tells us that waves propagate along geodesics, so the constructive interference is most pronounced along orbits which are traversed infinitely often (periodic orbits). For the wave trace, this phenomenon is reflected in the Poisson relation, which says that the singular support of the wave trace is contained in the length spectrum (the collection of lengths of all periodic orbits). For planar domains, the geodesic flow is replaced by the billiard (or broken bicharacteristic) flow and we see an interesting connection between dynamical and spectral properties of the domain. In this talk, we outline a method for analyzing the singularities of the wave trace near lengths corresponding to orbits of rational rotation number. This involves proving a dynamical theorem on the structure of such orbits and then constructing an explicit Fourier integral operator, which microlocally approximates the wave propagator in the interior.
\end{abstract}

\section{Introduction}\label{Introduction section}
\noindent The purpose of this paper is to develop a Hadamard-Riesz type parametrix for the wave propagator in bounded planar domains with smooth, strictly convex boundary. We then use this parametrix to produce asymptotic expansions for the distributional wave trace near isolated lengths in the length spectrum. Let $\Omega$ be such a domain and denote by $\Delta$ the Dirichlet Laplacian on $\Omega$. The even wave propagator $E(t)$ is defined to be the solution operator for the wave equation
%\marginpar{\red{Is it easier to just say $\Delta$ the Dirichlet Laplacian? This looks a little awkward}}
\begin{align}\label{wave equation}
\begin{cases}
(\d_t^2 - \Delta) E = 0 & (x \in \Omega)\\
E(0) = \text{Id} & (x \in \Omega)\\
\d_t E \big|_{t = 0} = 0 & (x \in \Omega),
%E \big|_{\d \Omega} = 0, & (x \in \d \Omega),
\end{cases}
\end{align}
with Dirichlet boundary conditions. In spectral theoretic terms, we can write $E(t) = \cos t \sqrt{-\Delta}$, which is the even part of the half wave propagator $e^{i t \sqrt{-\Delta}}$. In \cite{Ch76}, it is shown that there exist Lagrangian distributions $E_j(t,x,y)$ such that microlocally away from the tangential rays, the Schwartz kernel of $E(t)$ is given by
\begin{align}\label{Chazarain cosine}
E(t,x,y) = \sum_{j = - \infty}^{+\infty} E_j(t,x,y) + C^\infty(\R \times \Omega \times \Omega),
\end{align}
with $E_j$ corresponding to a wave of $j$ reflections at the boudary. The sum in \eqref{Chazarain cosine} is locally finite in time. If we restrict our attention to waves which make $j$ reflections and travel approximately once around the boundary, we have the following explicit parametrix for $E_j(t,x,y)$.
\begin{theo}\label{HRP}
Let $\Omega \subset \R^2$ be a bounded domain with smooth, strictly convex boundary. Then there exists $j_0 = j_0(\Omega)\in \N$ sufficiently large such that the following holds: for all $j \geq j_0$, there exists a tubular neighborhood $U_j$ of the diagonal of the boundary $\Delta \d \Omega \subset \Omega \times \Omega$ such that for all $(x, y) \in U_j$ and $t$ less than but sufficiently close to $|\d \Omega|$,
$$
E_j(t,x,y) = \sum_{\pm} \sum_{k = 1}^8 (-1)^j e^{\pm i \pi / 4} \int_{0}^\infty e^{\pm i\tau(t - \Psi_j^k(x,y))} b_{j, k, \pm} (\tau, x, y) d\tau + C^\infty(\R \times U_j),
$$
microlocally near geodesic loops of rotation number $1/j$. Here, $b_{j, k, \pm} \in S_{\text{cl}}^{1/2}(U_j \times \R^1)$ are classical elliptic symbols of order $1/2$ and the functions $\Psi_j^k(x,y)$ ($1 \leq k \leq 8$) are lengths of the $8$ billiard orbits connecting $y$ to $x$ in $j$ reflections and approximately one rotation (see Theorem \ref{8 orbit lemma}). The principal term in the asymptotic expansion for $b_{j,k,\pm}$ is given in boundary normal coordinates $x = (\mu, \phi)$, $y = (\nu, \theta)$ (see Section \ref{Oscillatory integral representation}) by
\begin{align*}
\frac{\tau^{1/2}}{2(1 - \mu \kappa)}  \Bigg|\frac{\d \Psi_j^k}{\d \mu} \frac{\d \Psi_j^k}{\d \theta}\frac{\d^2 \Psi_j^k}{\d \nu \d \phi} &+ \frac{\d \Psi_j^k}{\d \phi}\frac{\d \Psi_j^k}{\d \nu}\frac{\d^2 \Psi_j^k}{\d \theta \d \mu}\\ &-     \frac{\d \Psi_j^k}{\d \mu}\frac{\d \Psi_j^k}{\d \nu}\frac{\d^2 \Psi_j^k}{\d \theta \d \phi}       - \frac{\d \Psi_j^k}{\d \phi} \frac{\d \Psi_j^k}{\d \theta} \frac{\d^2 \Psi_j^k}{\d \nu \d \mu} \Bigg|^{1/2},
\end{align*}
where $\kappa$ is the curvature of $\d \Omega$ at $(0, \phi)$.
\end{theo}
%\noindent \red{The leading order terms in the asymptotic expansion for $a_{j,k,\pm}$ are computed in Theorem \ref{parametrix with sum}.}
\noindent Theorem \ref{HRP} bears a remarkable resemblance to Hadamard's parametrix for the wave propagator on boundaryless manifolds (see \cite{Hadamardparametrix1}, \cite{Hadamardparametrix1}), where the phase functions $\pm \tau(t - \Psi_j^k(x,y))$ are replaced by $\pm \tau(t - r(x,y))$, with $r(x,y)$ the geodesic distance from $x$ to $y$. In that case, for $x,y$ near the diagonal, there is only one geodesic connecting $y$ to $x$ for small time. In our setting, Theorem \ref{8 orbit lemma} shows that there are exactly $8$ orbits connecting $y$ and $x$ in $j$ reflections and approximately one rotation. The parametrix in Theorem \ref{HRP} is localized strictly away from but near the glancing set. We do not treat contributions of glancing orbits in this paper.
\begin{rema}
For $j \geq j_0(\Omega)$ and $x, y \in \d \Omega$ sufficiently close to the diagonal, there exist only two orbits connecting $y$ to $x$ in $j$ reflections and approximately one rotation. One is in the clockwise direction and the other is in the counterclockwise direction. In particular, when $x = y \in \d \Omega$, there exists a unique geodesic loop based at $x$ of rotation number $1/j$. Restricting the functions $\Psi_j^k(x,y)$ appearing in Theorem \ref{HRP} to the diagonal of the boundary yield the $j$\textit{-loop function}, which we denote by $\Psi_j(q,q)$ (see Definitions \ref{psi jk} and \ref{jloop}).
\end{rema}
\noindent As in the case of boundaryless manifolds, we can use the explicit parametrix in Theorem \ref{HRP} to prove trace formulas. It is known that $E(t)$ has a well defined distributional trace
\begin{align}\label{Wave trace}
\text{Tr} \cos t \sqrt{-\Delta} = \sum_{j = 0}^\infty \cos t \lambda_j,
\end{align}
where $\lambda_j^2$ are the Dirichlet eigenvalues of $-\Delta$. The sum in \eqref{Wave trace} converges in the sense of tempered distributions and has singular support contained in the length spectrum
$$
\text{LSP}(\Omega) = \overline{ \{\text{lengths of periodic billiard trajectories} \} } \cup \{0\},
$$
together with $- \text{LSP}(\Omega)$ (see Section \ref{Computing the wave trace}). Each periodic billiard orbit in $\Omega$ can be classified according to its winding number $m$ and the number of reflections $n$ made at the boundary. Denote the collection of periodic orbits of this type by $\Gamma(m,n)$, normalized so that $m \leq n/2$. $\Gamma(m,n)$ is never empty by a theorem of Birkhoff \cite{Birkhoff}. The length spectrum can be decomposed accordingly as
\begin{align}
\text{LSP}(\Omega) = \bigcup_{m, n \in \N} \text{length}(\Gamma(m,n)) \cup \N |\d \Omega|,
\end{align}
where $|\d \Omega|$ is the length of the boundary. Using the parametrix in Theorem \ref{HRP}, we can prove the following theorem.
\begin{theo}\label{Main theorem}
Assume the conditions from Theorem \ref{HRP} hold and $\Omega$ satisfies the noncoincidence condition:
\begin{align}\label{NCC}
\begin{split}
\text{there exists}\,\,\epsilon_0 > 0\,\, \text{such that} \,\,\bigcup_{\substack{m \geq 2\\ n \geq 1}} \text{length}(\Gamma(m,n)) \cap (|\d \Omega| - \epsilon_0, |\d \Omega|) = \emptyset.
\end{split}
\end{align}
For $j \geq j_0$, define $t_j = \inf_{q \in \d \Omega} \Psi_j(q,q)$ and $T_j = \sup_{q \in \d \Omega} \Psi_j(q,q)$. Then on any sufficiently small neighborhood of $[t_j,T_j]$, $\text{Tr} E(t)$ has the asymptotic expansion %\marginpar{\red{No need for real part in this formula... integrand is even in $\xi$ so its Fourier tranform (in $\xi$) is real! It does make sense to put the real part in the next two theorems where we actually integrate out $\xi$ to obtain a homogeneous distribution.}}
\begin{align*}
(-1)^j \Re \left\{e^{i \pi / 4} \int_{\d \Omega} \int_{0}^\infty e^{i \xi (t- \Psi_j(q,q) )} a^j(q,\xi)  d\xi dq\right\} + C^\infty(\R),
\end{align*}
where $a^{j}(q,\xi)$ is a classical elliptic symbol of order $1/2$ with principal part given by
\begin{align*}
a_0^j(q,\xi) =   {4}\xi^{1/2} \sin \omega_{j, 1}(q,q) \sin^{1/2} \omega_{j,2}(q,q) \left|\frac{\d \omega_{j,1}}{\d q'} (q,q) \right|^{1/2} X(q) \cdot N(q).
\end{align*}
Here, $X(q)$ is the position vector to a boundary point $q$ with respect to a fixed origin and $N(q)$ is the outward unit normal at $q$. The angles $\omega_{j,1}, \omega_{j,2}$ are the initial and final angles respectively of the unique billiard orbit $\gamma_j(q,q')$ which connects nearby boundary points $q$ and $q'$ in $j$ reflections and approximately one counterclockwise rotation. The function $\Psi_j(q,q')$ is the length of $\gamma_j(q,q')$ and its restriction to the diagonal is the $j$-loop function. % The quantity $\d \omega_{j,1} /\d q'$ is the $q'$ derivative of the initial angle corresponding to the unique orbit of $j$ reflections and approximately one rotation connecting (off diagonal) boundary points $q,q' \in \d \Omega$, evaluated at the diagonal $q' = q$.
\end{theo}
%\marginpar{\blue{is it not confusing that Theorem 1.1 has $8$ orbits and Theorem 1.2 has only one orbit? Actually theorem 1.2 has two orbits one clockwise and one is counterclockwise...}}
\begin{rema}
As the position vector $X$ is chosen with respect to an arbitrary interior point $p$, we can integrate out this symmetry over any measurable subset of the interior in the variable $p$ to obtain a more invariant formula. In particular, we can integrate over open sets, curves and by a limiting argument, the boundary itself in order to obtain a smooth density on $\d \Omega$.  The noncoincidence condition \eqref{NCC} can be weakened and is in particular satisfied for ellipses (see \cite{GuMe79a}) and nearly circular domains (see \cite{HeZe19}).
\end{rema}
\begin{rema}
In several cases, one can evaluate the integral appearing in Theorem \ref{Main theorem} more explicitly. If $L_j$ is an isolated length and the corresponding orbit is nondegenerate or the fixed point set of the time $L_j$ billiard flow is clean in the sense of Bott-Morse (see \cite{DuGu75}), then one can apply the method of stationary phase. The case of a one parameter family of orbits tangent to a rational caustic is discussed below. There is an apparent asymmetry between the incident and reflected angles in the symbol $a_0^j(q,\xi)$, but only periodic orbits contribute in the stationary phase computation. For periodic orbits, we have Snell's law $\omega_{j, 1} = \omega_{j,2}$. The integral formulas in Theorem \ref{Main theorem} are valid regardless of how complicated the structure the length spectrum and corresponding orbits may be.
\end{rema}
\begin{rema}
The noncoincidence condition \ref{NCC} was first formulated by Marvizi and Melrose. It is known that the set of domains satisfying this condition is $C^\infty$ dense in the set of all smooth, bounded strictly convex planar domains and moreover contains a $C^1$ neighborhood of the disk (Proposition 7.2, \cite{MaMe82}). It is also satsified for ellipses (Proposition 4.3, \cite{GuMe79a}). It is believed by the author to be satisfied by all smooth, bounded, strictly convex planar domains.
\end{rema}
\noindent Theorem \ref{Main theorem} provides an explicit formula for the principal term in the parametrix developed in \cite{MaMe82}. In contrast to the methods employed in \cite{MaMe82}, the proof developed in the remainder of this paper uses the explicit parametrix for the wave propagator appearing in Theorem \ref{HRP}. Theorem \ref{Main theorem} also provides clarity on a discrepancy in the literature regarding the order of the wave trace (cf. \cite{Popov1994}, \cite{MaMe82}). Note that in Theorem \ref{Main theorem}, no assumptions are made on the nondegeneracy of orbits. If the length spectrum has high multiplicity, the wave trace is in general quite complicated. However, when periodic orbits come in a one parameter family corresponding to a caustic, we have the following: %That such lengths $L_j$ exist in abundance is a consequence of Birkhoff's theorem on periodic orbits (see \cite{Birkhoff}).
\begin{theo}\label{Rational caustic theorem}
Suppose $\Omega$, $j \geq j_0$ are as in Theorem \ref{Main theorem} and $\mathcal{C}$ is a caustic for $\Omega$ of rotation number $1/j$. If periodic orbits tangent to $\mathcal{C}$ have length $L_j$, then near $t = L_j$, $\text{Tr} E(t)$ has the leading asymptotic
\begin{align*}
c_j \Re \left\{e^{i\pi/4} (t- L_j - i 0)^{-3/2}\right\},
\end{align*}
where $c_j$ is a wave invariant given by the formula
\begin{align*}
c_j = (-1)^{j+1} {4} \int_{\d \Omega} \sin^{3/2} \omega_{j,1}(q,q) \left|\frac{\d \omega_{j,1}}{\d q'}(q,q)\right|^{1/2} X(q) \cdot N(q) dq.
\end{align*}
In this case, $\omega_{j,1}(q,q) = \omega_{j,2}(q,q)$ is the measure of the angle of incidence for the unique periodic orbit of rotation number $1/j$ based at $q \in \d \Omega$.
\end{theo}
\noindent While KAM theory provides the existence of irrational caustics, it is shown in \cite{KaloshinZhangRationalCaustics} that for a fixed $j$, the set of all smooth convex domains possessing a rational caustic of rotation number $1/j$ is polynomially dense in the variable $1/j$ within the collection of all smooth strictly convex domains, equipped with the $C^\infty$ topology. Exponential density is also proven in the analytic category. As ellipses satisfy the noncoincidence condition \eqref{NCC} (see \cite{GuMe79a}) and are known to be completely integrable with confocal conic sections as caustics, we have the following corollary to Theorem \ref{Rational caustic theorem}.
\begin{coro}\label{Ellipse corollary}
For an ellipse $\Omega$ given by
$$
\Omega = \left\{(x,y): \frac{x^2}{a^2} + \frac{y^2}{b^2} \leq 1\right\},
$$
and $L_j \in \text{LSP}(\Omega)$ sufficiently close to $|\d \Omega|$ corresponding to the length of billiard orbits of rotation number $1/j$, the wave invariants in Theorem \ref{Rational caustic theorem} are given by
\begin{align*}
c_j = \int_0^{2\pi} \frac{(-1)^{j+1} 2 ab \sin \omega_j \sqrt{a^2 \cos^2 \phi + b^2 \sin^2 \phi} d\phi }{\sqrt{\cos \omega_j (a^2 \sin^2 \phi + b^2 \cos^2\phi) (b^2 + (a^2 - b^2) \sin^2 \phi) G(\zeta_j) \sqrt{1 - k_{\zeta_j}^2 \sin^2 \phi}}}.
\end{align*}
Here, $\zeta_j \in [0,b)$ is the parameter of the confocal ellipse
$$
\mathcal{C}_{\zeta_j} = \left\{(x,y): \frac{x^2}{a^2 - \zeta_j^2} + \frac{y^2}{b^2 - \zeta_j^2} = 1\right\},
$$
to which the orbits of length $L_j$ are tangent and $k_{\zeta_j}$ is given by
$$
k_{\zeta_j}^2 = \frac{a^2 - b^2}{a^2 - {\zeta_j}^2}.
$$
The parameter $\zeta_j$ depends on the rotation number ${1/j}$ and is defined implicitly by the equation
\begin{align*}
\frac{1}{j} = \frac{F(\arcsin \zeta_j/b; k_{\zeta_j})}{2 K(\zeta_j)},
\end{align*}
where $F(s; k)$ is the elliptic integral
\begin{align*}
\int_0^{s} \frac{d \tau}{\sqrt{1 - k^2 \sin^2 \tau }}
\end{align*}
and $K(\zeta_j) = F(\pi/2; k_{\zeta_j})$. The function $G(\zeta_j)$ is defined by
\begin{align*}
\frac{-k_{\zeta_j}^2}{(a^2 - \zeta_j^2)} \int_0^{2\pi} \frac{\sin^2\tau d\tau}{(1 - k_{\zeta_j}^2 \sin^2 \tau)^{3/2}} + (2j+2) \frac{d}{d \zeta^2} F(\arcsin \zeta / b; k_{\zeta})\Big|_{\zeta = \zeta_j},
\end{align*}
and $\omega_j = \omega_j(\phi)$ are the angles of reflection for orbits tangent to $\mathcal{C}_{\zeta_j}$, given implicitly by the equation
\begin{align*}
\zeta_j^2 = \sin^2 \omega_j (b^2 + (a^2 - b^2)\sin^2 \phi).
\end{align*}
\end{coro}

\begin{rema}
Analagous formulas to those appearing in Theorem \ref{Main theorem}, Theorem \ref{Rational caustic theorem} and Corollary \ref{Ellipse corollary} can also be proved for the Neumann and Robin wave traces as well. The formulas are less succinct but can be easily reproduced by keeping track of boundary terms in Section \ref{Computing the wave trace}. Alternatively, an earlier version of this paper used Hadamard type variational formulas which were derived for Robin boundary conditions in the author's previous work \cite{Vig18}.
\end{rema}

%A dynamical theorem (Theorem \ref{8 orbit lemma}) on the structure of approximate geodesic loops making a fixed number of reflections and rotations is also presented there. 

\subsection{Schematic Outline}
The proofs of Theorem \ref{Main theorem}, Theorem \ref{Rational caustic theorem} and Corollary \ref{Ellipse corollary} use techniques from \cite{Vig18}, which are reviewed throughout the paper. In Section \ref{Background}, we review relevant background on the inverse spectral problem, i.e. determining geometric information from the Laplace spectrum. The first step in the proof of Theorem \ref{Main theorem} is to generalize the Hadamard-Riesz type parametrix for the wave propagator constructed in \cite{Vig18} for ellipses to arbitrary bounded domains with strictly convex, smooth boundary as in Theorem \ref{HRP}. This requires a dynamical classification (Theorem \ref{8 orbit lemma}) of the cardinality and structure of all billiard orbits connecting interior points with a fixed number of reflections, analagous to Lemma 5.2 in \cite{Vig18}. Section \ref{Billiards} introduces language from dynamical systems necessary to describe the billiard (or broken bicharacteristic) flow, which later appears in the canonical relations of the wave propagator. The proof of Theorem \ref{8 orbit lemma} is relegated to Section \ref{Proof of 8 orbit lemma}, where it is first done in the simple case of the Friedlander model and then broken up into several intermediate lemmas using Lazutkin coordinates for the general case. This material is largely independent from the rest of the paper and is of separate interest from the perspective of dynamical billiards, irrespective of applications to spectral theory. In Section \ref{A parametrix for the wave propagator}, the length functionals corresponding to these orbits allow us to cook up explicit phase functions which parametrize the canonical relations for the wave propagator $e^{i t \sqrt{-\Delta}}$, which is a Fourier integral operator microlocally away from the tangential rays. We then carry out the analysis leading to the microlocal parametrix appearing in Theorem \ref{HRP}.  Construction of a parametrix for the wave propagator in the interior, with principal symbol given explicitly in terms of geometric data, microlocally near transversally reflected, nearly glancing rays appears to be new in the literature, with the exception of the author's previous work \cite{Vig18} in the special case of an ellipse. In Section \ref{Computing the wave trace}, an integration by parts allows us to compute the localized in time wave trace in terms of a boundary integral. In this case we can argue that only a select few of the billiard orbits in Theorem \ref{8 orbit lemma} contribute to the wave trace. Appropriate Maslov factors on each branch of the canonical relations are also computed here. As the order of the principal symbol computed in Section \ref{Computing the wave trace} appears to contradict other works in the literature, Section \ref{Check on the order} provides an auxilliary confirmation via stationary phase that the order derived here is indeed correct. To our knowledge, this paper contains the first explicit formulas for the principal symbol of the wave trace associated to a convex billiard table near the length of the boundary.

\section{Background}\label{Background}
\noindent The inverse spectral problem has a long history, dating back to Kac in 1966 \cite{Kac66}, who asked the famous question ``can one hear the shape of a drum?'' Mathematically, this corresponds to uniquely determining a domain $\Omega$ from the spectrum of its Dirichlet, Neumann or Robin Laplacian. For bounded domains, the spectrum is purely discrete, consisting of eigenvalues $\lambda_j^2$ satisfying
\begin{align}\label{inverse problem}
\begin{cases}
-\Delta u_j = \lambda_j^2 u_j & x \in \Omega\\
Bu_j = 0 & x \in \d \Omega,
\end{cases}
\end{align}
where $u_j$ are smooth eigenfunctions on $\Omega$ and $B$ is either the restriction operator (Dirichlet boundary conditions), normal differentiation (Neumann boundary conditions) or normal differentiation plus a prescribed function on $\d \Omega$ (Robin boundary conditions). A variety of approaches in local and global harmonic analysis have been taken to prove partial results in the direction of \cite{Kac66}. One particularly useful strategy is to use the wave group to deduce spectral information about the underlying geometric space, usually a Riemannian manifold. The motivation behind this approach stems from Duistermaat and H\"ormander's propagation of singularities theorem, which says that singularities of solutions to the wave equation propagate along (possibly broken) bicharacteristics, which are lifts to $T^* (\R \times \Omega)$ of geodesic or billiard orbits. As linear waves can be superimposed, constructive interference is most pronounced along geodesics which are traversed infinitely often, i.e. periodic orbits.  On the trace side, this is reflected in the Poisson relation:
\begin{align}\label{Poisson relation}
\text{SingSupp Tr} e^{i t \sqrt{-\Delta} } \subset \pm \overline{\text{LSP} (\Omega)} \cup \{0\},
\end{align}
where the lefthand side is the distributional trace of the half wave propagator (see Section \ref{Computing the wave trace}) and the righthand side is the length spectrum of $\Omega$ (the closure of all lengths of periodic geodesic or billiard orbits together with $\{0\}$). This is proven in \cite{AndersonMelrose} for smooth, strictly convex planar domains with boundary and \cite{PeSt92} for more general bounded domains. In particular, the formula \eqref{Poisson relation} generalizes the Poisson summation formula on the torus $\R^n/\Z^n$ from elementary Fourier analysis (see \cite{Ur98}, for example).
\\
\\
Asymptotic formulas near the singularities are given by the Selberg trace formula in the case of hyperbolic surfaces \cite{Sel56}, the Duistermaat-Guillemin trace theorem for general smooth manifolds under a dynamical nondegeneracy condition \cite{DuGu75}, and a Poisson summation formula for strictly convex bounded planar domains due to Guillemin and Melrose \cite{GuMe79b}. However, since these trace formulas involve sums over all periodic orbits of a given length, it is possible that the contributions of distinct orbits having the same length could cancel out and the wave trace is actually smooth near a point in the length spectrum. We say that the length $L \in \R$ of a periodic orbit $\gamma$ is simple if up to time reversal ($t \mapsto -t$), $\gamma$ is the unique periodic orbit of length $L$. Without length spectral simplicity, there is no way to deduce Laplace spectral information from the length spectrum alone. It is shown in \cite{PeSt92} that generically, smooth convex domains have simple length spectrum associated to only nondegenerate periodic orbits. In that case, the following theorem holds:
\begin{theo}[\cite{GuMe79b}, \cite{PeSt17}]
Assume $\gamma$ is a nondegenerate periodic billiard orbit in a bounded, strictly convex domain with smooth boundary and $\gamma$ has length $L$ which is simple. Then near $L$, the even wave trace has an asymptotic expansion
\begin{align}
\text{Tr} \cos t \sqrt{-\Delta} \sim \Re \left\{ a_\gamma (t - L + i0)^{-1} + \sum_{k = 0}^\infty a_{\gamma k} (t - L + i0)^{k} \log(t - L + i0) \right\},
\end{align}
where the coefficients $a_{\gamma k}$ are the wave invariants associated to $\gamma$.
\end{theo}
\noindent Calculations of the wave invariants associated to dynamically convenient orbits have proved extremely useful in the inverse spectral problem associated to \eqref{inverse problem}. For example, the case of rotationally symmetric metrics on $S^2$ is analyzed in \cite{Zelditch2}. In \cite{Zelditch3}, \cite{Zelditch4} and \cite{Zelditch5}, the wave invariants associated to bouncing ball orbits are explicitly calculated using Feynmann diagrams to analyze the stationary phase computation from which Balian-Bloch (resolvent) formulas are derived (see also \cite{Zelditch1}). Under mild dynamical conditions and some additional axial symmetry assumptions, these coefficients can be used to determine the Taylor series of a local boundary parametrization. In particular, this allows one to deduce that such domains are spectrally determined amongst a rich class of analytic, symmetric domains. Microlocal parametrices near the glancing set have also been constructed in \cite{AndersonMelrose}, \cite{PeSt17}, \cite{Eskin} for demonstrating propagation of singularities, \cite{MT1}, \cite{MT2} in the context of scattering by a convex obstacle and more recently \cite{ILP1}, \cite{ILP2} for proving dispersive estimates. However, there is a lack of precise information on their principal symbols in terms of geometric data and to the author's knowledge, the contributions these tangential rays to the wave trace have not yet been considered in the context of inverse problems.
\\
\\
The wave invariants have also proved useful in variational inverse problems, going back to the seminal papers \cite{GK1} and \cite{GK2}, where the authors proved spectral rigidity for closed manifolds with negative sectional curvature. This was recently generalized to Anosov surfaces in \cite{UPS}. In the setting of bounded domains, it is proved in \cite{HeZe12} that ellipses with Dirichlet/Neumann boundary conditions are infinitessimally spectrally rigid through smooth domains with the same symmetries. These results as well as spectral determination of the Robin function in \cite{GuMe79a} were generalized to Liouville billiard tables of classical type in \cite{PopTop03} and \cite{PopTop12}. They were also extended to the Robin setting in \cite{Vig18}, where both the domain and Robin function were allowed to vary simultaneously. A recent breakthrough was obtained in \cite{HeZe19}, where the authors showed that ellipses of small eccentricity are spectrally determined. Thorough surveys of the inverse spectral problem are contained in \cite{ZelditchSurvey2014}, \cite{ZelditchSurvey2}, \cite{HezariDatchevSurvey} and \cite{Melrosesurvey}.
\\
\\
The present article is inspired by \cite{MaMe82}, where the following theorem is proved:
\begin{theo}[\cite{MaMe82}]\label{Melrose theorem}
If $\Omega$ is a bounded and strictly convex planar region, there exists $N = N(\Omega)$ such that if $j > N$, then the contribution $\widehat{\sigma}_j$ of $E_j$ to $\widehat{\sigma}_D$ is of the form
\begin{align}\label{Melrose 1}
\widehat{\sigma}_j = \frac{1}{2\pi} \int_0^\infty \int_{0}^L e^{i (t - \mu_j(s)) \tau } a(s,\tau) ds d\tau, 
\end{align}
where in terms of an arclength coordinate $s$ on $\d \Omega$,
\begin{align*} %\label{Melrose 2}
\mu_j(s) = \Psi_j(s,s')\big|_{s = s'},
\end{align*}
with $\Psi(s,s') = L(\mathfrak{g})$, with $\mathfrak{g}$ the length of a $j$-fold geodesic from $s$ to $s'$ and $a_j$ is periodic in $s$, classical and elliptic of order zero, with principal part of the form $e^{i\pi r_j/4} \alpha_j(s)$, $\alpha_j(s) > 0$.
\end{theo}
\noindent Here, $E_j(t)$ is the cosine kernel associated to the parametrix for the $j$ reflection wave operator constructed in \cite{Ch76}, which is reviewed in Section \ref{Chazarain's Parametrix}. In particular, if $L_j$ is a simple length corresponding to a nondegenerate periodic orbit, then Theorem \ref{Melrose theorem} gives an asymptotic expansion for the localized wave trace. In fact, if $\Omega$ satisfies the noncoincidence condition \eqref{NCC}, i.e. $|\d \Omega|$ is not a limit point from below of the lengths of orbits of rotation number $m/n$ for $m \geq 2$, then the trace in Theorem \ref{Melrose theorem} is a spectral invariant. The purpose of this article is to explicitly calculate the principal symbol $a(s,\tau)$ of \eqref{Melrose 1} in terms of geometric data associated to the billiard map, both in the case of simple lengths corresponding to nondegenerate periodic orbits and also for one parameter families of degenerate periodic orbits,  all having the same length associated to a caustic of rational rotation number. It also corrects several errors in literature on the order of $a \in S_{\text{cl}}^{1/2}(\d\Omega)$.

\section{Billiards}\label{Billiards}
%\red{Things to add: Katok, Tabachnikov and Melrose-Guillemin, Lazutkin coordinates and Birkhoff normal forms, Lazutkin coordinates mention coordinate introduced by Denjoy \cite{Denjoy58}, \cite{Katok} also mention how it is an adaptation of KAM, cite Popov papers?},

\noindent Before obtaining a singularity expansion for the wave trace, we first review the relevant background needed on billiards. This will be useful in our discussion of Chazarain's parametrix in Section \ref{Chazarain's Parametrix}. In this section, we denote by $\Omega$ a bounded strictly convex region in $\R^2$ with smooth boundary. This means that the curvature of $\d \Omega$ is a strictly positive function. The billiard map is defined on the coball bundle of the boundary $B^* \d \Omega = \{(q, \zeta) \in T^*\d \Omega : |\zeta| < 1 \}$, which can be identified with the inward part of the circle (cosphere) bundle $S_{\d \Omega}^* \R{^2}$, via the natural orthogonal projection map. We can also identify $B^* \partial \Omega$ with ${\R}\slash {\ell \Z} \times (0, \pi)$, where $\ell=|\partial \Omega|$ is the length of the boundary. Define
\begin{align*}
t_{\pm}^1(y,\eta)& = \inf\{t > 0 : g^{\pm t}(y,\eta) \in  \d \Omega \},\\
t_{\pm}^{-1}(y,\eta) &= \sup \{t < 0 : g^{\pm t}(y,\eta) \in  \d \Omega \},
\end{align*}
where $\pi_1$ is projection onto the first factor and $g^{\pm t}$ is the forwards $(+)$ or backwards $(-)$ geodesic flow on $\R^2$, corresponding to the Hamiltonian $H_\pm = \pm|\eta|$ (see Section \ref{Chazarain's Parametrix}). %The times $t_{\pm}^j(y,\eta)$ are defined inductively for $j \in \Z$ to be the length of the link connecting the $j-1$st bounce to the $j$th bounce:
%\begin{align*}
%t_{\pm}^{j+1}(y,\eta)& = \inf\{t > 0 : \pi g^{t}(\beta^{j-1}(\widehat{g^{t_\pm^1}(y,\eta)})) \in  \d \Omega \},\\ %\quad j \in \Z,\\
%t_{\pm}^{-j-1}(y,\eta) &= \sup \{t < 0 : \pi g^t(\beta^{-(j-1)}(\widehat{g^{t_\pm^{-1}}(y,\eta)})) \in  \d \Omega \}, %\quad j \in \Z,
%\end{align*}
%\marginpar{\red{Make $\Z_{>0}$? Also this doesn't make sense - we need to reflect before flowing. Where are the $\lambda$s?}}
%We also define $T_\pm^n = \sum_{j = 1}^n t_\pm^j$ to be the total time of the flow from $(y,\eta)$ to $\beta^n(y,\eta)$.\\
If $(y,\zeta) \in B^* \d \Omega$ is mapped to the inward pointing covector $(y,\eta) \in T_{\d \Omega}^* \R^2$ under the inverse projection map, then we define
$$
\beta^{\pm1}(y,\eta) = \widehat{g^{t_\pm^1}(y,\eta)},
$$
where a point $\widehat{(x,\xi)}$ is the reflection of $\xi$ through the cotangent line $T_x^*\d \Omega$, i.e. $\widehat{(x,\xi)}$ has the same footpoint and (co)tangential projection as $(x,\xi)$, but reflected conormal component, so that it is again in the inward facing portion of the circle bundle. We call $\beta := \beta^{+1}$ the billiard map. It is well known that $\beta$ preserves the natural symplectic form induced on $B^*(\d \Omega)$ and is differentiable there, extending continuously up to the boundary. The maps $\beta^{\pm n}$ are defined via iteration and it is clear that $\beta^{-n} = (\beta^{n})^{-1}$ for each $n \in \Z$. Associated to the billiard map is the billiard flow, or broken bicharacteristic flow, which we denote by $\Phi^t$.\\
\\
Geometrically, a billiard orbit corresponds to a union of line segments which are called links. A smooth closed curve $\mathcal C$ lying in $\Omega$ is called a {caustic} if any link drawn tangent to $\mathcal C$ remains tangent to $\mathcal C$ after an elastic reflection at the boundary of $\Omega$. By elastic reflection, we mean that the angle of incidence equals the angle of reflection at an impact point on the boundary. We can map $\mathcal C$ onto the total phase space $B^* \partial \Omega$ to obtain a smooth closed curve which is invariant under $\beta$. If the dynamics are integrable, these invariant curves are precisely the Lagrangian tori which folliate phase space. A point $P$ in $B^*\partial \Omega$ is called $q$-periodic ($q \geq 2$) if $\beta^q(P)=P$. We define the rotation number of a $q$-periodic orbit $P$ by $\omega(P)= \frac{p}{q}$, where $p$ is the winding number of the orbit generated by $P$, which we now define. We may consider the modified billiard map $\widetilde{\beta} = \Pi^* \beta$, where $\Pi$ is the natural mapping from $\R / \ell \Z \times [0,\pi]$ to the closure of the coball bundle $\overline{B^*\d \Omega}$. Pulling back by $\Pi$ clearly preserves the notion of periodicity. There exists a unique lift $\widehat{\beta}$ of the map $\widetilde{\beta}$ to the closure of the universal cover $\R \times [0,\pi]$ which is continuous, $\ell$ periodic in first variable and satisfies $\widehat{\beta}(x,0) = (x,0)$. Given this normalization, for any point $(x,\theta) \in \R/\ell \Z \times [0,\pi]$ in a $q$ periodic orbit of $\widetilde{\beta}$, we see that $\widehat{\beta}^q(x,\theta) = (x + p \ell, \theta)$ for some $p \in \Z$. We define this $p$ to be the winding number of the orbit generated by $\Pi (x,\theta) \in \overline{ B^* \d \Omega}$. We see that even if a point $\Pi (x,\theta)$ generates an orbit which is not periodic in the full phase space but is such that $\pi_1 (\widetilde{\beta}^q(x,\theta)) = x$ for some $q \in \Z$, we can still define a winding number in this case. Such orbits are called loops or geodesic loops. For a given periodic orbit, the winding number is independent of which point in the orbit is chosen, so we sometimes write $\omega(\gamma) = \omega(P)$ for any $P \in \gamma = \{P, \beta(P), \cdots, \beta^{q-1}(P)\}$. For deeper results and a more thorough introduction to dynamical billiards, we refer the reader to \cite{Tabachnikov}, \cite{Katok}, \cite{Popov1994} and \cite{PopovTopalov}.
\\
\\
What will be crucial for us in later sections is a description of all orbits making a fixed number of reflections which connect interior points near the diagonal of the boundary in approximately one rotation. These orbits will allow us to cook up phase functions in Section \ref{Chazarain's Parametrix} which parametrize the canonical relation of the wave propagator.
\begin{theo}[8 Orbit Theorem]{\label{8 orbit lemma}}
There exist $C_0 > 0$ and $j_0 = j_0(\Omega)$ sufficiently large such that for $j \geq j_0$ and any two points $x,y \in \text{int}(\Omega)$ which are $C_0/j^4$ close to the diagonal of the boundary, there exist precisely four distinct, broken geodesics of $j$ reflections making approximately one counterclockwise rotation, emanating from $x$ and terminating at $y$. Similarly, there exist four such orbits in the clockwise direction. If $x,y \in \d \Omega$ and are $C_0/j^4$ close to the diagonal, there exists only one clockwise and one counterclockwise orbit connecting $x$ to $y$ in $j$ reflections and approximately one rotation. In particular, when $x = y \in \d \Omega$, there is a unique (up to time reversal) geodesic loop based at $x$ of rotation number $1/j$.
\end{theo}

\noindent The proof of Theorem \ref{8 orbit lemma} is based on several lemmas in Section \ref{Proof of 8 orbit lemma} below and is inspired by the author's previous work in \cite{Vig18}, where a similar construction was adapted to elliptical billiard tables. As in that paper, the proof actually provides more information on the topological structure of the orbits. The existence of orbits connecting nearby \textit{boundary} points and in particular geodesic loops of small rotation number is well known, although the material in Section \ref{Proof of 8 orbit lemma} below easily reproduces these results. The novelty of Theorem \ref{8 orbit lemma} is a complete description of orbits connecting \textit{interior points} as opposed to boundary points, which will ultimately allow us to extend microlocal parametrices for the wave propagator from the boundary (as in \cite{MaMe82}) to the interior.
\\
\\
We now explain what is meant by approximately one rotation. Let $\xi \in S_x^* \Omega$ be one of the $4$ covectors corresponding to the initial condition of a counterclockwise orbit described in Theorem \ref{8 orbit lemma}. Denote by $\widehat{x} = \pi_1 g^{t_1^+}(x,\xi)$ the first point of reflection at the boundary ($\pi_1$ is projection onto the first factor) and by $\widehat{y}$ the $(j+1)$st point of reflection at the boundary after the orbit reaches $y$. If $x, y$ are $O(j^{-1})$ close to the diagonal of the boundary, then $|\widehat{x} - \widehat{y}| = O(j^{-1})$ (see Section \ref{Proof of 8 orbit lemma}). Also let $\omega$ be the angle of reflection made by the orbit at $\widehat{x}$ and note that $\widehat{x}, \widehat{y}$ and $\omega$ all depend implicitly on $\xi$.
\begin{def1}
We say that an orbit makes \textit{approximately one counterclockwise rotation} if for each of the initial covectors $\xi \in S_x^*\Omega$ of the $4$ counterclockwise orbits provided by Theorem \ref{8 orbit lemma}, we have
\begin{equation*}
	|\pi_1 \widehat{\beta}^j(\widehat{x}, \omega) - \widehat{y} - \ell| \leq \ell/100.
\end{equation*}
\noindent Here, $\ell = |\d \Omega|$ and $\widehat{\beta}$ is the lift of the billiard map to the closure of the universal cover $\R \times [0,\pi]$ as described in Section \ref{Proof of 8 orbit lemma}. 
\end{def1}
\begin{rema}
The choice of $\ell/100$ is somewhat arbitrary, but having $\ell$ in the numerator allows for scale invariance and finding the optimal constant is irrelevant for our purposes. The notion of approximately one clockwise rotation is defined analogously.	
\end{rema}

\begin{def1}\label{TTNN}
Of the four counterclockwise orbits emanating from $x$, two of them become tangent to a level curve of the distance function $d(z) = \text{dist}(z, \d \Omega)$ before making a reflection at the boundary. We denote these orbits by $T$ orbits (for tangency) and call their first links $T$ links. The other two orbits make a reflection at the boundary before becoming tangent to a level curve of $d$ and we call these $N$ orbits (for nontangency); their first link is called an $N$ link. Within either $T$ or $N$ category for the first link, the final link of one of the orbits reaches $y$ before becoming tangent to a level curve (an $N$ link) and the other has a point of tangency before reaching $y$ (a $T$ link). In this way, we obtain four types of counterlockwise orbits from $x$ to $y$, which we denote by $TT$, $TN$, $NT$, and $NN$. The same classification also applies to the four clockwise orbits.
\end{def1}
\begin{figure}
\includegraphics[scale=0.25]{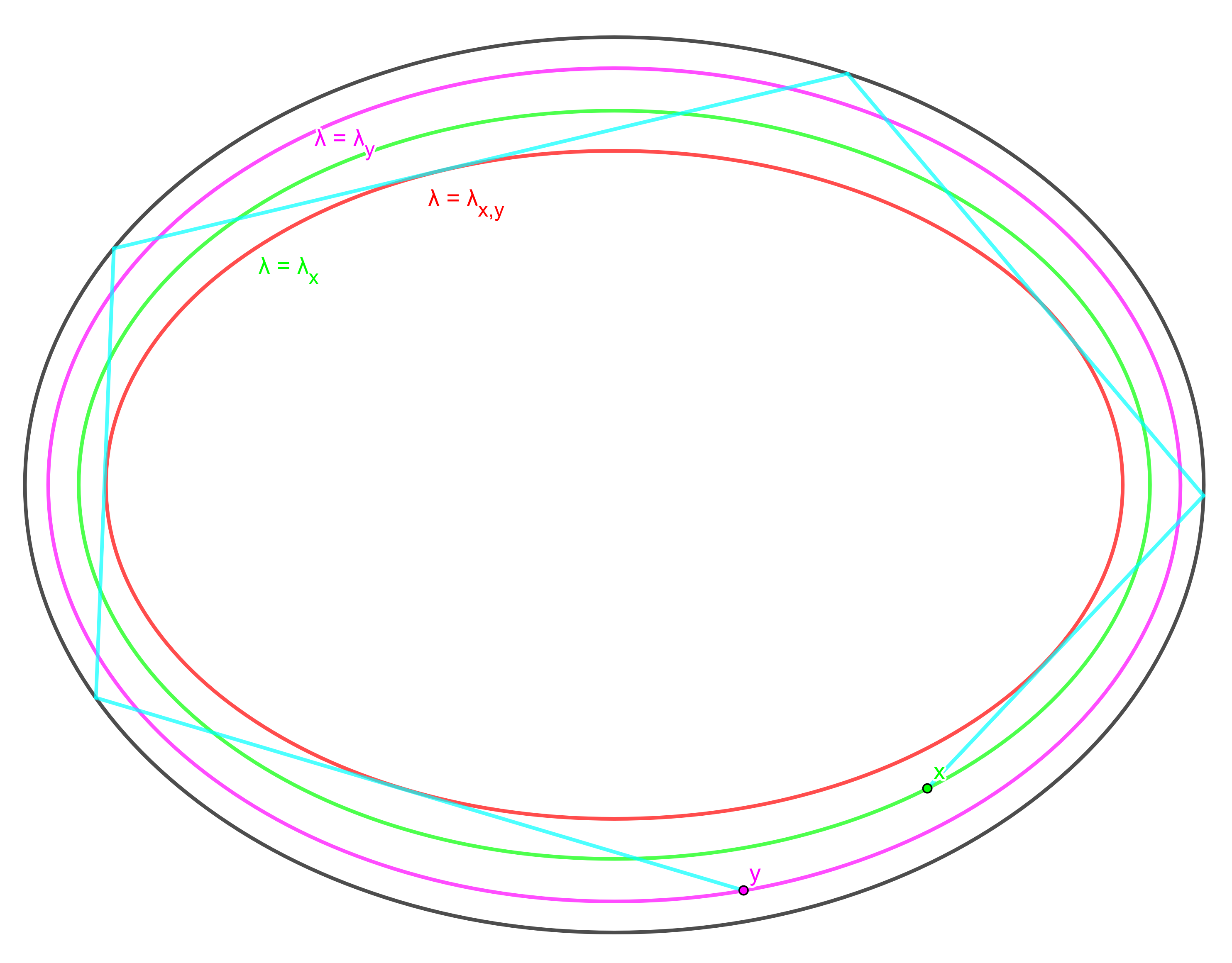}
\includegraphics[scale=0.25]{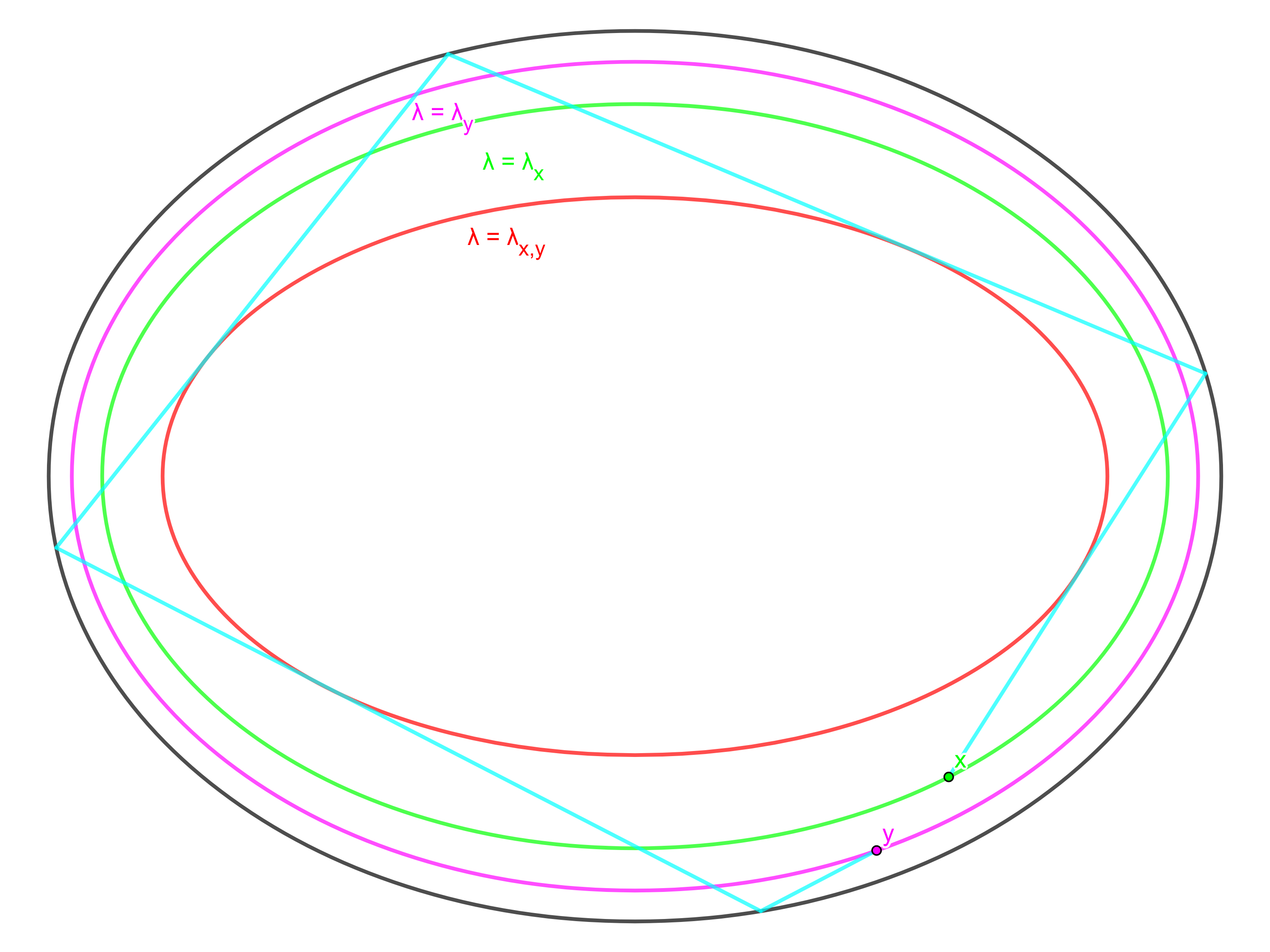}
\includegraphics[scale=0.25]{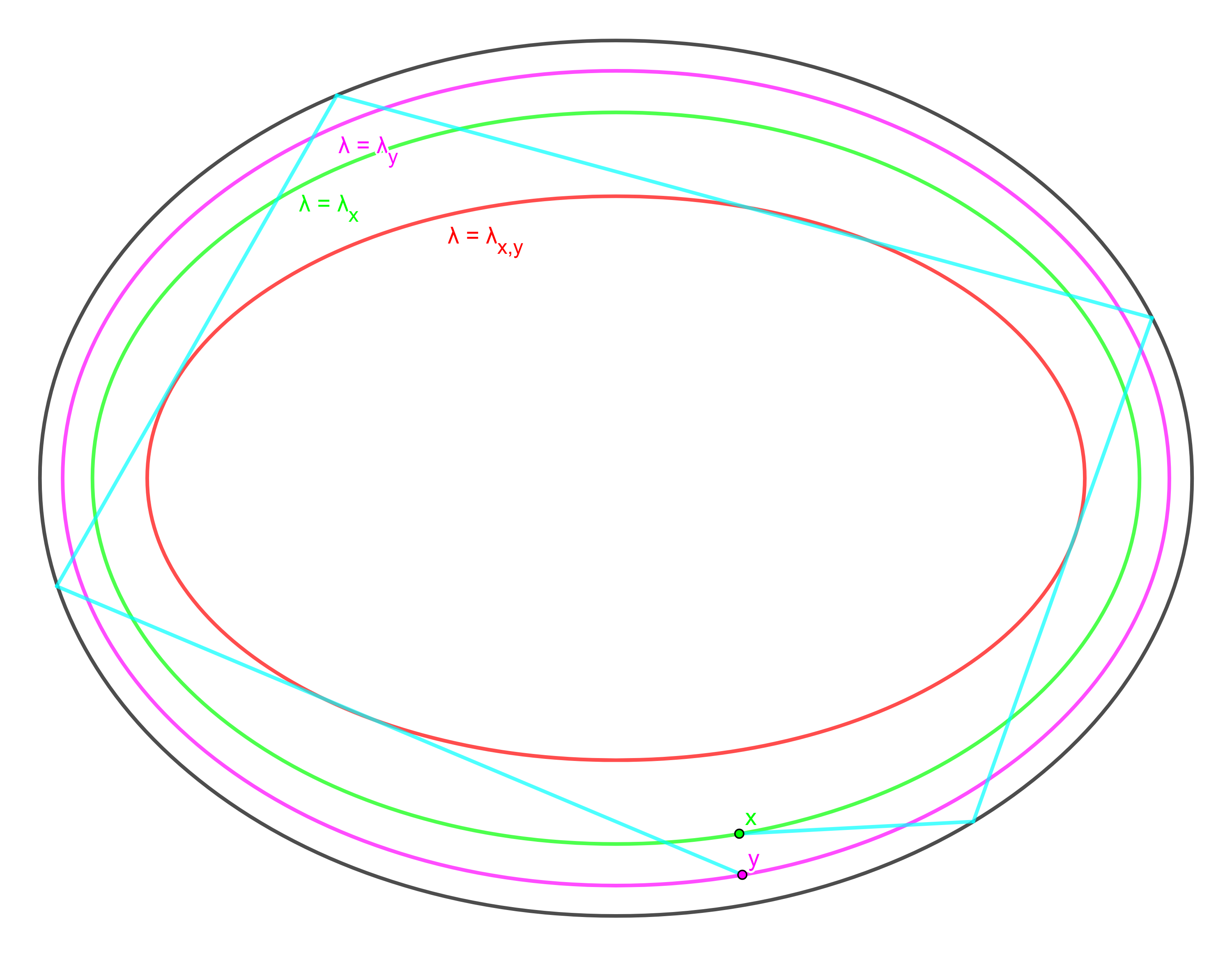}
\includegraphics[scale=0.25]{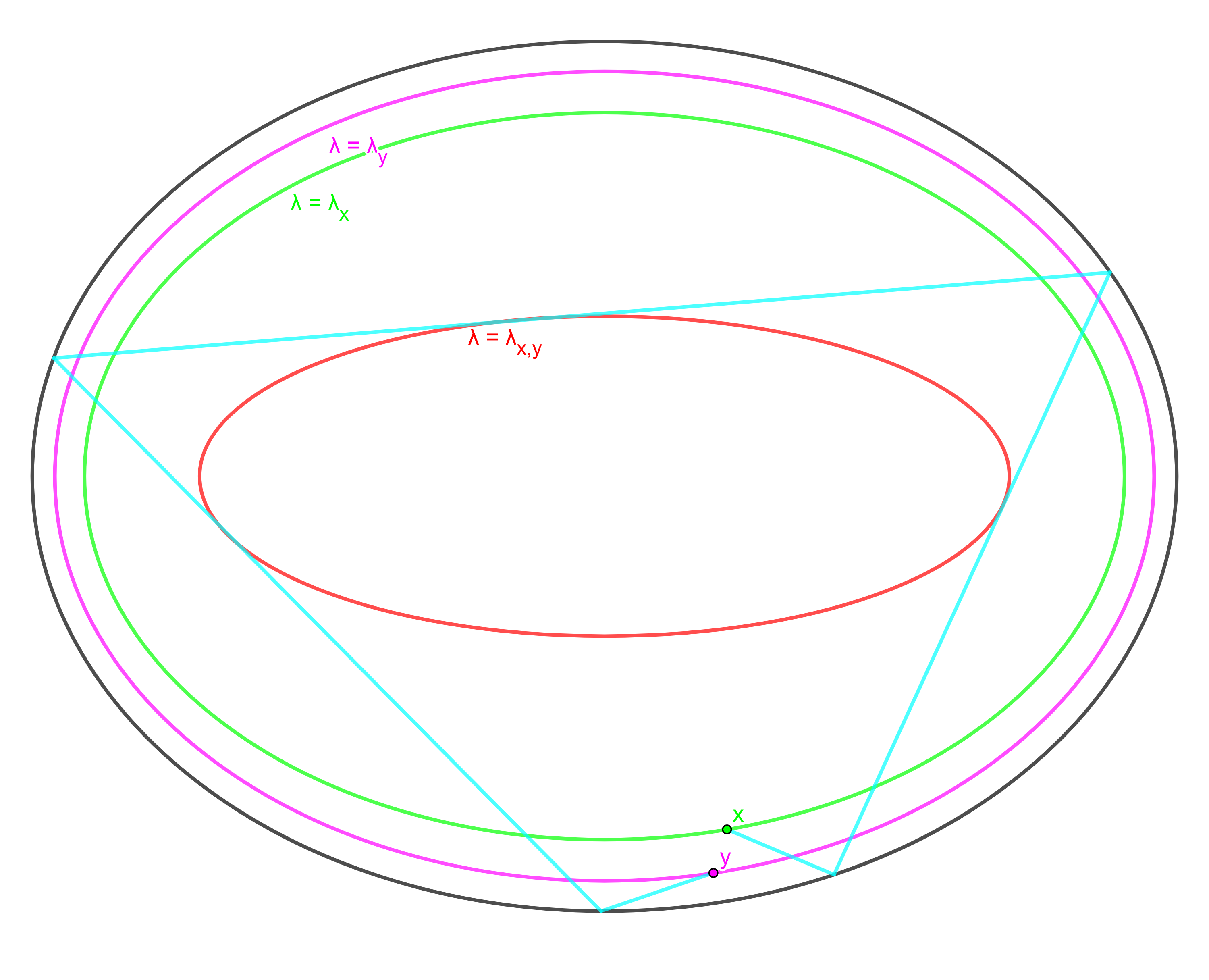}
\caption{Counterclockwise orbit configurations TT, TN, NT, and NN corresponding to $j = 4$. The green and pink curves are the distance curves on which $x$ and $y$ respectively lie. The red curve is a caustic to which the billiard orbit is tangent in the completely integrable setting.}
\label{configurations}
\end{figure}
%, which are obtained by reflecting the domain through the vertical axis, finding the counterclockwise orbits described in the proof below, and then reflecting back the domain.
\noindent See Figure \ref{configurations} for an example on the ellipse with $j=4$. These configurations will be important in determining which limiting orbits give rise to geodesic loops of precisely $j$ reflections as $(x,y) \to \Delta \d \Omega$, where $\Delta: \d \Omega \to \d \Omega \times \d \Omega$ is the diagonal embedding $x \mapsto (x,x)$.
%\begin{figure}
%\includegraphics[scale=0.25]{TTorbit1.png}
%\includegraphics[scale=0.25]{TNorbit1.png}
%\includegraphics[scale=0.25]{NTorbit1.png}
%\includegraphics[scale=0.25]{NNorbit2.png}
%\caption{Counterclockwise orbit configurations TT, TN, NT, and NN corresponding to $j = 4$. The green and pink curves are the confocal ellipses on which $x$ and $y$ lie, respectively. The red curve is the caustic of parameter $\lambda_{xy}$ to which the billiard orbit is tangent.}
%\label{configurations}
%\end{figure}
In Section \ref{A parametrix for the wave propagator}, we will actually be interested in orbits connecting $y$ to $x$ rather than $x$ to $y$ for reasons related to symplectic geometry and H\"ormander's conventions on the theory of Fourier integral operators. As Theorem \ref{8 orbit lemma} is clearly symmetric in $x$ and $y$, there is no problem in interchanging the initial and final points of the orbit.
\begin{def1}\label{psi jk}
For a billiard orbit $\gamma$ begining at $y$ and terminating at $x$, we define the length functional $\Psi(x,y)$ to the the Euclidean length of $\gamma$. As there are potentially many such $\gamma$ connecting $y$ and $x$, $\Psi(x,y)$ is multivalued. For $1 \leq k \leq 8$, denote by $\Psi_j^k(x,y)$ a branch of the length functional corresponding to one of the orbits of $j$ reflections in Theorem \ref{8 orbit lemma}. It depends only on $x,y, j$ and $k$. We use the convention that for a fixed number of reflections $j$, the indices $1 \leq k \leq 4$ correspond to the counterclockwise orbits $TT, TN, NT, NN$ in that order and the indices $5 \leq k \leq 8$ correspond to their clockwise counterparts in the same order. 
\end{def1}
\noindent The author learned of a similar function in \cite{MaMe82} (page 492), where its restriction to the boundary is defined. In such a case, i.e. if $x, y \in \d \Omega$, it is stated in \cite{MaMe82} and proved in \cite{Popov1994} that only a single counterclockwise orbit of $j$ reflections exists between the boundary points if they are sufficiently close and $j$ is sufficiently large.  The proof in Section \ref{Proof of 8 orbit lemma} below also shows that as $x$ and $y$ approach the diagonal of the boundary from the interior of $\Omega$, the corresponding orbits coalesce and converge to the orbits described in \cite{MaMe82}. This in fact proves the claims made in \cite{MaMe82} and provides an alternative to the methods employed in \cite{Popov1994}. The limiting orbits may have a different number of reflections though and this is addressed in the proof of Lemma \ref{Sj lemma}.
\begin{def1}\label{jloop}
For $x = y = q \in \d \Omega$, we denote by $\Psi_j(q,q)$ the length of the unique geodesic loop of $j$ reflections based at $q$. This is the $j$-loop function.
\end{def1}

\section{Proof of Theorem \ref{8 orbit lemma}}\label{Proof of 8 orbit lemma}

\subsection{Friedlander Model}
Let us first sketch the proof of \ref{8 orbit lemma} in the special case of the \textit{Friedlander model}, which can be considered as an approximation to the billiard map near tangential rays. Following \cite{CdVJG}, the Friedlander operator is defined to be
\begin{align}
	L = -\d_x^2 - (1+x) \d_y^2
\end{align}
on the manifold $M = [0, \infty) \times \R/2\pi \Z$, together with appropriate boundary conditions. The associated classical Hamiltonion is $H(x,y, \xi, \eta) = \xi^2 + (1+ x) \eta^2$, which we restrict to $S^*M = \{(x,y,\xi,\eta) \in T^*M: H(x,y,\xi,\eta) = 1\}$. $S^*M$ is a circle bundle over $[0, \infty) \times \R$ and the fiber over each base point $(x,y)$ is an ellipse, which can be parametrized in elliptical polar coordinates by
\begin{align}
	(\xi, \eta) = \left(\cos \theta, \frac{1}{\sqrt{1 + x}} \sin \theta \right), \qquad 0 \leq \theta < 2\pi.
\end{align}
The integral curves of the Hamiltonian vector field associated to $H$ with initial condition $(x_0, y_0, \xi_0, \eta_0) \in T_{(x_0, y_0)}^* M$ can easily be seen to be
\begin{align}\label{HJS}
	\begin{pmatrix}
		x(t)\\
		y(t)\\
		\xi(t)\\
		\eta(t)
	\end{pmatrix} =
	\begin{pmatrix}
		x_0 + \xi_0 t - \eta_0^2 t^2 / 4\\
		y_0 + \eta_0 t + \eta_0 \left(-\eta_0^2 t^3/ 12 + \xi_0 t^2/ 2 + x_0 t\right) \mod 2 \pi\\
		-\eta_0^2 t^2/4 + \xi_0 t + \xi_0\\
		\eta_0 
	\end{pmatrix}.
\end{align}
If $(x(t), y(t), \xi(t), \eta(t))$ is an orbit of \ref{HJS} on $[0,T]$ with $x(t) > 0$ for $0 \leq t < T$ and $x(T) = 0$, we reflect the covector by the law of equal angles and continue the flow past time $T$. In other words, the billiard map is discontinuous at $T_{\d M}^* M$ and we set
$$
\xi(T) = - \lim_{t \to T^-} \xi(t), \quad \eta(T) = \lim_{t \to T^-} \eta(t)
$$
to extend the flow for all $t > 0$.
\\
\\
Fix the bounce number $j \in \N$ large and let $a = (x_0, y_0)$ with $x_0  >0$ sufficiently small (in terms of $j$). If  $\theta_0 = \pi/2, 3\pi/ 2$, the corresponding orbit is tangent to the caustic $x = x_0$ and will make less than $1/4$ rotation around $\d M$. Denote by $\theta(t)$ the anglular component of $(\xi(t), \eta(t))$ in elliptical polar coordinates:
\begin{align}
	\theta(t) = 2\pi - \arctan \left( - \sqrt{1+ x(t)} \frac{\eta(t)}{\xi(t)}\right).
\end{align}
For an initial angle $\theta_0$, it can be seen from the formula for $x(t)$ in \ref{HJS} that the trajectory reaches the $y$ axis at time
\begin{align}\label{T}
	T(\theta_0) = \frac{2 \xi_0}{\eta_0^2} + \sqrt{4 \frac{\xi_0^2}{\eta_0^4} + 4 \frac{x_0}{\eta_0^2}} = 2 (1+ x_0) \frac{\cos \theta_0}{\sin^2 \theta_0} + 2 \sqrt{(1 + x_0)^2 \frac{\cos^2 \theta_0}{\sin^4 \theta_0} + (1 + x_0)\frac{x_0}{\sin^2 \theta_0}}.
\end{align}
One can then directly show that the angle
$$
\theta_1: = \theta(T(\theta_0)) = 3\pi - \lim_{t \to T(\theta_0)^-} \theta(t),
$$
from which the reflected trajectory emanates, has a global minimum at $\theta_0 = 3 \pi/2$ on the interval $[\pi, 2\pi]$: since the orbit is tangent to the caustic $x = x_0$, it is easy to see geometrically that the larger $\theta_1$ is, the greater the parameter $x$ of the associated caustic will be. Similarly, the first impact point $y_1 := y(T(\theta_0))$ on the universal cover of $\d M$ is monotone decreasing in $\theta_0$ and its speed is bounded for $\theta_0$ away from multiples of $2\pi$. Near $2\pi$, the orbits extend far into the right half plane and make too many rotations after $j$ bounces, or at $2\pi$ the corresponding orbit is a horizontal half ray. By translation invariance in $y$, it is clear that links connecting $\d M$ to itself are vertical translates of one another (Figure \ref{Fried}). Letting  $y_j$ denote the $y$ coordinate of the $j$th impact point on the boundary, it follows that $y_j$ is also continuous and decreasing as a function of $\theta_1$ with approximate speed $j$:
\begin{align}
	y_j(y_1, \theta_1) = y_1 + j \left(T \sin \theta_1  - T^3 \sin^3 \theta_1  /12 + T^2 \cos \theta_1 \sin \theta_1 / 2 \right)
\end{align}
where $T = T(\theta_1) = 2 \cos \theta_1 / \sin^2 \theta_1$ is defined analagously to \ref{T} so that the $x$ coordinate is again $0$ after time $T$ since the impact at $(0, y_1)$. Let $b = (\wt{x}, \wt{y})$ be sufficiently close to $a$ (in terms of $j$) and note that the orbit connecting $(0,y_j)$ to $(0,y_{j+1})$ then intersects the caustic $x = \wt{x}$ exactly twice. If $j$ is large, it is also easy to see that the $y$ coordinates of these intersection points are monotone decreasing in $\theta_1$, with approximate speed $j$.
\begin{figure}[h]
	\label{pic}
	\centering
	\includegraphics[scale = 0.3]{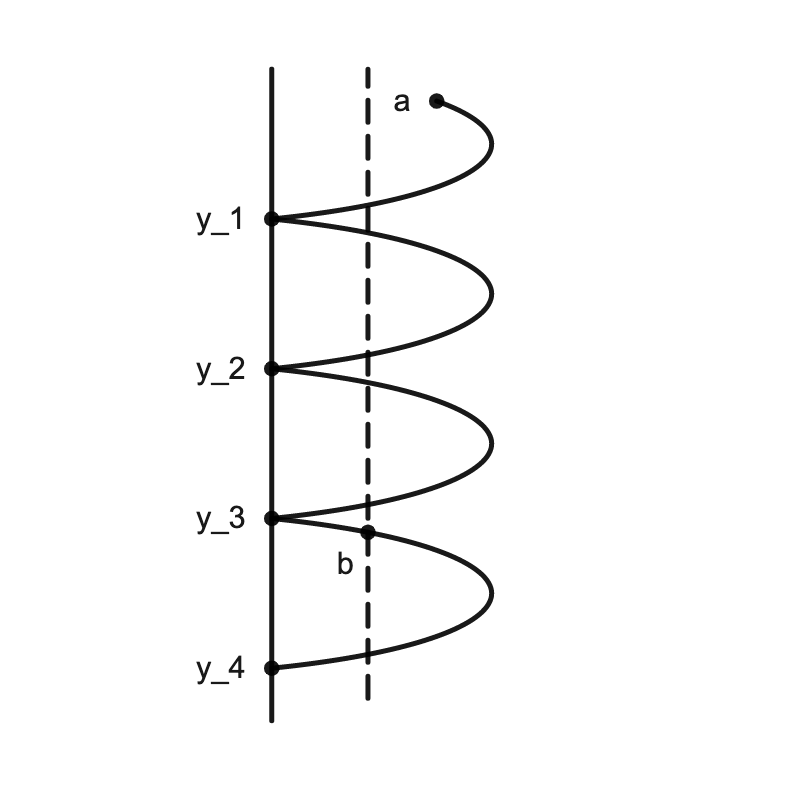}
	\caption{Four bounce forward orbit of the point $a = (x_0, y_0)$, interesecting the caustic $x = \wt{x}$ at the point $b$}
	\label{Fried}
\end{figure}
Hence, by the intermediate value theorem and decreasing $\theta_0$ on $\pi < \theta_0 < 3\pi/2$, we find exactly two angles for which the link $\overline{(0,y_j) (0,y_{j+1})}$ intersects the caustic $x = \wt{x}$ precisely at the point $b = (\wt{x}, \wt{y})$ in approximately one rotation. Similarly, there exist two such angles in $[3 \pi/2,  2\pi]$, and $4$ more angles in $[0, \pi]$ which generate orbits in the oppsite direction ($y$ increasing). This concludes our sketch of the proof of the $8$ Orbit Theorem in the special case of Friedlander's model. We will now more rigorously prove the claims made above directly for the billiard map on a convex domain by working in Lazutkin coordinates. The required proximity of orbit endpoints $a$ and $b$ to the diagonal and boundary will be made more precise in terms of $j$ below as well. We will use $x$ and $y$ in place of $a$ and $b$ when not working directly with Cartesian coordinates.

\subsection{General Convex Domains}
\noindent We will construct conic subbundles $C_\pm^*(\Omega; j) \subseteq  T^*\Omega$ over a tubular neighborhood $\widetilde{U}_j$ of the boundary with the following property: for each $x_0 \in \widetilde{U}_j$, all orbits emanating from $x_0$ which make $j$ reflections and approximately one rotation in the sense of Section \ref{Billiards} have initial covector in $C_{x_0}^*(\Omega;j)$, the fiber of $\cup_\pm C_\pm^*(\Omega;j)$ over $x_0$. At each $x_0 \in \widetilde{U}_j$, there also exist distinguished covectors $\xi_0^\pm \in T_{x_0}^*(\Omega) \backslash 0$ which are tangent to distance curves folliating a neighborhood of the boundary. If $\widetilde{U}_j$ is chosen to be sufficiently narrow, the $j$ reflection orbits emanating from $\xi_0^\pm$ will make less than a quarter rotation. By rotating $\xi$ away from $\xi_0^\pm$ in either direction within the fiber of $C_\pm^*(\Omega;j)$ at $x_0$, we will show that the angle of reflection at the first impact point on the boundary increases monotonically. From the boundary, we can then take advantage of the twist property of the billiard map, and show that monotonicity of the incident angles causes orbits of a large number of reflections to wind around $\Omega$. Two of the orbits from $x$ to $y$ will be obtained by perturbing $\xi_0^+$ in the counterclockwise direction and the other two will be obtained by perturbing $\xi_0^+$ in the clockwise direction. The four clockwise orbits will then be constructed in a similar manner, rotating $\xi_0^-$ in both the clockwise and counterclockwise fiber directions. The arguments in this section will also provide the additional topological structure of the resulting orbits referenced in Definition \ref{TTNN}, which can be seen in Figure \ref{configurations}.

% $\xi_0^\pm \in S_{x_0}^*\Omega \cap C_\pm^*$ which locally minimize the angle of incidence made by the counterclockwise $(+)$ or clockwise $(-)$ billiard rays at the first impact point on the boundary. Furthermore, all orbits from $x$ to $y$ making approximately one rotation and exactly $j$ reflections at the boundary have initial covector in the $C_\pm^*$ fiber at $x$.

\subsection{Lazutkin coordinates} Recall that the billiard map is defined on the coball bundle $B^*\d \Omega$, which can be identified with the collection of inward facing covectors in the circle bundle $S_{\d \Omega}^{*} \R^2$. Letting $s$ denote the arclength parameter on $\d \Omega$ and $\vartheta$ the angle an inward facing covector makes with the positively oriented boundary, we define the modified billiard map $\widetilde{\beta}$ in terms of $(s, \vartheta) \in \R/\ell \Z \times (0,\pi)$. In this coordinate system, the modified billiard map is given by $\widetilde{\beta}(s_1, \vartheta_1) = (s_2, \vartheta_2)$, where
\begin{align}
\begin{cases}
s_2 = s_1 + a_1(s_1) \vartheta_1 + a_2(s_1) \vartheta_1^2 + a_3(s_1) \vartheta_1^3 + F(s_1, \vartheta_1)\vartheta_1^4\\
\vartheta_2 = \vartheta_1 + b_2(s_1) \vartheta_1^2 + b_3(s_1) \vartheta_1^3 +  G(s_1, \vartheta_1)\vartheta_1^4,
\end{cases}
\end{align}
and $a_i, b_i, F$ and $G$ are smooth functions. This is a Taylor expansion in the angular variable near $\vartheta_1 = 0$. There are explicit formulas for the coefficients $a_i, b_i$ (\cite{Lazutkin}). In particular, $a_1(s_1) = 2 \rho(s_1)$ and $b_2 = -2/3 \rho'(s_1)$, where $\rho(s_1)$ is the radius of curvature at $s_1$. In \cite{Lazutkin}, the change of coordinates
\begin{align}\label{Lazutkin coord}
x = \frac{\int_0^{s} \rho^{2/3}(t)dt}{\int_0^{\ell} \rho^{2/3}(t)dt}, \qquad 
\alpha = \frac{4 \rho(s) \sin \vartheta/2}{\int_0^{\ell} \rho^{2/3}(t)dt}
\end{align}
was introduced near the circle corresponding $\vartheta = 0$. We call the coordinate system $(x, \alpha)$ in \eqref{Lazutkin coord} Lazutkin coordinates. The advantage of this change of variables is that in these coordinates, the billiard map becomes a small perturbation of the translation map
\begin{align}
\begin{pmatrix}
x_1\\
\alpha_1
\end{pmatrix} \mapsto
\begin{pmatrix}
1 & 1\\
0 & 1
\end{pmatrix}
\begin{pmatrix}
x_1\\
\alpha_1
\end{pmatrix}
= \begin{pmatrix}
x_1 + \alpha_1\\
\alpha_1
\end{pmatrix},
\end{align}
which is completely integrable. The billiard map is given in Lazutkin coordinates by $(x_1, \alpha_1) \mapsto (x_2, \alpha_2)$, where
\begin{align}
\begin{cases}
x_2 = x_1 + \alpha_1 + \alpha_1^3 f(x_1, \alpha_1)\\
\alpha_2 = \alpha_1 + \alpha_1^4 g(x_1,\alpha_1),
\end{cases}
\end{align}
for some smooth functions $f,g$. This change of variables was first derived in \cite{Lazutkin}, where it was shown as a consequence of the KAM theorem (see \cite{KAM1}, \cite{KAM4}, \cite{KAM2}, \cite{KAM3}) that there exist an abundance of invariant tori for the billiard map near $\alpha = 0$. We will use these coordinates to prove Theorem \ref{8 orbit lemma}. Without loss of generality, we will often interchange the use of arclength and Lazutkin coordinates in the domain of the billiard map $\beta$.

\subsection{Angles of reflection}
Before estimating the billiard map and its iterates from the boundary, we need several preliminary estimates on the angles of reflection. In this section, we denote $\widehat{\beta}^k(x_1, \alpha_1) = (x_k, \alpha_k)$.
\begin{lemm}\label{angle bound 1}
For each $c_1, c_2 > 0$, there exists $j_1 = j_1(\Omega, c_1, c_2) \in \N$ sufficiently large such that for $j \geq j_1$ and $c_1/j \leq \alpha_1 \leq c_2/j$, it follows that
$$
\frac{c_1}{(1 + c_1/j)^k j} \leq \alpha_k = \pi_2 \beta^k(x_1,\alpha_1) \leq \frac{c_2 (1+ \frac{c_2}{j})^k}{j} \qquad (1 \leq k \leq j),
$$
where $\pi_2$ is the projection onto the angular variable. In particular, 
$$
c_1 e^{-c_1}/j \leq \pi_2 \beta^k(x_1,\alpha_1) \leq c_2 e^{c_2}/j.
$$
\end{lemm}
\begin{proof}
Fix $c_1, c_2$ and choose $j_1$ large enough that for all $j \geq j_1$,
\begin{align}\label{j1}
	\begin{split}
	\frac{c_2^3e^{3c_2} \sup |g| }{j^3 } &\leq \frac{c_2}{j},\\
	\frac{c_1^3 \sup |g| }{j^3} &\leq \frac{c_1}{j}.
\end{split}
\end{align}
We will use Lazutkin coordinates and induction in $k$ to show the estimates in the lemma hold. Let $j \geq j_1$ and suppose the claim is true for $1, \cdots, k - 1$. Then 
\begin{align}
\begin{split}
\pi_2 \beta^k (x_1,\alpha_1) &= \pi_2 \beta \circ \beta^{k-1} (x_1, \alpha_1) = \alpha_{k-1} + \alpha_{k-1}^4 g(x_{k-1},\alpha_{k-1})\\
&\leq c_2(1 + c_2/j)^{k-1}/j + (c_2(1+ c_2/j)^{k-1}/j)^4 \sup |g|\\
&= \frac{c_2(1 + c_2/j)^{k-1}}{j} \left(1 + \frac{c_2^3(1 + c_2/j)^{3k-3} \sup |g|}{j^3} \right)\\
& \leq \frac{c_2(1 + c_2/j)^{k-1}}{j} \left(1 + \frac{c_2^3e^{3c_2} \sup |g|}{j^3} \right).
\end{split}
\end{align}
Given the choice of $j_1$ in the first line of \ref{j1}, the upper bound follows also for the $k$th iterate. Similarly,
\begin{align}
\begin{split}
\pi_2 \beta^k (x_1,\alpha_1) &= \pi_2 \beta \circ \beta^{k-1} (x_1, \alpha_1) = \alpha_{k-1} + \alpha_{k-1}^4 g(x_{k-1},\alpha_{k-1})\\
&\geq \frac{c_1}{(1 + c_1/j)^{k-1} j} - \left(\frac{c_1}{(1+ c_1/j)^{k-1}j}\right)^4 \sup |g|\\
&= \frac{c_1}{(1 + c_1/j)^{k-1}j} \left(1 - \frac{c_1^3}{(1 + c_1/j)^{3k-3} j^3} \sup |g| \right)\\
& \geq \frac{c_1}{(1 + c_1/j)^{k-1}j} \left(1 - \frac{c_1^3 \sup |g|}{j^3}  \right)
\end{split}
\end{align} which is greater than
\begin{align*}
\frac{c_1}{(1 + c_1/j)^k j},
\end{align*}
under the second condition in \ref{j1}, demonstrating the lower bound for $k$.
\end{proof}

\begin{lemm}\label{L3}
There exist $C_1 , C_2 > 0$ and $j_2 = j_2(\Omega) \in \N$ with the following property: for all  $j \geq j_2$ and $x, y$ which are $O(1/j)$ close to the diagonal of the boundary, any orbit $\gamma$ of $j$ reflections which emanates from $x$ and terminates at $y$ in approximately one rotation (in the sense of Section \ref{Billiards}) has angles of reflection $\alpha_k$ in the range
$$
\frac{C_1}{j} \leq \alpha_k \leq \frac{C_2}{j} \qquad (1 \leq k \leq j).
$$
\end{lemm}
\begin{proof}
Let $\widehat{x}, \widehat{y}$ denote the $1$st and $(j+1)$st points of reflection on the boundary. Recall that by approximately one rotation, we mean that
\begin{align*}
|\widehat{\beta}^j(\widehat{x}, \alpha_1) - \widehat{y} - \ell| \leq \ell /100,
\end{align*}
where $\widehat{\beta}$ is the lift to the universal cover and $\ell = |\d \Omega|$. Each arc, separated by moments of reflection $x_p, x_{p+1}$ on the boundary, has length $x_{p+1} - x_p$. As $\gamma$ makes approximately one rotation, we see that there must exist one arc of length $x_{p+1} - x_p \geq (99/100)\ell/j$. Similarly, there must exist an arc of length $x_{m+1} - x_m \leq (101/100) \ell/j$. By Proposition 14.1 of \cite{Lazutkin5}, for each $1 \leq k \leq j$,
\begin{align}\label{curvature bound}
2 \alpha_k \rho_\text{min} \leq x_{k+1} - x_k \leq 2 \alpha_k \rho_\text{max},
\end{align}
where $\rho_\text{min}$ and $\rho_\text{max}$ are the minimum and maximum radii of curvature respectively for $\d \Omega$. In particular, 
\begin{align}
\frac{99 \ell }{200 \rho_\text{max} j} \leq \alpha_p, \qquad \alpha_m \leq \frac{101 \ell}{200 \rho_\text{min} j}.
\end{align}
We now apply Lemma \ref{angle bound 1} to $\beta^\pm$ at the points $x_p$ and $x_m$ to conclude that for all $1 \leq k \leq j$,
\begin{align}
\frac{\exp\left( - \frac{99 \ell }{200 \rho_\text{max}}\right)}{j} \leq \alpha_k \leq \frac{\exp\left( \frac{101 \ell}{200  \rho_\text{min}} \right)}{j}.
\end{align}
Here we have used Lemma \ref{angle bound 1} applied to both the billiard map and its inverse, with the constants $c_1 = 99 \ell / 200 \rho_\text{max}$ and $c_2 = 101 \ell /200 \rho_\text{min}$ which depend only on $\Omega$. Hence, $j_2$ may be chosen uniformly for all orbits making approximately one rotation regardless of their initial positions.
\end{proof}

\noindent We now investigate the angle of reflection at the first point of impact on the boundary. If $d(z) = \dist(z, \d \Omega)$, then the level curves of $d$ folliate a neighborhood of the boundary. For $x_0$ in the tubular neighborhood $\widetilde{U}_j$ (whose diameter remains to be specified), we define $\xi_0^\pm \in \pm T_{x_0}^* d^{-1}(d(x_0)) \cap S_{x_0}^*\Omega$ so that $\xi_0^+$ points in the counterclockwise direction and $\xi_0^-$ points in the clockwise direction.
%\begin{figure}
%\begin{center}
%\begin{tikzpicture} [domain=0:2, scale = 1.1]
%\draw[->, thick] (-5,0)--(5,0); %node[right]{$x$};
%\draw[->, thick] (0,-2)--(0,4); %node[above]{$y$};
%\draw[->, thick] (0,-2)--(0,4);
%\draw[black, line width = 0.25mm]   plot[smooth,domain=-4:4] (\x, 0.2*\x*\x);
%%\draw (0,0) parabola (2,4);
%\coordinate (A) at (0, 0.5);
%\coordinate (B) at (3,9/5);
%\coordinate (C) at (1.5,9/5 - 9/5);
%\coordinate (D) at (4.5, 9/5 + 9/5);
%\coordinate (E) at (-5,0.5);
%\coordinate (F) at (5,0.5);
%%\coordinate (a-) at (-3.53553390593, 0);
%%\coordinate (a+) at (3.53553390593, 0);
%\coordinate (G) at (5,0.5);
%\coordinate (H) at (3 - 13/12, 0.5);
%\draw[black, thick] (A)--(B);
%\draw[black, thick] (C)--(D);
%\draw[black, thick, dashed] (E)--(F);
%%\draw[black, thick, dashed] (a-)--(E);
%%\draw[black, thick, dashed] (a+)--(F);
%%\draw[fill=black] (a-) circle (2pt) node[below left]{$a_-$};
%%\draw[fill=black] (a+) circle (2pt) node[below right]{$a_+$};
%\draw[fill=black] (A) circle (2pt) node[above left]{$b$};
%\draw[fill=black] (B) circle (2pt) node[below right]{$(x,f(x))$};
%%\draw[fill=black] (H) circle (2pt) node[above left]{$H$};
%\path 
%  pic["$\alpha$",draw=black,<->,angle eccentricity=1.5,angle radius=0.75cm] {angle=A--B--C};
%\path 
%    pic["$\xi$",draw=black,<->,angle eccentricity=1.5,angle radius=0.75cm] {angle=G--A--B};
%\path 
%    pic["$\eta$",draw=black,<->,angle eccentricity=1.2,angle radius=0.75cm] {angle=G--H--B};
%\end{tikzpicture}
%\end{center}
%\caption{Angular derivative of the billiard flow.}
%\label{picture of billiard flow}
%\end{figure}
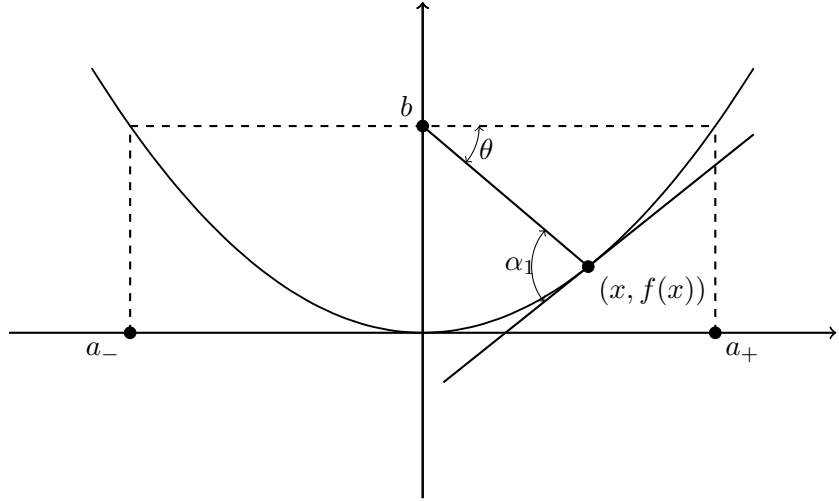
\begin{figure}
\begin{center}
\begin{tikzpicture} [domain=0:2, scale = 1.1]
\draw[->, thick] (-5,0)--(5,0); %node[right]{$x$};
\draw[->, thick] (0,-2)--(0,4); %node[above]{$y$};
\draw[->, thick] (0,-2)--(0,4);
\draw[black, line width = 0.25mm]   plot[smooth,domain=-4:4] (\x, 0.2*\x*\x);
%\draw (0,0) parabola (2,4);
\coordinate (A) at (0, 2.5);
\coordinate (B) at (2,0.8);
\coordinate (C) at (0.25,-0.6);
\coordinate (D) at (4, 2.4);
\coordinate (E) at (-3.53553390593,2.5);
\coordinate (F) at (3.53553390593,2.5);
\coordinate (a-) at (-3.53553390593, 0);
\coordinate (a+) at (3.53553390593, 0);
\draw[black, thick] (A)--(B);
\draw[black, thick] (C)--(D);
\draw[black, thick, dashed] (E)--(F);
\draw[black, thick, dashed] (a-)--(E);
\draw[black, thick, dashed] (a+)--(F);
\draw[fill=black] (a-) circle (2pt) node[below left]{$a_-$};
\draw[fill=black] (a+) circle (2pt) node[below right]{$a_+$};
\draw[fill=black] (A) circle (2pt) node[above left]{$b$};
\draw[fill=black] (B) circle (2pt) node[below right]{$(x,f(x))$};
\path 
  pic["$\alpha_1$",draw=black,<->,angle eccentricity=1.2,angle radius=0.75cm] {angle=A--B--C};
\path 
  pic["$\theta$",draw=black,<->,angle eccentricity=1.2,angle radius=0.75cm] {angle=B--A--F};
\end{tikzpicture}
\end{center}
\caption{Angular derivative of the billiard flow.}
\label{picture of billiard flow}
\end{figure}
Let $\xi \in S_{x_0}^* \Omega$ be identified with the corresponding point $\theta \in S^1 = \R/\Z$, which denotes the clockwise angular parametrization of the fiber $S_{x_0}^*\Omega$, normalized so that $\xi_0^+$ and $\xi_0^-$ correspond to $0$ and $\pi$ respectively. Denote by $\alpha_1 \in (0,\pi)$ the angle made between the billiard orbit emanating from $(x_0, \xi)$ and the positively oriented cotangent line at the first point of reflection at the boundary. It is clear that $|\d \xi / \d \theta| = 1$.
\begin{lemm}\label{L1}
There exist $c = c(\Omega) > 0$ and $j_3 = j_3(\Omega) \in \N$ such that if $j \geq j_3$ and $d(x_0) = O(1/j^2)$, then
$$
\frac{\d \alpha_1}{\d {\theta}} \geq c
$$
for all initial conditions $(x_0, \xi(\theta))$ which generate a counterclockwise orbit of $j \geq j_0$ reflections, making approximately one rotation. %and $\frac{\d \alpha }{\d {\xi}} = 0$ when ${\xi} = \xi_0^+$.
\end{lemm}

\begin{proof}
Let $x_0 \in \text{int} (\Omega)$ be near the boundary and consider the point ${x} \in \d \Omega$ which minimizes the Euclidean distance $|x_0- x|$ among all boundary points. We then apply an affine change of coordinates $y \mapsto A(y - x)$ with $A \in SO(2)$ so that $x$ is mapped to the origin and the positively oriented unit tangent vector at ${x}$ is mapped to $(1,0)$. The vector ${x} - x_0$ is perpendicular to both $T_x^*\d \Omega$ and $T_{x_0}^* d^{-1}(|x - x_0|)$ (see Figure \ref{picture of billiard flow}). Therefore, the point $A(x_0 - x)$ lies on the positive vertical axis and we denote it by $(0,b)$. In these coordinates, the boundary $\d \Omega$ is locally parametrized by the graph of a smooth convex function $f$. Since $f'(0) = 0$, we see that $f(t) = \kappa t^2 + R(t)$, with $\kappa$ denoting the curvature at the point corresponding to $(0,0)$ and $R(t) = O(|t|^3)$. Dilating by a factor of $\kappa^{-1}$, we see that in rescaled coordinates, the function becomes $f(t) =  t^2 + R(t)$, with $R(t)$ again of order $O(|t|^3)$.
\\
\\
In graph coordinates, $\alpha_1$ becomes the angle between $(t, f(t)) - (0,b)$ and the negatively oriented cotangent line $T^*_{(t,f(t))} \d \Omega$ as illustrated in Figure \ref{picture of billiard flow}. This shows that
\begin{align}\label{cos angle}
\cos \alpha_1 = \frac{(t,f(t) - b) \cdot (1, f'(t))}{\sqrt{t^2 + (f(t) - b)^2} \sqrt{1+ (f'(t))^2}} \geq 0,
\end{align}
since $0 \leq \alpha_1 \leq \pi/2$. Recall that $\xi_0^\pm$ are the unit covectors in the positive (+), resp. negative (-), cotangent line $T^*_{x_0} d^{-1}(|x- x_0|)$ corresponding to $\theta = 0$ and $\theta = \pi$ respectively. For now, we only consider counterclockwise orbits obtained by perturbing the orbit emanating from $(x_0,\xi_0^+)$, as they correspond to $-\pi/2 \leq \theta \leq \pi/2$ in the statement of the lemma. Denote the points $a_- = \min\{f^{-1}(b)\}$ and $a_+ = \max \{f^{-1}(b)\}$. We first show that $d \alpha_1/dt < 0$ on $[0,a_+]$ so that the angle of incidence at the first impact point on the boundary is increasing as the initial covector of the trajectory winds clockwise in $S_{x_0}^*\Omega$ away from $\xi_0^+$, i.e. $d\alpha/dt < 0$ on $[0,a_+]$.
\\
\\
Noting that $t$ and $\theta$ are negatively correlated, it suffices to show that the logarithmic $t$ derivative of $\cos \alpha_1(t)$ is positive:
\begin{align}\label{log cos derivative}
\frac{d}{dt} \log \cos \alpha_1(t) = \frac{(1 + f'' (f - b) + (f')^2)}{t + f' (f-b)} - \frac{t + f'(f-b)}{t^2 + (f-b)^2} - \frac{f'f''}{1 + (f')^2}.
\end{align}
Multiplying \eqref{log cos derivative} through by a common denominator, which is positive for $t, b$ small, we obtain
\begin{align}\label{logarithmic derivative times common den}
\begin{split}
g(t,b) :=& (t + f'(f - b))(t^2 + (f - b)^2)(1 + (f')^2) \frac{d}{dt} \log \cos \alpha(x)\\
=& (1 + f'' (f - b) + (f')^2)(t^2 + (f - b)^2)(1 + (f')^2)\\
&- (t + f'(f - b))^2(1 + (f')^2)  - f' f''(t +f'(f-b))(t^2 + (f-b)^2).
\end{split}
\end{align}
To show that \eqref{logarithmic derivative times common den} is positive, we plug in our second order Taylor approximation for $f$ and expand $g$ to fourth order in a parabolic neighborhood of $t= 0, b = 0$:
\begin{gather}\label{fourth order expansion}
g(t,b) = 3t^4 + 10t^2 b^2 -2b^3 + b^2 + O((t^2 + b^2)^{5/2}).
\end{gather}
As $f(t) = t^2 + O(t^3)$, we see that $a_\pm = \pm b^{1/2} + O(b^{3/2})$, so we choose $b \leq 1/j^2$ sufficiently small in order to make \eqref{fourth order expansion} positive in a parabolic neighborhood of origin. Examining the initial set up, choosing $b$ small amounts to choosing $x_0$ close to the boundary, all of which can be done uniformly in $x_0$ and the curvature $\kappa$. Choosing a uniform $b$ gives us a tubular neighborhood of the boundary $\widetilde{U}_j = \{x: 0< d(x,\d \Omega) < j^{-2} \}$, as mentioned in the begining of the section. We will later shrink $\widetilde{U}_j$ by a factor of $j^{-2}$, following Lemma \ref{clockwise int points}.
\\
\\
While the lower bound in \eqref{fourth order expansion} is of order $b^2 = O(j^{-4})$, we in fact need uniform bounds for $\d \alpha_1/ \d \theta$ which we now provide. The prefactor in \eqref{logarithmic derivative times common den} is nonvanishing for $t \neq 0, t \neq a_+$ and in particular, for all $t$ corresponding to an orbit which makes approximately one rotation in $j$ reflections. We first need to compare $\d/\d t$ and $\d / \d \theta$ in terms of $b$, which depends on $j$. In graph coordinates, $\theta$ becomes the angle of the first link from the horozontial axis (see Figure \ref{picture of billiard flow}). Hence, $\theta = \arctan((b- f(t))/t)$ and
\begin{align*}
\frac{\d \theta}{\d t} = \frac{1}{1 + (b - f(t))^2/t^2 } \left( - \frac{ f'(t)}{t} - \frac{b - f(t)}{t^2}\right),
\end{align*}
or equivalently,
\begin{align}\label{d x d xi}
\frac{\d t}{\d \theta} = - \left(\frac{t^2 + (b - f(t))^2}{t f'(t) + b - f(t)}\right),
\end{align}
where $t$ is defined implicitly as a function of $\theta$. Plugging \eqref{d x d xi} into  \eqref{log cos derivative} and using that \eqref{fourth order expansion} is bounded below by $b^2/2$ for $b$ sufficiently small, we see that
\begin{align}\label{dalphadxi}
\begin{split}
\frac{\d \alpha}{\d \theta} &\geq \frac{ b^2 (t^2 + (b - f(t))^2) }{2 (tf'(t) + b - f(t)) (t + f'(t)(f(t) - b)) (t^2 + (f(t) - b)^2 )(1 + (f'(t))^2)  \tan \alpha_1  }\\
&=  \frac{ b^2 }{2 (tf'(t) + b - f(t) ) (t + f'(t)(f(t) - b)) (1 + (f'(t))^2)  \tan \alpha_1  }.
\end{split}
\end{align}
The denominator in \eqref{dalphadxi} is nonvanishing and can be estimated above on $(0, a_+]$ by
\begin{align}
c' t b \tan \alpha_1 + O(b^{5/2}).
\end{align}
for some $c' > 0$, independent of $j$. Observe that if $t \leq b$ for example, then the denominator is bounded above by $c' b^2$ and $\d \alpha_1/ \d \theta \geq c''$, for some $c'' > 0$ which is also independent of $j$.
\\
\\
We now find a similar upper bound on $t(\theta)$ in terms of $b$, corresponding to the set of covectors $\xi(\theta)$ which produce orbits of $j$ reflections and approximately one rotation. Let $b = r b_0$, where $b_0 = j^{-2}$ for $r \in (0,1]$. Recall from Lemma \ref{L3} that there exist constants $C_1, C_2 > 0$ such that for all orbits making approximately one rotation and $j$ reflections, we have $C_1/j \leq \alpha_k \leq C_2/j$ for each $1 \leq k \leq j$. Hence,
\begin{align}\label{b_0 bound}
C_1 b_0^{1/2} \leq \alpha_1 \leq C_2 b_0^{1/2}.
\end{align}
Observe that equation \eqref{cos angle} gives
\begin{align}\label{log cos alpha}
\begin{split}
\log \cos \alpha_1(t) &= \log(t + f'(t)(f(t) - b)) - 1/2 \log (t^2 + (f(t)-b)^2)\\
&- 1/2 \log(1 +(f'(t))^2).
\end{split}
\end{align}
Denote the terms in \eqref{log cos alpha} above by $A, B$ and $C$. Then,
\begin{align*}
\begin{split}
A &= \log( t + (2t + O(t^2) ) (t^2 + O(t^3) - b ))\\
&= \log t + \log(1 + 2t^2 - 2b + O(t^3) + O(b))\\
&= \log t + 2t^2 + O(t^3) + O(b).
\end{split}
\end{align*}
Similarly,
\begin{align*}
\begin{split}
B &= -1/2 \log(t^2 + t^4 - 2t^2b + b^2 + O(t^5) + O(t^3 b) )\\
&= -\log t - 1/2 (t^2 - 2b + b^2/t^2 + O(t^3) + O(tb))
\end{split}
\end{align*}
and
\begin{align*}
C = - \log(1 + (2t + O(t^2)) ) = -4t^2 + O(t^3).
\end{align*}
As $t = O(b^{1/2})$, we have
\begin{align*}
A + B + C = \frac{-b^2}{2t^2} + O(b).
\end{align*}
On the other hand, 
\begin{align*}
\begin{split}
A + B + C &= \log \cos \alpha_1 = \log (1 - a_1^2 / 2 + O(\alpha_1^4)).\\
&= - \alpha_1^2/2 +  O(b_0^{4/N}),
\end{split}
\end{align*}
which implies that 
\begin{align*}
b^2/t^2 = r^2 b_0^2 / t^2 \sim \alpha_1^2 / 2.
\end{align*}
In particular, \eqref{b_0 bound} implies 
\begin{align}\label{x sim}
t \sim \frac{\sqrt{2}r b_0}{\alpha_1} \in \left[\frac{\sqrt{2} rb_0^{1/2}}{C_2} , \frac{\sqrt{2} rb_0^{1/2}}{C_1} \right].  %\leq \frac{\sqrt{2}}{c} r^{1- 1/N} b_0^{1- 1/N} = \frac{\sqrt{2}}{c} b^{1 - 1/N}
\end{align}
Note also that $\tan \alpha_1 \leq 2 \alpha_1 \leq 2 C_2 b_0^{1/2}$. Combining this with \eqref{x sim}, we see that the denominator in equation \eqref{dalphadxi} is bounded above by
\begin{align}
c' t b \tan \alpha_1 + O(b^{5/2}) \leq c''' b  (r b_0^{1 - 1/2}) b_0^{1/2} = c''' b^2,
\end{align}
for some $c''' > 0$ independent of $j$. This establishes the claim that $\d \alpha/ \d \theta$ is uniformly bounded below on the set of initial covectors corresponding to orbits of $j$ reflections and approximately one rotation.  %Examining the initial set up, choosing $b$ small amounts to choosing $x_0$ close to the boundary, all of which can be done uniformly in $x_0$ and the curvature $\kappa$.}
%Choosing a uniform $b$ gives us a tubular neighborhood of the boundary $U = \{x: 0< d(x,\d \Omega) < b\}$, as mentioned in the begining of the section.}
\end{proof}

\subsection{Monotonicity and the twist property}
The significance of increasing the angle of reflection at the first impact point on the boundary is that after $j-1$ reflections at the boundary, the $jth$ impact point winds around $\d \Omega$ with approximate speed $j$. This is essentially the twist property of the billiard map (see \cite{Tabachnikov} for a general definition of twist maps). The following lemma makes this notion quantitative.
\begin{lemm}\label{boundary monotonicity}
There exists $j_4 = j_4(\Omega) \in \N$ such that if a trajectory makes $j \geq j_4$ reflections at the boundary and approximately one rotation, then the Jacobian of the iterated billiard map satisfies
$$
D \widehat{\beta}^j (x_1 ,\alpha_1) = \begin{pmatrix}
1 & j\\
0 & 1
\end{pmatrix}
+ O(1/j)
$$
in Lazutkin coordinates lifted to the universal cover. In particular, there exist constants $C_3, C_4, C_5, C_6 > 0$ depending only on $\Omega$ such that for any $(x_1, \alpha_1)$, we have
\begin{align*}
C_3 j \leq \frac{\d}{\d \alpha_1} \pi_1 \widehat{\beta}^j(x_1, \alpha_1) \leq C_4 j\\
C_5 \leq \frac{\d}{\d \alpha_1 } \pi_2 \widehat{\beta}^j(x_1, \alpha_1)  \leq C_6,
\end{align*}
where $\pi_1$ and $\pi_2$ denote projections onto the first and second components respectively.
\end{lemm}
\begin{proof}
In Lazutkin coordinates, $\beta(x_k,\alpha_k) = (x_k + \alpha_k + \alpha_k^3 f(x_k,\alpha_k), \alpha_k + \alpha_k^4 g(x_k,\alpha_k))$. Hence,
\begin{align}\label{Dbetaj}
D\widehat{\beta}^j(x_1,\alpha_1) = \prod_{k = 1}^j \begin{pmatrix}
1 + \alpha_k^3 \frac{\d f}{\d x}(x_k, \alpha_k) & 1 + 3\alpha_k^2 f(x_k,\alpha_k) + \alpha_k^3 \frac{\d f}{\d \alpha}(x_k,\alpha_k)\\
\alpha_k^4 \frac{\d g}{\d x}(x_k,\alpha_k) & 1 + 4 \alpha_k^3 g(x_k,\alpha_k) + \alpha_k^4 \frac{\d g}{\d \alpha}(x_k, \alpha_k)
\end{pmatrix},
\end{align}
where $(x_k, \alpha_k) = \widehat{\beta}^k(x_1,\alpha_1)$ and $f,g$ are extended to be $\ell$ periodic in $x$ on $\R \times [0,\pi]$. Each of the terms in the product can be decomposed into $(I_2 + N + R_k)$, where $I_2$ is the $2 \times 2$ identity matrix, $N$ is the nilpotent matrix
\begin{align}
N = \begin{pmatrix}
0 & 1\\
0 & 0
\end{pmatrix}
\end{align}
and the remainder term is
\begin{align*}\label{Bk}
	R_k =  \begin{pmatrix}
		\alpha_k^3 \frac{\d f}{\d x}(x_k, \alpha_k) &  3\alpha_k^2 f(x_k,\alpha_k) + \alpha_k^3 \frac{\d f}{\d \alpha}(x_k,\alpha_k)\\
		\alpha_k^4 \frac{\d g}{\d x}(x_k,\alpha_k) & 4 \alpha_k^3 g(x_k,\alpha_k) + \alpha_k^4 \frac{\d g}{\d \alpha}(x_k, \alpha_k)
	\end{pmatrix} = \begin{pmatrix}
		O(j^{-3}) & O(j^{-2})\\
		O(j^{-4}) & O(j^{-3})
	\end{pmatrix},
\end{align*}
which is small in norm. First choose $C$ and $j_4$ so that for all $j \geq j_4$ and $1 \leq k \leq j$, $\| R_k \|_{\ell^\infty} \leq Cj^{-2}$ (which is possible by Lemma \ref{angle bound 1}). All matrix norms below are assumed to be $\ell^\infty$. We claim that for $j \geq j_4$ and each $p = 1, \cdots, j$,
\begin{align}
	\begin{split}
	&\left\|  \prod_{k = 1}^p \begin{pmatrix}
		1 + \alpha_k^3 \frac{\d f}{\d x}(x_k, \alpha_k) & 1 + 3\alpha_k^2 f(x_k,\alpha_k) + \alpha_k^3 \frac{\d f}{\d \alpha}(x_k,\alpha_k)\\
		\alpha_k^4 \frac{\d g}{\d x}(x_k,\alpha_k) & 1 + 4 \alpha_k^3 g(x_k,\alpha_k) + \alpha_k^4 \frac{\d g}{\d \alpha}(x_k, \alpha_k)
	\end{pmatrix} -
\begin{pmatrix}
	1 & p\\
	0 & 1
\end{pmatrix} \right\|\\
&\leq \frac{C(1 + 2p) (1 + C/j^2)^p}{j^2}.
\end{split}
\end{align}
The claim is clearly true for $p = 1$. Assuming it holds for $1, \cdots, p-1$, we then have
\begin{align*}
	\begin{split}
		&\left\|  \prod_{k = 1}^p (I_2 + N + R_k) - (I_2 + p N)\right \| \\
		&= \left\| \left(\prod_{k = 1}^{p-1} (I_2 + N + R_k) - (I_2 + (p-1) N)\right) (I_2 + N + R_p) - (I_2 + (p-1) N) R_p \right\|\\
		&\leq \left(\frac{C (1 + 2(p -1))(1 + C/j^2)^{p-1}}{j^2} \right) (1 + C/j^2)  + \|(I_2 + (p-1)N) R_p\|\\
		& \leq \left(\frac{C (1 + 2p )(1 + C/j^2)^{p}}{j^2} \right) - \left( \frac{2 C (1 + C/j^2)^p}{j^2 }\right) + \frac{2 C}{j^2}\\
		& \leq \left(\frac{C (1 + 2p) (1 + C/j^2)^{p}}{j^2}\right).
	\end{split}
\end{align*}
Hence, the claim also follows for $p$. Having shown this bound for all $1 \leq p \leq j$, $(j \geq j_4)$, it follows from the case $p = j$ that
\begin{align}
	\| D \hat{\beta}^j(x_1, \alpha_1) - (I_2 + j N) \| \leq \left(\frac{C (1 + 2j) (1 + C/j^2)^{j}}{j^2}\right) \leq \frac{3 C e^C}{j}.
\end{align}
%We have
%\begin{align}
%\begin{split}
%D \widehat{\beta}^j &= (I_2 + N + O(1/j^2))^j = \sum_{k = 0}^j \binom{j}{k}\begin{pmatrix}
%O(j^{-3}) & 1 + O(j^{-2})\\
%O(j^{-4}) & O(j^{-3})
%\end{pmatrix}^k
%\\
%&= I_2 + j N + O(1/j)  + \sum_{k = 2}^j \binom{j}{k}(N + O(1/j^2))^k.
%\end{split}
%\end{align}
%Noting that $N$ is nilpotent of order $2$ (i.e. $N^2 = 0$), we can estimate the sum above by
%\begin{align}\label{key decomp}
%\begin{split}
%D\widehat{\beta}^j = \begin{pmatrix}
%1 & j\\
%0 & 1
%\end{pmatrix}
%+ O(1/j) + \sum_{k = 2}^j \binom{j}{k} O(j^{-2(k-1)}).
%\end{split}
%\end{align}
%We want to show that the remainder term $\sum_{k = 2}^j \binom{j}{k}  O(j^{-2(k-1)})$ in \eqref{key decomp} is small. Choosing $j$ large so that $O(j^{-2(k-1)}) \leq 2^{-k} / j $,
%\begin{align}
%\sum_{k \geq 2} O(j^{-2(k-1)}) \leq \frac{1}{j} \sum_{k = 0}^j \binom{j}{k} 2^{-k} = \frac{1}{j}.
%\end{align}
Noting that angular derivatives $\d \pi_1 \widehat{\beta}^j / \d \alpha_1$ and $ \d \pi_2 \widehat{\beta}^j / \d \alpha_1$ are the $(1,2)$ and $(2,2)$ components of $D\widehat{\beta}^j$ respectively concludes the proof of the lemma.
\end{proof}

\begin{rema}
As $j$ increases, Lemma \ref{boundary monotonicity} shows that the Poincar\'e map associated to a periodic orbit of rotation number $1/j$ becomes more and more degenerate, corresponding to the accumulation of caustics at the boundary. In fact, there exist higher order Lazutkin coordinates for which the remainder above could be replaced by $O(j^{-N})$ for any $N \in \N$. However, this is not needed in the remainder of the paper.
\end{rema}

\begin{figure}
\begin{center}
\begin{tikzpicture} [domain=0:2, scale = 1.1]
\begin{scope}[rotate =0]
\coordinate (x) at (2,0);
\coordinate (y) at (0,2);
\coordinate (z) at (3.2,-1.2);
\coordinate (a) at (-2.75,2);
\coordinate (b) at (2.75,2);
\coordinate (c) at (5.2,0.133333333);
\coordinate (d) at (1.2, -2.5333333333);
\draw[black, thick] (x)--(z);
\draw[black, thick] (x)--(y);
\draw[black, thick] (a)--(b) node[above]{$T_{x_1}^*\d \Omega$};
\draw[black, thick] (c)--(d) node[below right]{$T_{x_{-1}}^*\d \Omega$};
%NEED TO DRAW THE TANGENT AT Z AND PUT IN THE ANGLES
\draw[black, thick, ->] (x)--(2,0.9) node[above right]{$\xi_0^+$};
%covector \xi_0^+ at x_0
\draw[black, thick] (0,0) ellipse (4cm and 2cm);
%big ellipse domain
\draw[black, dashed] (0,0) ellipse (2cm and 1cm);
%equidistant folliation curve
\draw[black, dashed] (x) circle (0.9cm);
%fiber circle at x_0
\draw[fill=black] (x) circle (2pt) node[label = {[xshift = 0.7cm, yshift = -0.42cm] $x_0$}]{};
%label of x_0
\draw[fill=black] (z) circle (2pt) node[below right]{$x_{-1}$};
\draw[fill=black] (y) circle (2pt) node[above left]{$x_1$};
\fill[fill opacity = 0.8, fill=green!20!white]
    (2,0) -- (2 + 0.9 * cos 110, 0.9*sin 110 ) arc (110:160:0.9cm) -- cycle;
\fill[fill opacity = 0.8, fill=green!20!white]
    (2,0) -- (2 + 0.9 * cos 200, 0.9*sin 200 ) arc (200:250:0.9cm) -- cycle;
\fill[fill opacity = 0.8, fill=blue!20!white]
    (2,0) -- (2 + 0.9 * cos 290, 0.9*sin 290 ) arc (290:340:0.9cm) -- cycle;
\fill[fill opacity = 0.8, fill=blue!20!white]
    (2,0) -- (2 + 0.9 * cos 20, 0.9*sin 20 ) arc (20:70:0.9cm) -- cycle;
\path 
  pic["$\alpha_1$",draw=black,<->,angle eccentricity=1.5,angle radius=0.75cm] {angle=x--y--b};
\path  
  pic["$\alpha_{-1}$",draw=black,<->,angle eccentricity=1.5,angle radius=0.5cm] {angle=c--z--x};
\end{scope}
\end{tikzpicture}
\end{center}
\caption{A counterclockwise perturbation of $\xi_0^+$. The green cones at $x_0$ are the length one covectors in $C_{x_0,+}^*(\Omega;j)$ and the blue cones at $x_0$ are the length one covectors in $C_{x_0,-}^*(\Omega;j)$.}
\label{picture of counterclockwise perturbation}
\end{figure}
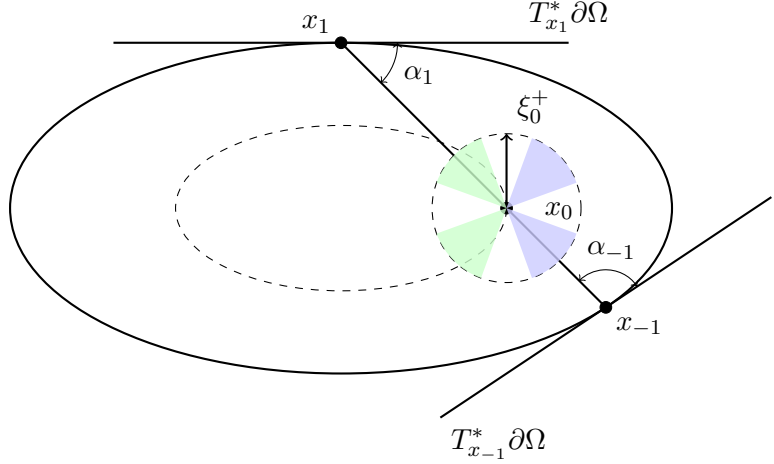
%\red{Should use $\widehat{\beta}^j$ or $\beta^j$?}
\noindent The point $x_j$ can be obtained in two ways. The first way is by flowing forwards to the boundary point $x_1 := \pi_1 g^{t_1^+}(x_0, \xi)$, iterating the billiard map $j$ times and projecting onto the first component. The second way is by flowing backwards to the boundary point $x_{-1} := g^{t_1^-}(x_0, \xi)$ and iterating the billiard map $j+1$ times. If $\xi$ is rotated in the clockwise direction away from $\xi_0^+$, it is convenient to use the angle $\alpha_{1}$ at $x_{1}$, since it was shown in Lemma \ref{L1} that $\alpha_{1}$ is increasing in $\theta$. If $\xi$ is rotated counterclockwise away from $\xi_0^+$, it is instead advantageous to use the point $(x_{-1}, \alpha_{-1})$, since this again corresponds to a \textit{clockwise} perturbation of $\xi_0^-$. In either regime $\theta < \theta_0^+$ or $\theta > \theta_0^+$, we can then take advantage of the monotonicity of $x_j$ in $\alpha_{\pm 1}$ as $\xi$, or equivalently $\theta$, varies. The next lemma shows that the point $x_j$ is monotonically increasing for all orbits making approximately one rotation as $\xi(\theta)$ winds in either direction.

\begin{lemm}\label{winding in xi}
There exist $C_7 > 0$ and $j_5 = j_5(\Omega) \in \N$ such that for all $\theta \in (0, \pi/2)$ corresponding to a $j \geq j_5$ reflection orbit emanating from $(x_0, \xi(\theta))$ which makes approximately one rotation, we have $\d x_j / \d \theta \geq C_7 j$. Similarly, for $\theta \in (-\pi/2, 0)$ corresponding to an orbit of $j$ reflections and approxmately one rotation, we have $\d x_j / \d \theta \leq - C_7 j$.
\end{lemm}
\begin{proof}
Denote $x_j = \pi_1 \beta^j(x_1,\alpha_1)$, where $\pi_1$ is the projection onto the first component. Then $x_j$ depends only on $(x_0, \xi)$, or equivalently $(x_0, \theta)$ and can be written as the composition
\begin{align*} %\label{xj composition}
x_j = \pi_1 \beta^j(x_1(x_0,\theta), \alpha_1(x_0,\theta)).
\end{align*}
As $\theta$ increases or decreases from $0$, we have
\begin{align}\label{chain rule}
\frac{\d x_j}{\d \theta} = \frac{\d x_j}{\d x_1} \frac{\d x_1}{\d \theta} + \frac{\d x_j}{\d \alpha} \frac{\d \alpha}{\d \theta}.
\end{align}
Recall from Lemma \ref{boundary monotonicity} that
\begin{align}\label{large derivative}
C_3 j \leq \frac{\d x_j}{\d \alpha} \leq C_4 j,
\end{align}
for $j$ large and positive constants $C_3$ and $C_4$. This derivative can be made arbitrarily large by choosing $j$ accordingly. Also, Lemma \ref{L3} showed that all orbits making approximately one rotation with $j$ reflections at the boundary have angles of reflection in the range $C_1/j \leq \alpha \leq C_2/j$ for positive constants $C_1$ and $C_2$. We can now use Lemma \ref{L1}, which showed that $|\d \alpha/\d \theta| \geq c > 0$ independently of $j$ for all $\theta$ producing an orbit of approximately one rotation in $j$ reflections. Using Lazutkin coordinates and Lemma \ref{boundary monotonicity}, we also have that
\begin{align}
\frac{\d x_j}{\d x_1} \frac{\d x_1}{\d \theta} = (1 + O(1/j)) \frac{\d x_1}{\d \theta}.
\end{align}
In graph coordinates (see proof of Lemma \ref{L1}), one can calculate that
\begin{align}
\begin{split}
\left|\frac{\d x_1}{\d \theta}\right| &= \left| \frac{\d x_1}{\d x} \frac{\d x}{\d \theta}\right| = \sqrt{1 + (f'(x))^2} \frac{x f'(x) + b - f}{x^2 + (b - f(x))^2}\\
& \leq 1 + O(b) \leq 2
\end{split}
\end{align}
which is bounded independently of $j$. Hence, the term \eqref{large derivative} dominates in the expression \eqref{chain rule} and $x_j$ winds monotonically around $\d \Omega$.
\end{proof}

\noindent At the endpoints $\theta = \pm \pi/2$, it is clear that the angle of reflection at $x_{\pm 1}$ is $\pi/2$ since the distance curves are parallel. Angles outside the range $\theta \in (-\pi/2, \pi/2)$ correspond to clockwise orbits and Lemma \ref{L3} showed that for an orbit making approximately one rotation and $j$ reflections at the boundary, all angles of reflection satisfy $C_1/j \leq \alpha_k \leq C_2/j$, $(1 \leq k \leq j)$. If $j$ is large and $x_0$ is $1/j^4$ close to the boundary as in the statement of Theorem \ref{8 orbit lemma}, it is clear that $j$ reflection orbits emanating from $\xi_0^+$ (corresponding to $\theta = 0$) will make at most a quarter rotation: in graph coordinates $\d \Omega = \text{graph} \{\kappa t^2 + R(t) \}$  (as in the proof of Lemma \ref{L1}), the orbit with initial covector $\xi_0^+$ intersects $\d \Omega$ at $(t_1, j^{-4}) \sim (\kappa^{-1/2} j^{-2}, j^{-4})$, where the tangent line has slope $m \sim 2 \kappa^{1/2} j^{-2}$. The angle at this first point of impact point on the boundary is then $\sim \arctan 2 \kappa^{1/2} j^{-2}  \sim 2 \kappa^{1/2} j^{-2}) = O(j^{-2})$. Iterating the billiard map $j$ times in Lazutkin coordinates then gives
\begin{align*}
	x_j = x_1 +\sum_{k =1}^{j-1} \left( \alpha_k + \alpha_k^3 f(x_k, \alpha_k) \right).
\end{align*}
Using Lemma \ref{angle bound 1} with $c_2 = 1$ and $j$ replaced with $j^2$ then gives
\begin{align*}
	|x_j - x_1| \leq (j-1) e j^{-2} + (j-1) e^3 j^{-6} \sup_{B^* \Omega} |f| = O(j^{-1}).
\end{align*}
In other words, the base point moves $O(1/j)$ along the boundary, which can certainly be made less than $|\d \Omega|/4$.
\\
\\
It is then clear that the collection of covectors at $x_0$ whose trajectories make approximately one rotation in $j$ reflections consists of two connected components in $S_{x_0}^*\Omega$:
\begin{def1}
The positive admissible cone $C_{x_0, +}^*(\Omega; j)$ at $x_0$ is defined to be the set of homogeneous extensions to $T_{x_0}^* \Omega$ of the two components in $S_{x_0}^*(\Omega)$ described above. The negative admissible cone $C_{x_0, -}^*(\Omega; j)$ at $x_0$ is defined by the same property for orbits making $j$ reflections in the clockwise direction. Their union is denoted by $C_{x_0}^*(\Omega;j) = C_{x_0, +}^*(\Omega;j) \cup C_{x_0, -}^*(\Omega;j)$.
\end{def1}
\noindent See Figure \ref{picture of counterclockwise perturbation} for an illustration of $C_{x_0}^*(\Omega;j)$.
%We can now estimate the final point of impact $x_j$ at the boundary in terms of $\xi$.

\subsection{Intersection points}
We are finally ready to show that the intersection points of the last link with the distance curve of $y$ wind around monotonically as we twist $\xi$ in either direction. We first explain why the last link necessarily intersects $d^{-1}(\dist(y, \d \Omega) )$ twice.
\begin{lemm}\label{2int}
There exists $j_6 = j_6(\Omega) \in \N$ such that for $j \geq j_6$ and any $N \geq 4$, the distance curve $\{x \in \Omega: \text{dist}(x,\d \Omega) = j^{-N}\}$ is intersected exactly twice by any link emanating from the boundary at an angle greater than or equal to $C_1/j$, with $C_1$ the same constant appearing in Lemma \ref{L3}. 
\end{lemm}
\begin{proof}
For each point $p$ in the distance curve $\{z \in \Omega : d(z, \d \Omega) = j^{-N} \}$, the tangent line $T_p \{z \in \Omega : d(z, \d \Omega) = j^{-N} \}$ intersects $\d \Omega$ exactly twice by convexity. Recall that Lemma \ref{L3} gave $C_1/j \leq \alpha_k \leq C_2/ j$ for each $1 \leq k \leq j$. We now show that the angles of reflection for links which are tangent to the distance curve $d^{-1}(j^{-N})$ are of order $O(j^{-N/2})$. In graph coordinates, $\d \Omega$ is locally given by $(x,f(x))$ where $f(x) = \kappa x^2 + R(x)$, $\kappa$ is the curvature of $\d \Omega$ at the point corresponding to $(0,0)$ and $R(x) = O(x^3)$. By rescaling, we may assume that $\kappa = 1$. The distance curve $d^{-1}(j^{-N})$ can be parametrized in graph coordinates by 
\begin{align*}
d^{-1}(j^{-N}) = \left\{ (t, f(t)) + r N(t): t \in \R  \right\},
\end{align*}
where $r = j^{-N}$ and
\begin{align*}
N(t) = \frac{(-f'(t), 1)}{\sqrt{1+ (f'(t))^2}}
\end{align*}
is the unit normal to the graph at $(t,f(t))$. There exist precisely two lines through the origin which are tangent to $d^{-1}(j^{-N})$ in graph coordinates. The positive parameter $t$ corresponding to a point of tangency satisfies
\begin{align}\label{slopes}
\frac{f(t) + r/(1 + (f'(t) )^2)^{1/2}}{t - rf'(t)/(1 + (f'(t) )^2 )^{1/2} } = \frac{\d_t (f(t) + r/(1 + (f'(t) )^2 )^{1/2} ) }{\d_t (t - rf'(t) /(1 + (f'(t) )^2)^{1/2})},
\end{align}
where the lefthand side of \eqref{slopes} is the slope of a line connecting the origin to a point on $d^{-1}(r)$ and the righthand side is the slope of the tangent to $d^{-1}(r)$. Solving for $t$ in terms of $r$ in equation \eqref{slopes} yields $f'(t) t - f(t) = t^2 + O(t^3) = O(r)$. Plugging this into the righthand side of \eqref{slopes}, we see that the angle of tangency satisfies $\tan (\alpha_{\text{tangency}}) = O(r^{1/2})$ and hence $\alpha_{\text{tangency}} = O(j^{-2})$ if $N \geq 4$. The proof is complete by noting that the angles $\alpha_k$ coming from orbits making approximately one rotation and $j$ reflections are bounded below by $C_1/j > O(j^{-4/2})$ for $j$ sufficiently large.

%We calculate that the angle of reflection corresponding to a counterclockwise link tangent to $d^{-1}(j^{-N})$ is given by
%\begin{align*}
%\alpha_{\text{tangency}, j^{-N}} = 
%\end{align*}

%As in the proof of Lemma 4.2 \ref{L3}, Proposition 14.1 in \cite{Lazutkin5} guarantees that for each link adjoining $x_k$ to $x_{k+1}$ with angles $\alpha_k$ and $\alpha_{k+1}$ respectively, we have
%\begin{align}
%2 \alpha_k \rho_\text{min} \leq x_{k+1} - x_k \leq 2 \alpha_k \rho_\text{max},
%\end{align}
\end{proof}

\begin{lemm}\label{clockwise int points}
Under the hypotheses of Lemma \ref{2int}, denote by $w_1$ and $w_2$ the arclength coordinates on $d^{-1}(d(y))$ corresponding to the intersection points of $d^{-1}(d(y))$ and the last link of the $j$ reflection orbit emanating from $x$. There exist $C_8 > 0$ and $j_7 = j_7(\Omega) \in \N$ such that if $j \geq j_7$ and $x,y$ are $O(1/j^4)$ close the diagonal of the boundary, then $|\d w_i/ \d\theta| \geq C_8 j$ ($i = 1,2$) for all $\theta$ corresponding to $\xi(\theta)$ in the cone of admissible covectors at $x_0$.
\end{lemm}
\begin{proof}
In graph coordinates, $\d \Omega$ is again locally parametrized by $(t,f(t))$, where $f(t) = \kappa t^2 + R(t)$, $\kappa$ is the curvature of $\d \Omega$ at the point corresponding to $(0,0)$ and $R(t) = O(t^3)$. By rescaling, we may assume that $\kappa = 1$. The distance curve $d^{-1}(d(y))$ on which $y$ lies can be locally parametrized in graph coordinates by 
\begin{align}\label{d parametrization}
d^{-1}(d(y)) = \left\{ (t, f(t)) + r N(t)  \right\},
\end{align}
where $r = \text{dist}(y,\d \Omega) = O(j^{-4})$ and
\begin{align*}
N(t) = \frac{(-f'(t), 1)}{\sqrt{1+ (f'(t))^2}}
\end{align*}
is the unit normal to the graph at $(t,f(t))$. The final link of the orbit connecting $x_j(\theta) \in \d \Omega$ to $x_{j+1}(\theta) \in \d \Omega$ can be parametrized by the line
\begin{align}\label{last link line}
Y = \tan (\alpha_j(\theta) + \beta_j(\theta)) X + f(t_j(\theta)) - \tan(\alpha_j(\theta) + \beta_j(\theta)) t_j(\theta),
\end{align}
where
\begin{align*}
\beta_j = \arccos \left( \frac{1}{\sqrt{1 + (f'(x_j))^2}}  \right)
\end{align*}
is the angle of the tangent to the graph with the horizontal axis and $t_j = t_j(\theta) \in \R$ is defined implicitly by the equation $(t_j,f(t_j)) = x_j \in \d \Omega$. Noting that the Cartesian coordinates of the parametrization of $d^{-1}(d(y))$ satisfy
\begin{align}
\begin{split}
X &= t - \frac{rf'(t)}{\sqrt{1+ (f'(t))^2}},\\
Y &= f(t) + \frac{r}{\sqrt{1+ (f'(t))^2}},
\end{split}
\end{align}
we plug these into equation \eqref{last link line} to obtain
\begin{align}\label{implicit eqn}
\begin{split}
\left(f(t) + \frac{r}{\sqrt{1+ (f'(t))^2}}\right) &= \tan (\alpha_j(\theta) + \beta_j(\theta)) \left(t - \frac{rf'(t)}{\sqrt{1+ (f'(t))^2}} \right)\\
&+ f(t_j(\theta)) - \tan(\alpha_j(\theta) + \beta_j(\theta)) t_j(\theta).
\end{split}
\end{align}
\noindent By Lemma \ref{2int}, there exist precisely two solutions $t = z_1(\theta), z_2(\theta)$ of equation \eqref{implicit eqn} which correspond to the two intersection points of the last link with the distance curve $d^{-1}(d(y))$. Plugging $z_{i}(\theta)$ into equation \eqref{implicit eqn}, differentiating in $\theta$ and evaluating at $\theta_0$ corresponding to $x_j(\xi(\theta_0))$ at the origin, we obtain
\begin{align}\label{zprime}
\begin{split}
f'(z_i)& z_i' - \frac{r f'(z_i) f''(z_i) z_i'}{(1 + f(z_i)^2)^{3/2}} =   (\alpha_j' + \beta_j') (1 + \tan^2 \alpha_j) \left(z_i - \frac{rf'(z_i)}{\sqrt{1+ (f'(z_i))^2}} \right)\\
&+  \tan \alpha_j \left(z_i' - \frac{rf''(z_i) z_i'}{\sqrt{1+ (f'(z_i))^2}}  + \frac{r(f'(z_i))^2 f''(z_i) z_i'}{(1 + (f'(z_i))^2 )^{3/2}} \right) + 0 - \tan \alpha_j t_j'.
\end{split}
\end{align}
We now collect all terms with $z_i'$ in \eqref{zprime}: 
\begin{align}\label{divide through}
\begin{split}
A_i z_i' &=   (\alpha_j' + \beta_j') (1 + \tan^2\alpha_j ) \left(z_i - \frac{rf'(z_i)}{\sqrt{1+ (f'(z_i))^2}} \right) - \tan \alpha_j t_j',
\end{split}
\end{align}
where
\begin{align*}
A_i = f'(z_i)  - \frac{r f'(z_i) f''(z_i)}{(1 + f(z_i)^2)^{3/2}} -  \tan \alpha_j \left(1 - \frac{rf''(z_i)}{\sqrt{1+ (f'(z_i))^2}}  + \frac{r(f'(z_i))^2 f''(z_i) }{(1 + (f'(z_i))^2 )^{3/2}}\right).
\end{align*}
Before dividing through equation \eqref{divide through} by $A_i$, we Taylor expand $A_i$ in powers of $1/j$ to show that it is nonvanishing:
\begin{align}
A_i = 2 z_i - \tan \alpha_j + O(j^{-2}).
\end{align}
At this point, it is important to distinguish between $z_1$ and $z_2$. If $z_1 < z_2$, then insisting that $y$ be $O(j^{-4})$ close to the boundary amounts to setting $z_1 = O(j^{-4})$ and $z_2 = t_{j+1} + O(j^{-4})$. Here, $t_{j+1} = t_{j+1}(\theta)$ is defined implicitly by the equation $(t_{j+1}, f(t_{j+1})) = x_{j+1} \in \d \Omega$. When $i = 1$, $A_1 = - \alpha_j + O(j^{-4}) \leq -C_1 / (2j)$ is nonvanishing for $j$ sufficiently large. For $i = 2$, note that $\tan \alpha_j = f(z_2)/z_2 + O(j^{-4})$, which implies that $A_2 = z_2 + O(j^{-2})$ is again nonvanishing. Hence, we may divide through equation \eqref{divide through} by $A_i$. To estimate the righthand side of \eqref{divide through}, we calculate that
\begin{align}
\beta_j' = \frac{f'(t_j)}{\sqrt{1 - \frac{1}{1 + (f'(t_j))^2 }}} \frac{ f''(t_j) t_j '}{(1 + (f(t_j)^2))^{3/2}}.
\end{align}
The first factor above is well defined by continuity at $t_j = 0$ and equals $1$. Hence, $\beta_j' = 2 t_j' + O(1/j)$. If $i = 1$, equation \eqref{divide through} becomes
\begin{align}
(- \tan \alpha_j  +O(j^{-4})) z_1' = - \tan \alpha_j t_j' + O(j^{-3}),
\end{align}
which implies that $z_1' = t_j' + O(j^{-2})$. When $i = 2$, equation \eqref{divide through} becomes
\begin{align}
(z_2 + O(j^{-2}))z_2' = t_j' z_2 + O (1/j),
\end{align}
which again implies that $z_2' = t_j' + O(1/j)$. Recall that Lemma \ref{winding in xi} showed $x_j'$ is of order $j$ in Lazutkin coordinates when $\xi(\theta) \in C_{x_0}^*(\Omega;j)$. To compare Lazutkin coordinates with $t_j$, note that in graph coordinates, the arclength parameter $ds$ on $\d \Omega$ is given by
$$
ds = (1 + (f'(t))^2)^{1/2} dt.
$$
As arclength coordinates and Lazutkin coordinates are comparable independently of $j$, we conclude that $t_j'$ and hence $z_i'$ are also of order $j$. We conclude the proof by noting that the arclength parameter $ds'$ on $d^{-1}(d(y))$ in the parametrization \eqref{d parametrization} is given by
$$
ds' = \left((1 + (f'(t))^2)^{1/2} + O(r)\right) dt,
$$
which is also comparable to $dt$ independently of $j$.
\end{proof}
From Lemma \ref{clockwise int points}, we conclude that perturbing $\xi_0^+$ either clockwise or counterclockwise within the cone of admissible covectors (but still such that the corresponding orbit is counterclockwise) results in monotone increasing arclength coordinates for the intersection points with $d^{-1}(d(y))$. By the intermediate value theorem, such intersection points will then coincide with $y$ exactly $4$ times for orbits making approximately one rotation (two intersection points for the clockwise perturbation and two intersections points for the counterclockwise perturbation).

%We now show that this monotonicity implies the existence of periodic orbits. Given that $c/j \leq \alpha \leq C/j$ for all orbits of $j$ reflections connecting $x$ to $y$, we can insist that $y$ be sufficiently close to $x$ and $\d \Omega$ so that all links making an angle greater than $c_\kappa/j$ with a tangent line to $\d \Omega$ will intersect the level set $\{y' : d(y',\d \Omega) = d(y,\d \Omega)\}$ exactly twice. As each level set of the distance function is diffeomorphic to a circle, we want to show that the intersection points of the final link of an orbit of $j$ reflections will wind around the level set of $y$ monotonically. This is obvious however, as $x_j$ moves with approximate speed $j$ and the final angle satisfies $\alpha_j = \alpha + j O_s(\alpha^2)$, so that
%$$
%\frac{\d \alpha_j}{\d \alpha} = 1 + j O(\alpha) = O(1).
%$$

%\red{By symmetry, the same will hold for counterclockwise perturbations of $\xi_0^-$, although we only focus on perturbations of $\xi_0^+$ in this section as all clockwise orbits can then be obtained via conjugation by a reflection through the vertical axis.}

\subsection{Clockwise orbits}
From Lemma \ref{clockwise int points}, we saw that there were precisely $4$ counterclockwise orbits connecting $x = x_0$ to $y$ in $j$ reflections and approximately one rotation. The only constraint on $x$ and $y$ was that they were confined to an $O(j^{-4})$ neighborhood of the diagonal of the boundary. By reflecting $\Omega$ through the verticle axis, one obtains another smooth strictly convex domain and the reflections of $x$ and $y$ remain $O(j^{-4})$ close to the diagonal of the boundary. Hence, the same procedure produces exactly $4$ counterclockwise orbits of $j$ reflections from $x$ to $y$ making approximately one rotation in the reflected domain. Reflecting back through the vertical axis carries these $4$ orbits to clockwise orbits in the original domain. This completes the proof of Theorem \ref{8 orbit lemma}.

\begin{rema}\label{Uj}
We can take $j_0 = j_0(\Omega) = \max\{j_i: 2 \leq i \leq 7\}$, with $j_i$ as they appear in Lemmas \ref{L3}, \ref{L1}, \ref{boundary monotonicity}, \ref{winding in xi}, \ref{2int}, and \ref{clockwise int points}. Similarly, we can choose a uniform constant $C_0$ as in the statement of Lemma \ref{8 orbit lemma}. The tubular neighborhood $U_j$ referenced in Theorem \ref{HRP} can be taken to be $\{(x,y) \in \Omega \times \Omega: \text{dist}((x,y), \Delta \d \Omega) \leq j^{-4}\}$ for $j \geq j_0$, where $\Delta: \d \Omega \times \d \Omega$ is the diagonal embedding $x \mapsto (x,x)$. For $(x,y) \in U_j$, both $x$ and $y$ satisfy the conditions of Lemma \ref{clockwise int points}. The cone bundle $C_{x}^*(\Omega;j)$ is well defined whenever $\dist(x,\d \Omega) = O(j^{-4})$.
\end{rema}

%The same characterization also applies to the four clockwise orbits in Theorem \ref{8 orbit lemma}, which can be obtained by reflecting the domain through the vertical axis, finding all orbits making approximately one counterclockwize rotation from $x$ to $y$, and then reflecting these orbits back through the vertical axis (see Section \ref{proof of lemma} for a more detailed discussion).}

\begin{rema}
The proof of Theorem \ref{8 orbit lemma} in this section could actually be extended to a larger region of validity. In particular, the same methods allow us to prove the existence of $8$ orbits of rotation number $k/j$ for $j$ sufficiently large in terms of $k$, connecting points in a comparable open neighborhood of the diagonal of the boundary. Additionally, we could allow $x$ and $y$ to be further away from the diagonal as long as they are both sufficiently close to the boundary. However, we are only concerned with near diagonal terms in this paper for the purposes of deriving trace formulas.
\end{rema}

\begin{rema}\label{boundary points}
If $x \in \d \Omega$, then $\xi_0^\pm$ are tangent to the boundary and perturbing $\xi$ away from $\xi_0^+$ in the clockwise direction is no longer well defined. Rather, it is equivalent to reflecting and then rotating $\xi$ in the counterclockwise direction. Similarly, $\xi_0^-$ cannot be rotated in the counterclockwise direction. Each of these restrictions reduces the number of $j$ reflection orbits to $y$ in appoximately one rotation by $2$. If additionally $y \in \d \Omega$, the final link only makes one intersection with $d^{-1}(0) = \d \Omega$ so there are only two orbits of $j$ reflections connecting $x$ to $y$ in approximately one rotation. One is in the counterclockwise direction while the other is in the clockwise direction. In particular, when $x = y$, there is a \textit{unique} geodesic loop (up to parametrization) of $j$ reflections and exactly one rotation (i.e. rotation number $1/j$). The existence of such loops is well known and the proof for boundary points is much simpler, as it only requires Lemma \ref{boundary monotonicity}
\end{rema}

\section{A parametrix for the wave propagator}\label{A parametrix for the wave propagator}
\noindent In this section, we use microlocal analysis to obtain an oscillatory integral parametrix for the wave propagator in an open subset of $\R \times \Omega \times \Omega$ which contains the diagonal $\R \times \Delta(\d \Omega)$. The wave kernel $e^{it\sqrt{- \Delta}}$ is actually \textit{not} a Fourier integral operator (FIO) near the tangential rays (see \cite{AndersonMelrose}, \cite{MelroseTaylor}), so we microlocalize the wave kernels near periodic transversal reflecting rays. We begin by reviewing FIOs and Chazarain's parametrix for the wave propagator in Sections \ref{Fourier Integral Operators} and \ref{Chazarain's Parametrix}. In Section \ref{Oscillatory integral representation}, we then cook up oscillatory integrals for each term in Chazarain's parametrix, which microlocally approximate the wave propagator near the orbits described in Theorem \ref{8 orbit lemma}. This approach is inspired by the formulas in \cite{MaMe82} and will rely on the symbol calculus in Section \ref{Chazarain's Parametrix}.

\subsection{Fourier Integral Operators}\label{Fourier Integral Operators}
Let $X$ and $Y$ be open sets in $\R^{n_X}$ and $\R^{n_Y}$ respectively. If $a \in S^\mu_{1, 0}(X \times \R^N)$ is a classical symbol of order $\mu$ and $\Theta \in C^\infty(X  \times \R^N)$ is a nondegenerate phase function, then the linear operator
$$
A(u) = \int_X \int_{\R^N} e^{i \Theta(x,\theta)} a(x,\theta) u(x) \,d\theta dx
$$
is called a Lagrangian or Fourier integral distribution on $X$. Recall that a continuous linear operator $A : C_0^\infty(Y) \to \mathcal{D}'(X)$ has an associated Schwartz kernel $K_A \in \mathcal{D}'(X \times Y)$. If $K_A$ is given by a locally finite sum of Lagrangian distributions on $X \times Y$, then we say $A$ is a Fourier integral operator (FIO). One can then show that the wavefront set of the kernel is contained in the image of the map $\iota_{\Theta}: (x, y) \mapsto (x, y, d_x \Theta, d_y\Theta)$ when restricted to the critical set $C_\Theta := \{d_\theta \Theta = 0\}$. The image of $\iota_\Theta$ is in fact a conic Lagrangian submanifold $\Lambda_\Theta \subset T^* (X \times Y)$ and the map $\iota_{\Theta}$ is a local diffeomorphism from $C_\Theta$ onto $\Lambda_\Theta$. In this case, we say ``$\Theta$ parametrizes $\Lambda_\Theta$.'' The canonical relation or wavefront relation of $A$ is defined by
$$
WF'(A) = \{(x,\xi), (y, \eta) : (x,y,\xi, -\eta) \in WF(K_A)\} \subset T^*X \times T^*Y,
$$
and describes how the operator $A$ propagates singularities of distributions on which it acts. If $\omega_X$ and $\omega_Y$ denote the natural symplectic forms on $T^*X$ and $T^*Y$ respectively, one can more invariantly consider FIOs associated to general conic Lagranigan submanifolds $\Lambda \subset T^*X \times T^*Y$ (canonical relations), with respect to the symplectic form $\omega_X - \omega_Y$. The notion of a principal symbol for Fourier integral operators is more subtle than that for pseudodifferential operators. The principal symbol of $A$ is a half density on $\Lambda_\Theta$ given in terms of the parametrization $\iota_\Theta$:
\begin{align}\label{principal symbol of an FIO}
e = {\iota_\Theta}_* (a_0 |dC_\Theta|^{1/2}),
\end{align}
where $a_0$ is the leading order term in the asymptotic expansion for $a$ and $|dC_\Theta|^{1/2}$ is the half density associated with the Gelfand-Leray form on the level set $\{d_\theta \Theta = 0\}$. Here, we have ignored Maslov factors coming from the Keller-Maslov line bundle over $\Lambda_\Theta$. These are nonzero factors $e^{i\sigma \pi/4}$ ($\sigma$ is known as the Maslov index) which appear in front of the principal symbol as a result of the multiplicity of phase functions parametrizing the canonical relation $\Lambda_\Theta$, possibly in different coordinate systems. While these factors allow the principal symbol to be defined in a more geometrically invariant way, we defer computation of the Maslov indices until Section \ref{Maslov factors}. For a more thorough reference on the global theory of Lagrangian distributions, see \cite{Du96}, \cite{Ho71}, \cite{DuHo72}, and \cite{Ho485}. The order of a Fourier integral operator is defined in such a way that when two Fourier integral operators' canonical relations intersect transversally, then the composition is again a Fourier integral operator and order of the composition is the sum of the orders:
\begin{align}\label{order of a lagrangian distribution}
\text{order}(A) = m = \mu + \frac{1}{2}N - \frac{1}{4}(n_X + n_Y).
\end{align}
Recall that here, $n_X$ and $n_Y$ are the dimensions of $X$ and $Y$ respectively. In this case, we write $A \in I^m(X \times Y, \Lambda)$. This convention on orders also generalizes that of pseudodifferential operators, where $X = Y$ and $m = \mu$ coincides with the order of the corresponding symbol class. Sufficient conditions which guarantee that the composition exists are clean or transversal intersection of the two operators' canonical relations. In general, composition of Fourier integral operators and the associated symbol calculus is somewhat complicated, but we will not directly use the composition formula in what follows.

\subsection{Chazarain's parametrix}\label{Chazarain's Parametrix}
The parametrix developed by Chazarain in \cite{Ch76} provides a microlocal description of the wave kernels near periodic transversal reflecting rays. The parametrices for
$$
E(t) = \cos t \sqrt{- \Delta}, \qquad \qquad S(t) = \frac{\sin t \sqrt{- \Delta}}{\sqrt{-\Delta}}
$$
are constructed in the ambient Euclidean space $\R \times \R^n \times \R^n$. We only consider $S(t)$, as the formula for $E(t)$ is easily obtained from that of $S(t)$ by differentiating in $t$. We write $E(t,x,y)$ and $S(t,x,y)$ for the Schwartz kernels of $E(t)$ and $S(t)$ respectively. Following the work in \cite{Ch76}, we can find a Lagrangian distribution
\begin{align}\label{Chazarain sum}
\widetilde{S}(t,x,y) = \sum_{j = - \infty}^{\infty} S_j(t,x,y), \qquad S_j \in I^{-5/4}(\R \times \R^n \times \R^n, \Gamma_\pm^j),
\end{align}
which approximates $S(t)$ microlocally away from the tangential rays modulo a smooth kernel. We will describe the canonical relations $\Gamma_\pm^j$ momentarily and in particular, show that the sum in \eqref{Chazarain sum} is locally finite. We first explain what is meant by approximating $S(t)$ ``microlocally away from the tangential rays.'' In general, two distributions $f,g \in \mathcal{D}'(\R^n)$ are said to agree microlocally near a closed cone $\Lambda_1 \subset T^*\R^n$ if $WF(u - v) \cap \Lambda_1 = \emptyset$. Similarly, using the language from Section \ref{Fourier Integral Operators}, two operators $A, B : C^\infty(Y) \to C^\infty(X)$ are said to agree microlocally near a closed cone $\Lambda_2 \subset T^*X \times T^*Y$ if $WF'(A-B) \cap \Lambda_2 = \emptyset$. This second notion is what we will use to say that our parametrix approximates $S(t)$ microlocally near the canonical relations $\Gamma_{\pm}^j$.
\\
\\
To describe the canonical relations precisely, we first introduce some notation following the presentations in \cite{Ch76} and \cite{HeZe12}. As the Euclidean wave operator $\Box_{\R^2}$ factors into $(\d_t - \sqrt{\Delta})(\d_t + \sqrt{-\Delta})$, there are two Hamiltonians corresponding to the symbol $\pm |\eta|$ of $\pm \sqrt{-\Delta}$. Let $H_\pm(y,\eta) = \pm |\eta|$ and $g^{\pm t}$ be the Hamiltonian flow, i.e. the flow map associated to the system
\begin{align*}
\begin{cases}
\d_t y = \frac{\d H_\pm}{\d \eta},\\
\d_t \eta = -\frac{\d H_\pm}{\d y},
\end{cases}
\end{align*}
which is in fact just the reparametrized geodesic flow on $\R^2$. For $(y,\eta) \in T^*\Omega$ (or $T_{\d \Omega}^*\R^2$ such that $\eta$ is transversal to the boundary and inward pointing), recall that in Section \ref{Billiards} we defined
\begin{align*}
\begin{split}
t_\pm^1(y,\eta) &= \inf_{t > 0} \{t: \pi_1 g^{\pm t}(y, \eta) \in \d \Omega \},\\
t_\pm^{-1}(y,\eta) &= \sup_{t < 0} \{t: \pi_1 g^{\pm t}(y,\eta) \in \d \Omega \},
\end{split}
\end{align*}
where $\pi_1$ is projection onto the spatial variable. We have $t_\mp^{1}(y,\eta) = -t_\pm^{-1}(y,\eta)$. We then set
\begin{align*}
\begin{split}
\lambda_\pm^{1}(y, \eta) &= g^{\pm t_\pm^1}(y,\eta),\\
\lambda_\pm^{-1}(y, \eta) &= g^{\pm t_\pm^{-1}}(y,\eta).
\end{split}
\end{align*}
Also define $\widehat{\lambda_\pm^{1}(y,\eta)}$ to be the reflection of $\lambda_\pm^{1}(y,\eta)$ through the cotangent line at the boundary. In other words, $\widehat{\lambda_\pm^{1}(y,\eta)}$ and $\lambda_\pm^{1}(y,\eta)$ have the same cotangential components but opposite conormal components so that $\widehat{\lambda_\pm^{1}(y,\eta)}$ is inward pointing. The point $\widehat{\lambda_\pm^{-1}(y,\eta)}$ is defined analogously. We can then inductively define $t_\pm^j(y,\eta)$ and $t_\pm^{-j}(y,\eta)$ by the formulas
\begin{align*}
\begin{split}
t_\pm^j &= \inf_{t > 0} \{t : \pi_1 g^{\pm t}(\widehat{\lambda_\pm^{j-1}(y,\eta)}) \in \d \Omega \},\\
t_\pm^{-j} &= \sup_{t < 0} \{t : \pi_1 g^{\pm t}(\widehat{\lambda_\pm^{-(j-1)}(y,\eta)}) \in \d \Omega \}.
\end{split}
\end{align*}
The total travel time after $j$ reflections is defined by
\begin{align*}
T_\pm^j(y,\eta) = \begin{cases}
\sum_{k = 1}^j t_\pm^k(y,\eta) & j > 0\\
0 & j = 0\\
\sum_{k = j}^{-1} t_\pm^k(y,\eta) & j < 0.
\end{cases}
\end{align*}
%If we continue such a construction for all $j \in \Z$, the associated phase function of the corresponding FIO $S_j$ should parameterize the cannonical relations
To study how the fundamental solution of $\Box_{\Omega}$ behaves at $\d \Omega$ when we impose boundary condtions, we propagate the intial data by the free wave propagator on $\R^2$, restrict it to the boundary, reflect, and then propagate again.  If such a construction is continued for $j \in \Z$ reflections at the boundary, it is shown in \cite{Ch76} that the FIOs $S_j$ must have canonical relations
$$
\Gamma_\pm^j = \begin{cases}
(t, \tau, g^{\pm t}(y,\eta), y, \eta): \tau = \pm |\eta| & j =0,\\
(t, \tau, g^{\pm (t - T_\pm^j(y,\eta))}(\widehat{\lambda_\pm^j(y,\eta)}), y, \eta)): \tau = \pm |\eta| & j \in \Z \backslash \{0\}.
\end{cases}
$$
\noindent Since $\widetilde{S}(t)$ is a microlocal parametrix, the canonical relation of the true solution operator $S(t)$ is also
$$
\Gamma = \bigcup_{j \in \Z, \pm} \Gamma_{\pm}^j.
$$
\noindent Here, $j > 0$ and $j < 0$  correspond to reflections on the inside and outside of the boundary. For $t > 0, \pm \tau > 0$, the canonical relations corresponding to $j >0$ project onto the interior of $\Omega$ when $t$ is small (inside bounces). If $t > 0, \pm \tau > 0$ and $j < 0$, the canonical relations project onto the exterior of $\Omega$, corresponding to outside bounces. In fact, there are four modes of propagation associated to the canonical relations $\Gamma_\pm^j$, corresponding to $\pm \tau \geq 0$ (forwards and backwards time) and $\pm j \geq 0$ (inside and outside bounces).  See \cite{HeZe12} for an explicit local model in the case of one reflection, where $\Omega$ is replaced by a half plane. We see that a point $(t, \tau, x, \xi, y, \eta)$ belongs to the canonical relation $\Gamma_{\pm}^j \subset T^*(\R \times \R^2 \times \R^2)$ if the broken geodesic of $j$ reflections emanating from $(y,\eta)$ passes through $(x,\xi)$ in time $t$. For a more thorough discussion of Chazarain's parametrix, see \cite{Ch76} and \cite{GuMe79b}.
\\
\\
\noindent To be precise, Chazarain actually showed that there exists FIOs $S_j$ such that the sum in \eqref{Chazarain sum} is a parametrix for the wave propagator $S(t)$ with canonical relation
$$
\Gamma = \bigcup_{j \in \Z, \pm} \Gamma_{\pm}^j.
$$
However, the canonical relations were never parametrized by explicit phase functions and the principal symbols of the operators $S_j$ were not computed in \cite{Ch76}. For the remainder of this section, we concern ourselves with the task of explicitly computing them in terms of geometric and dynamical data associted to the billiard map.

\subsection{Oscillatory integral representation}\label{Oscillatory integral representation}
In this section, we cook up an oscillatory integral such that microlocally near the canonical relations $\Gamma_\pm^j$,
$$
S_j(t,x,y) = \int_{-\infty}^\infty e^{i\Theta_j(t,\tau, x,y)} {a_j}(\tau, x,y)\,d\tau,
$$
where $S_j$ is the $j^{th}$ term in Chazarain's parametrix corresponding to a wave with $j$ reflections and $a_j \in S_{\text{cl}}^\mu$ is a classical symbol of order $\mu$. We will only compute the principal symbol and use $L.O.T$ to denote lower order terms in the sense of Lagrangian distributions in what follows. Due to the presence of different Maslov factors on each branch of $\Gamma_{\pm}^j$ corresponding to $\pm \tau > 0$ (cf. Sections \ref{Fourier Integral Operators} and \ref{Maslov factors}), it is actually more convenient to find operators
\begin{align}\label{osc int}
S_{j, \pm}(t,x,y) = \int_{0}^\infty e^{i\Theta_{j,\pm}(t,\tau, x,y)} { a_{j,\pm}}(\tau, x,y)\,d\tau,
\end{align}
so that $S_j = S_{j,+} + S_{j,-}$ and the phase functions associated to $S_{j,\pm}$ paramaterize $\Gamma_{+}^j$ and $\Gamma_-^j$ individually.
\\
\\
We first make precise the notion of microlocalized FIOs. We would like to microlocalize $S(t)$ near periodic orbits of rotation number $1/j$. Oftentimes, it is required that such orbits be nondegenerate, in the sense that $1$ is not an eigenvalue of the linearized Poincar\'e map. However, this assumption is not needed for our trace formulas, which work both for simple nondegenerate orbits as well as degenerate orbits coming in one parameter families as in the case of an ellipse (cf. Theorem \ref{Rational caustic theorem} and Corollary \ref{Ellipse corollary}). We recall that $\Omega$ is said to satisfy the noncoincidence condition \eqref{NCC} if there exist no periodic orbits of rotation number $m/n$, $m \geq 2$ having length in a sufficiently small neighborhood of $|\d \Omega|$. In this case, for $j$ sufficiently large, periodic orbits of rotation number $1/j$ come in isolated families. This follows from the results in \cite{MaMe82}, which shows in particular that for $j$ sufficiently large, no two orbits of distinct rotation numbers $1/j$ and $1/k$ can have the same length. 
\\
\\
Recall also that for $j$ sufficiently large, there exists a unique geodesic loop of rotation number $1/j$ at each boundary point $q$, whose length we denote by $\Psi_j(q,q)$ (the $j$-loop function). It can be shown that periodic billiard orbits arise as critical points of the $j$-loop function (see \cite{Vig18}, \cite{HeZe19}). As in the statement of Theorem \ref{Main theorem}, denote $t_j = \inf_{q \in \d \Omega}  \Psi_j(q,q)$ and $T_j = \sup_{q \in \d \Omega} \Psi_j(q,q)$. If $\Omega$ satisfies the noncoincidence condition \ref{NCC}, we can find a smooth cutoff function $\chi_1(t)$ which is identically equal to $1$ on an open nieghborhood of $[t_j,T_j]$ and vanishes in a neighborhood of all other $L \in \text{Lsp}(\Omega)$. As noted in Section \ref{Chazarain's Parametrix}, each propagator $S_j$ has canonical relations $\Gamma_{\pm}^j$. Denote by $\chi_2$ a smooth cutoff function which is identically equal to $1$ on $\Gamma^j = \cup_\pm \Gamma_\pm^j$, vanishing near the gliding rays $T^* \d \Omega$ and conic in the fiber variables $\tau, \xi$ and $\eta$. Quantizing $\chi_2$ gives a pseudodifferential operator which microlocalizes near the support of $\chi_2$. For a reference, see Chapter 18 of \cite{Ho385}. We call such an operator a microlocal cutoff on $\Gamma_\pm^j$. The composition $\chi_1(t)\chi_2(t,x,y, D_t, D_x, D_y) S(t)$ is then smoothing away from the geodesic loops of rotation number $1/j$ and the trace of this composition is equal to the wave trace modulo $C^\infty$ in an open neighborhood of $[t_j,T_j]$. %\red{Recall Theorem \ref{caustics} in Section \ref{Billiards}, which provides an ample number of caustics having simple length in any neighborhood of $|\d \Omega|$. Hence, for $j \in \Z$ large and positive, we can consider periodic orbits having simple length $T_j$, making a single rotation and precisely $j$ reflections at the boundary.}
\\
\\
We now use Theorem \ref{8 orbit lemma} to find suitable phase functions $\Theta_{j,\pm}$ which parametrize $\Gamma_\pm^j$. Define phase functions $\Theta_j^k$ by the formula
$$
\Theta_{j,\pm}^k(t,\tau, x,y) = \pm \tau(t - \Psi_j^k(x,y)),
$$
where $\Psi_j^k$ are given in Definition \ref{psi jk}. We then have,

\begin{lemm}\label{parameterization of cannonical relations}
The phase functions $\Theta_{j,\pm}^k(t,\tau,x,y)$ are smooth in an open neighborhood of the diagonal of the boundary and locally parametrize the canonical relations $\Gamma_\pm^j$. In particular, the fibers of both $\Gamma_+^j$ and $\Gamma_-^j$ lying over this neighborhood are unions of $8$ connected components, which we denote by $\Gamma_\pm^{j,k}$ corresponding to $1 \leq k \leq 8$ as in Definition \eqref{psi jk}.
\end{lemm}
\begin{proof}
	For any $x,y \in \Omega$ let
	\begin{align}\label{length functional}
	\begin{cases}
	L_{x,y}: \d \Omega^j \to \R_+\\
	L_{x,y}(q_1, q_2, \cdots q_j) = |x - q_1| + \left\{ \sum_{m = 2}^j |q_m - q_{m-1}| \right\} + |q_j - y|
	\end{cases}
	\end{align}
	denote the length functional. We first show that billiard trajectories from $x$ to $y$ are in one to one correspondence with critical points of \eqref{length functional} with respect to $q \in \d \Omega^j$. Let $g \in C^\infty(\R^2)$ be a defining function for $\d \Omega$ and consider $q$ as a variable in $\R^2 \times \cdots \times \R^2 = \R^{2j}$ rather than $\d \Omega^{j}$. If $q$ is a critical point of \eqref{length functional}, then as in the method of Lagrange multipliers, by setting $x = q_0$ and $y = q_{j+1}$, we find that for $1 \leq m \leq j$, there exists $\lambda_m \in \R$ such that
	$$
	\frac{\d L_{x,y}}{\d q_m} = \frac{q_m - q_{m-1}}{|q_m - q_{m-1}|} + \frac{q_m - q_{m+1}}{|q_m - q_{m+1}|} = \lambda_m \nabla_{q_m} g.
	$$
	Since $\nabla_{q_m} g \perp \d \Omega$, this implies that the two unit vectors in the formula for $\d_{q_m} L_{x,y}$ have opposite tangential components, which is precisely the condition giving elastic collision at the boundary (angle of incedince equals angle of reflection). Similarly, if this condition is satisfied, then $q$ is a critical point for \eqref{length functional}.\\
	\\
	We now consider the functions $\Psi_j^k$ in Definition \ref{psi jk}. We have
	\begin{align}\label{psi j}
	\Psi_j^k(x,y) = |x - q_1^k| + \left\{ \sum_{m = 2}^j |q_m^k - q_{m-1}^k| \right\} + |q_j^k - y|,
	\end{align}
	where $q_m^k(x,y)$ is the $m$th impact point on the boundary for the billiard trajectory corresponding to $\Psi_j^k$. As opposed to the $q_m$ in the length functional \eqref{length functional}, $q_m^k$ will in general have a nontrivial dependence on $x$ and $y$. Differentiating \eqref{psi j} in $x$, we obtain
	\begin{align}\label{psi j derivative}
	\frac{\d \Psi_j^k}{\d x_i} = \frac{x - q_1^k}{|x - q_1^k|} \cdot \frac{\d}{\d x_i} (x - q_1^k) + \left\{\sum_{m =2}^j \frac{q_m^k - q_{m-1}^k}{|q_m^k - q_{m-1}^k|} \cdot \frac{\d}{\d x_i}(q_m^k - q_{m-1}^k) \right\} + \frac{q_j^k - y}{|q_j^k - y|} \cdot \frac{\d}{\d x_i} (q_j^k - y).
	\end{align}
	Since for each $x, y \in \Omega$, the path defined by $(x,q^k,y)$ corresponds to a billiard trajectory, we see that all of the terms except the first telescope in \eqref{psi j derivative}. Hence,
	\begin{align}\label{psi j x derivative}
	d_x \Psi_j^k = \frac{x - q_1^k}{|x - q_1^k|}.
	\end{align}
	Similarly, differentiating \eqref{psi j} in $y$, we obtain
	\begin{align}\label{psi j y derivative}
	d_y \Psi_j^k = \frac{y - q_j^k}{|y - q_j^k|}.
	\end{align}
	Geometrically, these gradients are the incident and (reflected) outgoing unit directions of the billiard trajectories described in Theorem \ref{8 orbit lemma}.\\
	\\
	We now consider the maps
	\begin{align}\label{parametrize Gamma}
	\iota_{{\Theta_{j, \pm}^k}}: (t, \tau, x, y) \mapsto (t, \pm \tau, x, d_x \Theta_{j, \pm}^k,y,  - d_y\Theta_{j,\pm}^k) = (t, \pm \tau, x, \mp \tau d_x \Psi_j^k,y, \pm \tau d_y \Psi_j^k)
	\end{align}
	on the critical set $C_{\Theta_{j, \pm}^k} = \{ t - \Psi_j^k = 0 \}$. Inserting formulas \eqref{psi j x derivative} and \eqref{psi j y derivative} into \eqref{parametrize Gamma} and comparing with the canonical graphs
	$$
	\Gamma_\pm^j = \begin{cases}
	(t, \tau, g^{\pm t}(y,\eta), y, \eta): \tau = \pm |\eta| & j =0,\\
	(t, \tau, g^{\pm (t - T_\pm^j(y,\eta))}\widehat{\lambda_j(y,\eta)}, y, \eta): \tau = \pm |\eta| & j \in \Z \backslash \{0\}
	\end{cases}
	$$
	from Section \ref{Chazarain's Parametrix}, we see that $\iota_{{\Theta_{j, \pm}^k}}: C_{\Theta_{j, \pm}^k} \to \Gamma_\pm^j$ is a local diffeomorphism (this is why we chose orbits connecting $y$ to $x$ rather than $x$ to $y$ in Definition \ref{psi jk}).
\end{proof}
%For a fixed $j$, the multiple components $\Gamma_\pm^{j,k}$ of the canonical relation indexed by $k$ can be seen to intersect over the boundary $\d \Omega$. \red{This is due to the fact that there are only $2$ orbits connecting near diagonal boundary points in $j$ reflections, while there are $8$ such orbits connecting near diagonal interior points (cf. Theorem \ref{8 orbit lemma}).} 
As our parametrices for the propagators described in \cite{Ch76} are in fact modified by microlocal cutoffs supported away from the tangential rays in $S^*\d \Omega$, we may assume $\Gamma_{\pm}^{j,k}$ are smooth nonintersecting Lagrangian submanifolds over the interior. It would be interesting to study mapping properties of the operators $S_j^k$ as the canonical relations degenerate near the glancing set in future work. We now want to derive an explicit oscillatory integral representation \eqref{osc int} for $S(t)$ in a specific coordinate system adapted to $\Omega$. We first need to better understand the forwards and backwards symbols on $\Gamma$.
\begin{prop}\label{symbol prop}
Let $e_\pm$ denote the principal symbol of ${S}(t)$ on $\Gamma = \bigcup_{j \in \Z, \pm} \Gamma_\pm^j$. Modulo Maslov factors, we then have
\begin{align*}
e_\pm &= \frac{(-1)^j}{2 \tau i} |dt \wedge dy \wedge d\eta|^{1/2} \qquad \text{(Dirichlet)}\\
e_\pm &= \frac{1}{2 \tau i} |dt \wedge dy \wedge d\eta|^{1/2} \qquad \text{(Neumann/Robin)}
\end{align*}
where $|dt \wedge dy \wedge d\eta|^{1/2}$ is the canonical half density.
\end{prop}
\begin{proof}
For a proof in the Dirichlet/Neumann case, see \cite{HeZe12}. For the Robin boundary conditions, see \cite{Vig18}.
\end{proof}

\subsection{Proof of Theorem \ref{HRP}}
\noindent As Proposition \ref{symbol prop} gives $e_\pm = \frac{1}{2 \tau i} |dt \wedge dy \wedge d\eta|^{1/2}$ and we now know that the phase functions $\Theta_{j,\pm}^k(t,\tau,x,y) = \pm \tau(t - \Psi_j^k(x,y))$ parametrize connected components of $\Gamma_\pm^j$ lying over an open neighborhood of the diagonal of the boundary, we can compute the principal term in the asymptotic expansion of $a_{j,\pm}$ appearing in the expression \eqref{osc int}. Since there are $16$ phases for each $j$, corresponding to $1 \leq k \leq 8$ and $\pm \tau > 0$, \eqref{osc int} should actually be a sum of $16$ oscillatory integrals:
\begin{align}\label{Lagrangian distribution}
S_j(t,x,y) = \sum_{\pm} \sum_{k = 1}^8  \int_0^\infty e^{i \Theta_{j, \pm}^k(t,\tau, x, y)} a_{j,k,\pm}(\tau, x, y) d\tau.
\end{align}
Recalling formula \eqref{principal symbol of an FIO} for the principal symbol of a Fourier integral operator, we must compute the Gelfand-Leray form on the critical set $d_\tau \Theta_{j, \pm}^k = 0$. The Leray measure coming from the Gelfand-Leray form is coordinate invariant and it is ultimately more convenient to first introduce boundary normal coordinates, which we now describe.
\\
\\
Fix a point $p \in \d \Omega$. For each $q \in \d \Omega$ near $p$, we denote by $\gamma_q(\mu)$ the unit speed geodesic with initial condition conormal to $\d \Omega$ at $q$ and inward pointing. Now denote by $\phi$ the boundary coordinate which parametrizes $\d \Omega$ with respect to arclength, such that $p$ is given by $\phi = 0$. The coordinate $\phi$ can be extended smoothly inside $\Omega$ so that it is constant along $\gamma_q(\mu)$ for every fixed $q$ near $p$. For $\epsilon > 0$ sufficiently small, $(\mu,\phi)$ is then a smooth coordinate system in an $\epsilon$ neighborhood of $p \in \overline{\Omega}$. In these coordinates, the Euclidean metric is locally given by the warped product
$$
g_{\text{Eucl.}} = d\mu^2 + f(\mu,\phi) d\phi^2,
$$
where $f$ is a locally defined function which is smooth up to the boundary.
\begin{lemm}
For planar domains, the metric is given by
$$
d\mu^2 + (1- \kappa \mu)^2 d \phi^2,
$$
where $\kappa$ is the curvature of the boundary at $(0,\phi)$.
\end{lemm}
\begin{proof}
Let $|\d \Omega| = \ell$ and $x(\phi) = (x_1(\phi), x_2 (\phi))$ be a parametrization of $\d \Omega$ with respect to arclength. The exponential map $\exp: [0, \epsilon) \times \R / \ell \Z \to \Omega$ is given in boundary normal coordinates by $(\mu, \phi) \mapsto x+ \mu J \dot{x}$, where
$$
J = \begin{pmatrix}
0 & -1\\
1 &0
\end{pmatrix}
$$
is the $\pi/2$ rotation matrix and $\dot{x} = \d x / \d \phi$. We calculate that in these coordinates,
$$
\exp^* g_{\text{Euclid.}} = D \exp^T g_{\text{Euclid.}}  D \exp = \begin{pmatrix}
A & B\\
C & D
\end{pmatrix},
$$
where
\begin{align*}
A &= (\dot{x_1})^2 + (\dot{x_1})^2 = 1,\\
B &= C = -\dot{x}_2\dot{x}_1 + \mu \dot{x_2}\ddot{x}_2 +\dot{x}_1 \dot{x}_2 + \mu \dot{x}_1 \ddot{x}_1 = \mu(\dot{x}_1 \ddot{x}_1 + \dot{x}_2 \ddot{x}_2) =0,\\
D &= 1 - 2 \mu \dot{x}^T J \ddot{x} + \mu^2|\ddot{x}|^2 = (1 - \mu \kappa)^2,
\end{align*}
and $\kappa = |\ddot{x}|$.
\end{proof}
We maintain the notation $f(\mu, \phi) = (1 - \mu \kappa)^2$ throughout the rest of the paper. This coordinate system is convenient near the boundary because conformal multiples the vector fields $\d/\d \mu$ and $\d/\d \phi$ extend the orthogonal and and tangential gradients respectively in a tubular neighborhood of the boundary. As the canonical relations $\Gamma_{\pm}^j$ involve both $x$ and $y$ variables, we change $x$ variables to $(\mu, \phi)$ and $y$ variables to $(\nu, \theta)$ according to the procedure described above. In \cite{Vig18}, elliptical polar coordinates were instead used of boundary normal coordinates to acheive a similar decomposition into normal and tangential vector fields.

\subsection{Proof of Theorem \ref{HRP}}
Without loss of generality, we use boundary normal coordinates instead of Euclidean coordinates in the domain of $\Psi_j^k$ from here on. We now compute the Gelfand-Leray form in boundary normal coordinates.
\begin{lemm}\label{GL Form and Critical set}
The canonical relation of each operator in \eqref{Lagrangian distribution} is parametrized in boundary normal coordinates by
$$
\Gamma_\pm^{j,k} = \{(t,\tau, \mu, \phi, \xi, \nu, \theta, \eta) = (\Psi_j^k, \tau, \mu, \phi, \mp \tau d_{\mu,\phi} \Psi_j^k, \nu, \theta, \pm \tau d_{\nu, \theta} \Psi_j^k) \}.
$$
The Gelfand-Leray form on $C_{\Theta_{j, \pm}^k}$ is given by
\begin{align*}
dC_{\Theta_{j, \pm}^k} =\mp f(\mu, \phi) f(\nu, \theta) d\tau \wedge d\mu \wedge d\phi \wedge d\nu \wedge d\theta.
\end{align*}
\end{lemm}
\begin{proof}
The first claim follows directly from Lemma \ref{parameterization of cannonical relations} and the observation that
$$
d_\tau {\Theta_{j, \pm}^k} = \pm (t - \Psi_j^k) = 0
$$
on the critical set $C_{\Theta_{j, \pm}^k}$. From this, it is clear that $(\tau, \mu, \phi, \nu, \theta)$ form a smooth coordinate system on $C_{\Theta_{j, \pm}^k}$. The Gelfand-Leray form is uniquely defined on $C_{\Theta_{j, \pm}^k}$ by the condition
\begin{align}\label{GLF}
d \left(d_\tau{\Theta_{j, \pm}^k}\right) \wedge dC_{\Theta_{j, \pm}^k} = f(\mu, \phi) f(\nu, \theta) dt \wedge d \tau \wedge d(\mu, \phi) \wedge d(\nu, \theta),
\end{align}
where the righthand side of \eqref{GLF} coincides with the Euclidian volume form on $\R_t \times \R_\tau \times \Omega \times \Omega$. Hence, 
$$
dC_{\Theta_{j, \pm}^k} =\mp f(\mu, \phi) f(\nu, \theta) d\tau \wedge d\mu \wedge d\phi \wedge d\nu \wedge d\theta, 
$$
as claimed.
\end{proof}
%\begin{align*}
%dC_{\Theta_{j,\pm}^k} &= \mp d\tau \wedge dx \wedge dy\\
%&= \mp f(\mu, \phi) f(\nu, \theta) d\tau \wedge d\mu \wedge d\phi \wedge d\nu \wedge d\theta.
%\end{align*}
%The critical set is given by \marginpar{\red{Need to change $\pm$ signs?}}
%$$
%C_{\Theta_{j, \pm}^k} = \{(t,\tau, \mu, \phi, \xi, \nu, \theta, \eta) = (\Psi_j^k, \tau, \mu, \phi, - \tau d_{\mu,\phi} \Psi_j^k, \nu, \theta, \tau d_{\nu, \theta} \Psi_j^k), \}.
%$$
\noindent We now change variables and use Lemma \ref{GL Form and Critical set} to compute the principal symbol of \eqref{Lagrangian distribution}. Dropping the $j,k$ indices on $\Psi_j^k$ in place of differentiation, we have
\begin{align*}
dt &= \Psi_{\mu} d\mu + \Psi_{\phi}d\phi +\Psi_{\nu} d\nu + \Psi_{\theta}d\theta,\\
dy &= f(\nu, \theta) d\nu \wedge d\theta,\\ 
d{\eta_1} &= \Psi_{\nu} d\tau + \tau( \Psi_{\mu \nu} d\mu + \Psi_{\nu \phi} d\phi + \Psi_{\nu \nu} d\nu + \Psi_{\nu \theta} d\theta),\\
d\eta_2 &= \Psi_{\theta} d\tau +  \tau( \Psi_{ \theta\mu} d\mu + \Psi_{\theta \phi} d\phi + \Psi_{ \theta \nu} d\nu + \Psi_{ \theta \theta} d\theta  )
\end{align*}
on the canonical relation $\Gamma_{\pm}^{j,k}$. Wedging all these terms together, we see that
\begin{align}\label{Aj}
\begin{split}
dt \wedge dy \wedge d\eta = & \tau (f(\nu, \theta)  ( \Psi_{\mu} \Psi_{ \theta} \Psi_{\nu \phi} + \Psi_{\phi} \Psi_{\nu} \Psi_{ \theta \mu}\\
&- \Psi_{\mu}\Psi_{\nu} \Psi_{\theta \phi} - \Psi_{\phi} \Psi_{\theta} \Psi_{\nu \mu} )) d\tau \wedge d\mu \wedge d\phi \wedge d\nu \wedge d\theta.
\end{split}
\end{align}
\begin{def1}
Denote by $A_j^k(\mu,\phi, \nu, \theta)$ the functions
$$
\frac{1}{f(\mu, \phi)}  ( \Psi_{\mu} \Psi_{ \theta} \Psi_{\nu \phi} + \Psi_{\phi} \Psi_{\nu} \Psi_{ \theta \mu} - \Psi_{\mu}\Psi_{\nu} \Psi_{\theta \phi} - \Psi_{\phi} \Psi_{\theta} \Psi_{\nu \mu} ),
$$
where each $\Psi = \Psi_j^k$ depends implicitly on $j, k, \mu,\phi,\nu$ and $\theta$.
\end{def1}
\noindent On each of the canonical relations $\Gamma_\pm^{j,k}$ (cf. Lemma \ref{parameterization of cannonical relations}), equation \eqref{Aj} implies that
\begin{align}\label{symbol D}
e_\pm = (-1)^j \frac{|dt \wedge dy \wedge d\eta|^{1/2}}{2\tau i} = {\iota_{\Theta_{j, \pm}^k}}_* \left( (-1)^j  \frac{1}{2 \tau^{1/2} i}|A_j^k(\mu, \phi, \nu, \theta) dC_{\Theta_{j, \pm}^k}|^{1/2} \right)
\end{align}
for Dirichlet boundary conditions and
\begin{align}\label{symbol N}
e_\pm = \frac{|dt \wedge dy \wedge d\eta|^{1/2}}{2\tau i} =  {\iota_{\Theta_{j, \pm}^k}}_* \left(\frac{1}{2 \tau^{1/2} i}|A_j^k(\mu, \phi, \nu , \theta) dC_{\Theta_{j, \pm}^k}|^{1/2} \right)
\end{align}
for Neumann boundary conditions. As a result, we conclude the following description of the operators $S_j$.
\begin{theo}\label{parametrix with sum}
Microlocally near $\Gamma_\pm^j$, the following oscillatory integral is a parametrix for $S_j(t,x,y)$ with Dirichlet boundary conditions in an open neighborhood of $\Delta \d \Omega \subset \Omega \times \Omega$:
\begin{align*}
S_j(t)  = (-1)^j \sum_{\pm} \sum_{k = 1}^8 \int_{0}^\infty e^{\pm i\tau(t - \Psi_j^k)} \frac{1 }{2 i \tau^{1/2} } |A_j^k|^{1/2} d\tau + L.O.T.
\end{align*}
Here, L.O.T. denotes lower order Lagrangian distributions, using the convention \eqref{order of a lagrangian distribution}.
\end{theo}
\begin{def1}
We define the operators $S_{j,\pm}^k(t)$ appearing in Theorem \ref{parametrix with sum} by 
$$
S_{j,\pm}^k(t) =  (-1)^j \sum_{\pm} \sum_{k = 1}^8 \int_{0}^\infty e^{\pm i\tau(t - \Psi_j^k)} \frac{ 1 }{2 i \tau^{1/2}} |A_j^k|^{1/2} d\tau.
$$
\end{def1}
\noindent Together with Theorem \ref{8 orbit lemma} and the choice of $j_0$, $U_j$ in Remark \ref{Uj}, we can immediately derive Theorem \ref{HRP} from Theorem \ref{parametrix with sum}. To see this, note that
$$
E(t) = \frac{d}{dt} S(t).
$$
On a symbolic level, this implies that the symbol of $E(t)$ is $i \tau e_\pm$, with $e_\pm$ given by Proposition \ref{symbol prop}. This explains the order of $b_{j,k,\pm} \in S^{1/2}( U_j \times \R^1)$ appearing in Theorem \ref{HRP}. Each $(\pm)$ branch of the propagators $S_{j,\pm}^k$ should be multiplied by Maslov factors $e^{\pm i \pi \sigma_j/4}$. We will compute $\sigma_j = 1$ in Section \ref{Maslov factors}, which completes the proof of Theorem \ref{HRP}. In the next section, we will use this explicit oscillatory integral representation to compute the wave trace.

\section{Computing the wave trace}\label{Computing the wave trace}
\noindent In this section, we use the parametrix in Theorem \ref{HRP}, or equivalently Theorem \ref{parametrix with sum}, to compute an integral formula for the wave trace. Formally, the wave trace is the Fourier transform of the spectral measure $\sum_j \delta(\lambda - \lambda_j)$, where $\{\lambda_j^2\}$ are the Dirichlet (resp. Neumann or Robin) eigenvalues of $-\Delta$ on $\Omega$. This is a distribution of the form
\begin{align}\label{eigenvalue wave trace}
\sum_j e^{i t \lambda_j},
\end{align}
which can be seen to be weakly convergent by Weyl's law on the asymptotic distribution of Laplace eigenvalues (see \cite{Irvii16}, \cite{Zayed2004}). The connection between this distribution and the wave equation lies in the fact that \eqref{eigenvalue wave trace} is actually the trace of $e^{i t \sqrt{- \Delta}}$, the propagator associated to the half wave operator $(\d_t - i \sqrt{-\Delta})$. Since such a unitary operator is not trace class, we mean that for any Schwartz function $\phi$, the regularized operators
$$
\int_{-\infty}^\infty \phi(t) e^{it \sqrt{-\Delta}}dt
$$
are of trace class and have trace
$$
\sum_j \int_{-\infty}^\infty \phi(t) e^{it \lambda_j}dt.
$$
The same holds for the even and odd wave operators
\begin{align}
E(t) = \cos t\sqrt{- \Delta}, \qquad S(t) = \frac{\sin t \sqrt{-\Delta}}{\sqrt{-\Delta}},
\end{align}
and we consider these as they appear more naturally in Chazarain's parametrix (cf. Section \ref{Chazarain's Parametrix}) and $E(t)$ solves the initial boundary value problem \eqref{wave equation}.
\\
\\
Recall that $\chi_1(t)$ is a cutoff function near the lengths $[t_j,T_j]$ of geodesic loops and $\Omega$ is assumed to satisfy the noncoincidence condition \eqref{NCC}. Modulo Maslov factors and a smooth error term, we have
\begin{align}\label{localized wave trace}
\chi_1(t) \text{Tr} E &= \int_{\Omega} E_j(t,x,x) dx,
\end{align}
where $E_j = \frac{\d}{\d t} S_j$ and $S_j$ is the $j$ bounce wave appearing in Chazarain's parametrix \eqref{Chazarain sum}.

\subsection{Reduction to boundary}
As our parametrix in Theorem \ref{parametrix with sum} is only valid near the boundary, we want to write the wave trace as an integral over the boundary. There are several ways to do this, one of which involves using Hadamard type variational formulas from \cite{HeZe12} and \cite{Vig18} to integrate the radially differentiated wave trace, which via an integration by parts puts the integral \eqref{localized wave trace} on the boundary. Here we use a different technique, suggested by the referree. Let us first establish some notation. Fixing an arbitrary point $O \in \Omega$ to be the origin, we denote by $X(q)$ the position vector of a point $q \in \d \Omega$ relative to $O$ and $N(q)$ the outward unit normal at $q$. Let $\nabla^\perp$ and $\nabla^T$ denote the unit normal and tangential gradients with respect to $\d \Omega$ and recall the notation $D_t = -i \d_t$.
\begin{lemm}
	For each Dirichlet eigenfunction $u_j$ of eigenvalue $\lambda_j^2$ in \ref{inverse problem}, we have
	\begin{align}
	\lambda_j^2 = \frac{1}{2} \int_{\d \Omega} \langle X, N \rangle |\nabla^\perp u_j(q)|^2 dq,
	\end{align}
	and hence
	\begin{align}
	D_t^2 \text{Tr} S(t) = \frac{1}{2} \int_{\d \Omega} \langle X, N \rangle \nabla_1^\perp \nabla_2^\perp S(t,q,q) dq.
	\end{align}
\end{lemm}
\begin{proof}
	For the first formula, we use the commutator identity $[\Delta, r \d_r] = 2 \Delta$, where $r = |X|$ and $\d_r$ is the radial vector field, to see that
	\begin{align*}
	2 \lambda_j^2 &= - 2 \int_{\Omega} \Delta u_j \overline{u_j} dx = \int_{\Omega} [- \Delta, r\d_r] u_j \overline{u_j} dx = \int_{\d \Omega} r \frac{\d u_j}{\d r} \frac{\d \overline{ u_j}}{\d N} dq.
	\end{align*}
Here we have integrated by parts and used Dirichlet boundary conditions. Decomposing $r\d_r$ into the normal and tangential vector fields $N, \d_q$ and again using the boundary condition, we obtain
\begin{align*}
r \frac{\d u}{\d r} = \langle X, N \rangle \nabla^\perp u_j.
\end{align*}
The trace formula follows by writing
\begin{align*}
D_t^2 \text{Tr} S(t) &= \sum_j \lambda_j^2 \frac{\sin \lambda_j t}{\lambda_j} = \frac{1}{2} \int_{\d \Omega} \langle X, N \rangle \sum_j \frac{\sin t \lambda_j}{\lambda_j} |\nabla^\perp u_j|^2 dq\\
&=  \frac{1}{2} \int_{\d \Omega} \langle X, N \rangle \nabla_1^\perp \nabla_2^\perp S(t,q, q) dq.
\end{align*}	
\end{proof}

As we are localizing the wave trace near lengths of geodesic loops with $j$ reflections, it turns out we only need to consider a select few of the operators from Theorem \ref{parametrix with sum} in the trace formula. From here on, we only consider Dirichlet boundary conditions. Similar formulas exist in the Neumann and Robin cases. We also use the notation $D_t^{-1} \in \Psi_{\text{ell}}^{-1}(\R)$ to denote an elliptic parametrix (Fourier multiplier $\tau^{-1}$) for $D_t$.
\begin{lemm}\label{Sj lemma}
Modulo Maslov factors and lower order distributions, the even Dirichlet wave trace localized near $[t_j,T_j]$ is given by
\begin{align*}
\text{Tr} E(t) = \sum_\pm \int_{\d \Omega} \frac{\langle X, N \rangle}{2} i D_t^{-1} \nabla_1^\perp \nabla_2^\perp (&S_{j-1,\pm}^1  +S_{j-1,\pm}^5 + S_{j,\pm}^2 + S_{j,\pm}^3 \\
& + S_{j,\pm}^6 + S_{j,\pm}^7 + S_{j + 1,\pm}^4 +  S_{j +1,\pm}^8 )(t,q,q) \,dq.
\end{align*}
%in the Dirichlet case and
%\begin{align*}
%\sum_\pm \int_{\d \Omega} \frac{t}{2} (-\nabla_1^T \nabla_2^T - \Delta_2)(&S_{j-1,\pm}^1  +S_{j-1,\pm}^5 + S_{j,\pm}^2 + S_{j,\pm}^3 \\
%& + S_{j,\pm}^6 + S_{j,\pm}^7 + S_{j + 1,\pm}^4 +  S_{j +1,\pm}^8 )(t,q,q) \dot{\rho}\,dq
%\end{align*}
%in the Neumann case.
\end{lemm}
\noindent For completeness, we repeat the proof derived in \cite{Vig18} as it contains substantial geometric insight.
%and motivates Definition \ref{psi j coincide on boundary}. \marginpar{\red{Editing stopped here at end of uber ride.}}
\begin{proof}
First note that $E(t) = \d_t S(t)$, so that
\begin{align*}
\text{Tr} E(t) = i D_t^{-1} D_t^2 \text{Tr} S(t) = \int_{\d \Omega} \frac{\langle X, N \rangle}{2} i D_t^{-1} \nabla_1^\perp \nabla_2^\perp S(t,q,q)dq.
\end{align*}
For the localized wave trace, we only need to consider orbits which contribute to the singularities in $[t_j,T_j]$. Recall that for positive time, Theorem \ref{8 orbit lemma} gives $8$ orbits connecting $x$ to $y$ in $j$ reflections and approximately one rotation. These orbits coalesce into geodesic loops as $(x, y) \to \Delta \d \Omega$. However, as the orbits coalesce within various configurations, not all of the limiting orbits will have $j$ reflections. As $\Omega$ satisfies the noncoincidence condition \eqref{NCC}, only the limiting geodesic loops of rotation number $1/j$ will contribute to the wave trace near $[t_j,T_j]$. See Figure \ref{configurations} for visualizing the geometric arguments which follow. We will say that a sequence of billiard orbits $\gamma_n$ converges geometrically to another orbit $\gamma_0$ if all impact points on the boundary converge to those of $\gamma_0$ in order. Note that the limiting orbit may have a different number of reflections if certain impact points coalesce in the limit. To demonstrate geometric convergence in this sense, it suffices to show convergence of any two consecutive impact points to two distinct points in the limiting orbit. All other points of reflection on the boundary are then smoothly and implicitly determined by the limiting link formed between these two.
\\
\\
As $(x, y) \to \Delta \d \Omega$, the two orbits of $j$ reflections in $TT$ configuration $(k = 1,5)$ converge geometrically to a loop of $j+1$ reflections. The additional vertex appears at the boundary point where $x$ and $y$ coalesce. Similarly, the two $NN$ orbits of $j$ reflections $(k = 4,8)$ can be seen to converge to loops of $j-1$ reflections. In this case, the first and last points of reflection at the boundary converge to a single impact point. The four orbits of $j$ reflections in $TN$ $(k = 2,6)$ and $NT$ $(k = 3,7)$ configurations preserve exactly $j$ reflections in the limit. Hence, when $x, y \in \text{int} \Omega$ converge to $\Delta \d \Omega$, only $4$ of the $8$ orbits contribute to geodesic loops of $j$ reflections. However, in the same limit, two additional $TT$ orbits of $j-1$ reflections converge to a loop of $(j - 1) + 1 = j$ reflections. Similarly, two $NN$ orbits of $j+1$ reflections converge to a loop of $(j+1) -1 = j$ reflections. Any other orbit from $x$ to $y$ with strictly less than $j-1$ or strictly more than $j+1$ reflections at the boundary cannot converge to a loop of $j$ reflections. As we have localized the wave trace near the isolated set of lengths $[t_j,T_j]$ and $\Omega$ satisfies the noncoincidence condition \eqref{NCC}, only the $4 + 2 + 2 = 8$ orbits which converge geometrically to a loop of exactly $j$ reflections will contribute to singularities here. All additional orbits contribute smooth errors to the wave trace in a small neighborhood of $[t_j,T_j]$.
\\
\\
It should also be clarified that although the parametrices $S_j(t,x,y)$ are constructed in the interior, we can in fact extend them continuously to the diagonal of the boundary and this extension coincides with that of the true propagator $S(t,x,x)$ ($x \in \d \Omega$) modulo lower order terms. Both propagators agree up to lower order Lagrangian distributions in the interior, microlocally near the canonical relations $\Gamma_\pm^j$. The explicit oscillatory integral representation for each $S_j(t,x,y)$ in fact shows that they extend continuously up to the boundary and its diagonal, since the functions $\Psi_j^k(x,y)$ do. The true wave kernel $S(t,x,y)$ also extends continuously up to the boundary as a family of distributions. To see this, note that
$$
S(t,x,y) = \sum_j \frac{\sin ( t \lambda_j)}{\lambda_j} \psi_j(x) \overline{\psi_j(y)},
$$
where $(\psi_j)_{j =1}^\infty$ is an $L^2$ orthonormal basis of Dirichlet or Neumann eignenfunctions corresponding to eigenvalues $(\lambda_j^2)_{j = 1}^\infty$. Multiplying by a test function $\phi(t)$ and integrating by parts $4k$ times, we see that
\begin{align}\label{wave kernel}
\int_{-\infty}^\infty S(t,x,y) \phi(t) dt = \int_{-\infty}^\infty \sum_j \frac{\sin (t \lambda_j)}{\lambda_j^{4k+1}} \psi_j(x) \overline{\psi_j(y)} \d_t^{4k}  \phi(t) dt. 
\end{align}
Combining Weyl's law on the asymptotic growth of $\lambda_j$ (see \cite{Zayed2004}, \cite{Irvii16}) and  $L^\infty$ bounds for eigenfunctions on manifolds with boundary (see \cite{Grieser}), we see that the integrand in \eqref{wave kernel} can be made absolutely convergent for $k$ sufficiently large. An application of the dominated convergence theorem then shows that \eqref{wave kernel} is actually smooth in $x,y$, so $S(t,x,y)$ has a smooth extension to the diagonal of the boundary as a distribution in $t$. In particular, both distributions agree up to lower order terms microlocally near the fibers of $\Gamma_\pm^j$ lying over diagonal of the boundary.
\end{proof}

\begin{def1}\label{psi j coincide on boundary}
As shown in the proof of Lemma \ref{Sj lemma} above, for each $j$, there exist $8$ limiting trajectories which converge geometrically to geodesic loops of exactly $j$ reflections. We denote the set of these trajectories by $\mathcal{G}_j(x,y)$ and say that $\gamma_{m,k} \in \mathcal{G}_j$ if $\gamma_{m,k}$ makes $m = j-1, j$ or $j+1$ reflections at the boundary and corresponds to the length functional $\Psi_m^k$. Since there is a unique geodesic loop of rotation number $1/j$ at each boundary point, the length functionals $\Psi_j^2,\Psi_j^3 , \Psi_j^6 , \Psi_j^7 , \Psi_{j + 1}^4 , \Psi_{j +1}^8 , \Psi_{j-1}^1$ and $\Psi_{j-1}^5$ corresponding to orbits in $\mathcal{G}_j$ coincide for $x =y \in \d \Omega$ on the diagonal. We denote their common value by $\Psi_j(x,x)$, which is the $j$-loop function appearing in Definition \ref{jloop}.
\end{def1}

\subsection{Boundary calculations}
\noindent As we obtained a rather explicit formula for $S_j(t,x,y)$ in Theorem \ref{parametrix with sum}, it now remains to differentiate the kernels $S_m^k(t,x,y)$ and substitute them into Lemma \ref{Sj lemma}. Using our oscillatory integral representation for $S_m^k(t,x,y)$ in Theorem \ref{parametrix with sum}, we find that microlocally near $\Gamma_\pm^{m,k}$ and $t \in [t_j,T_j]$, modulo lower order terms in an open neighborhood of the diagonal of the boundary, we have
\begin{align}\label{parametrix with 8 Dirichlet}
i D_t^{-1} \nabla_1^\perp \nabla_2^\perp S_{m,\pm}^k = (-1)^{m+1} \sum_\pm \int_{0}^\infty e^{\pm i\tau(t - \Psi_m^k)} \frac{\pm (\nabla_1^\perp \Psi_m^k)(\nabla_2^\perp \Psi_m^k) \tau^{3/2}}{2\tau } |A_m^k|^{1/2} d\tau
\end{align}
for Dirichlet boundary conditions. % and \marginpar{\red{There's a problem with the Neumann formula!!}}
%\red{\begin{align}\label{parametrix with 8 Neumann}
%(- \nabla_1^T \nabla_2^T - \Delta_2) S_m^k =  -\int_{-\infty}^\infty e^{i \tau(t - \Psi_m^k(x,y))} \frac{ (-i\tau)^2|\nabla_2^\perp \Psi_m^k|^2 {|A_m^k|^{1/2}} \text{sgn}(\tau)}{2 |\tau|^{1/2} i} \dot{\rho}\,d\tau
%\end{align}}
%for Neumann boundary conditions.
We have only written the terms coming from $\nabla_{1,2}^\perp$ falling on the exponential in equation \eqref{parametrix with 8 Dirichlet}, as all other terms don't contribute positive powers of $\tau$ and can be regarded as lower order in the singularity expansion. The operators $\nabla_{1,2}^\perp$ in the integrand of \eqref{parametrix with 8 Dirichlet} are conformal multiples of the vector fields $\frac{\d}{\d \mu}$ and $\frac{\d}{\d \nu}$ coming from boundary normal coordinates.
\\
\\
As Lemma \ref{Sj lemma} tells us that the wave trace is given by integrating the normally differentiated sine kernels over the diagonal of the boundary, we want to understand the restriction of \eqref{parametrix with 8 Dirichlet} to the boundary. We already noted that the $j$-loop function is well defined for $j$ sufficiently large.
%\marginpar{\blue{not true... one CCW and one CW}} The $8$ functions $\Psi_m^k$ indexed by $k$ and corresponding to $\gamma_{m,k} \in \mathcal{G}_j$ for a fixed $j$ also coincide, as mentioned in the remarks following Theorem \ref{8 orbit lemma}. This is true on $\Delta \Omega$ in particular. Recall that $\Delta: \d \Omega \to \d \Omega \times \d \Omega$ is the diagonal embedding and should not be confused with the Laplacian. In that case, we drop the subscript $k$ and denote the length functionals' common value by $\Psi_j$, which is the length of the unique geodesic loop of rotation number $1/j$ based at $q$. By geodesic loop, recall that we mean that the corresponding billiard orbit returns to its initial position $q$ after finite time but with a possibly different angle of incidence. We denote this unique geodesic loop of rotation number $1/j$ based at a boundary point $q$ by $\gamma_j(q)$. When $x = y = q \in \d \Omega$, $\Psi_j(q,q)$ gives precisely the length $L_j$ of an orbit with $j$ reflections both emanating from and terminating at $q$. This is sometimes called the $j$ loop function.
The differentiated kernels in equation \eqref{parametrix with 8 Dirichlet} also have factors of $A_j^k$ and $\nabla_{1,2}^\perp \Psi_j^k$ in the integrand. We now discuss how to extend these derivatives of $\Psi_j^k$ to the diagonal of the boundary in a similar manner. In the proof of Lemma \ref{parameterization of cannonical relations}, the $x$ and $y$ gradients of the functions $\Psi_j^k$ are shown to be:
\begin{align}\label{x y gradients}
d_x \Psi_j^k = \frac{x - q_1^k}{|x - q_1^k|}, \quad d_y \Psi_j^k = \frac{y - q_j^k}{|y - q_j^k|}.
\end{align}
Geometrically, these are the incident and reflected outgoing unit directions of the corresponding billiard trajectories at $x$ and $y$. The expression $\nabla_1^\perp \Psi_j^k$ in \eqref{parametrix with 8 Dirichlet} can easily be seen to be $\pm \sin \omega_{j,1}^k$, where $\omega_{j,1}^k$ is the angle made between the initial link of the billiard trajectory and the oriented tangent line to the level set of the distance function on which $x$ lies (as in the folliation used in the proof of Theorem \ref{8 orbit lemma}). We use the positively oriented tangent line for $1 \leq k \leq 4$ and the negatively oriented tangent line for $5 \leq k \leq 8$. The sign $\pm$ depends on the $TT, TN, NT$ or $NN$ configuration of the corresponding orbit. Similarly, $\nabla_2^\perp \Psi_j^k = \pm \sin \omega_{j,2}^k$, where $\omega_{j,2}^k$ is the angle made between the final link of the billiard trajectory and the positively $(1 \leq k \leq 4)$ or negatively $(5 \leq k \leq 8)$ oriented tangent line to the distance curve on which $y$ lies. As $x,y \to  \Delta \d \Omega$, the absolute value of the angles associated to trajectories in the $\mathcal{G}_j$ converge to the initial and final angles of reflection of the unique limiting geodesic loop. We are careful to point out that only the absolute values of the angles converge, since for example, the final angles of reflection at $y$ associated to orbits in $TN$ and $NN$ configurations actually converge to the \textit{negative} of the final angle in the limiting trajectory. All limiting loops are automatically in $TT$ configuration.
%\red{This phenomenon will be important for avoiding cancellation in the equation \eqref{D1} of Section \ref{Integrating the principle symbol}, where we add together all wave kernels and integrate them over the diagonal to obtain a trace formula.}
\begin{lemm}\label{Aj w derivatives}
On the diagonal of the boundary, the factors $A_{j-1}^1, A_{j-1}^5 A_j^2, A_j^3, A_j^6, A_j^7, A_{j+1}^4$ and $A_{j+1}^8$ corresponding to orbits in $\mathcal{G}_j$ coincide up to a sign. We denote their common (absolute) value by $|A_j|$, which at a point $(0,\phi, 0, \phi)$ in boundary normal coordinates satisfies
\begin{align*}
|A_j(0, \phi, 0, \phi)| =\frac{1}{\sin \omega_{j,2}}  \left| \frac{\d \omega_{j,1}}{\d \theta} \right|.
\end{align*}
Here, $\omega_{j,1}$ and $\omega_{j,2}$ are the initial and final angles of incidence respectively made by the unique counterclockwise parametrized geodesic loop of rotation number $1/j$ based at $(0, \phi, 0, \phi) \in \Delta \d \Omega$.
\end{lemm}
\begin{proof}
Let $(m,k)$ denote an admissable pair of indices corresponding to an orbit $\gamma_{m,k} \in \mathcal{G}_j$. Recall the notation in the proof of Lemma \ref{parameterization of cannonical relations} (cf. \cite{Vig18}), where we described a billiard trajectory by the point $(x,q,y) \in \Omega \times \d \Omega^m \times \Omega$. Let us first assume $1 \leq k \leq 4$. If $x \in \Omega$, we denote the positive angle between $q_1 - x$ and the positively oriented tangent line to the leaf of the folliation by distance curves on which $x$ lies by $\omega_{j,1}^k$ (cf. Section \ref{Proof of 8 orbit lemma}). Similarly, if $y \in \Omega$, let us also denote the positive angle between $y- q_j$ and the positively oriented tangent line to the distance curve on which $y$ lies by $\omega_{j,2}^k$. Note that $\omega_{j,1}^k$ and $\omega_{j,2}^k$ depend on $x, y, j$ and $k$. Using the warped product structure of boundary normal coordinates $(\mu, \phi, \nu, \theta)$, we have
\begin{align}\label{BNCG}
\begin{split}
\nabla_x^T &= {(f)}^{-1/2} \frac{\d}{\d\phi}, \qquad
\nabla_y^T = {(f)}^{-1/2} \frac{\d}{\d\theta},\\
\nabla_x^\perp &=  \frac{\d}{\d\mu}, \qquad \qquad \,\,\,\,\,\,\,\,
\nabla_y^\perp =  \frac{\d}{\d\nu}.
\end{split}
\end{align}
Equation \eqref{x y gradients}, which was derived from the proof of Lemma \ref{parameterization of cannonical relations} in \cite{Vig18}, then shows that
\begin{align}\label{perp grads}
\begin{split}
\nabla_x^\perp \Psi_j^k(x,y) = \begin{cases}
- \sin \omega_{j,1}^k, & k = 1,2\\
\sin \omega_{j,1}^k, & k = 3,4,
\end{cases} \qquad
\nabla_y^\perp \Psi_j^k(x,y) = \begin{cases}
- \sin \omega_{j,2}^k, & k = 1,3\\
\sin \omega_{j,2}^k, & k = 2,4,
\end{cases}
\end{split}
\end{align}
and
\begin{align}\label{tangential grads}
\nabla_x^T \Psi_j^k(x,y) = -\cos \omega_{j,1}^k, \quad (1 \leq k \leq 4) \qquad \qquad \nabla_y^T \Psi_j^k(x,y) = \cos \omega_{j,2}^k, \quad (1 \leq k \leq 4).
\end{align}
Combining \eqref{BNCG}, \eqref{perp grads}, and \eqref{tangential grads}, we obtain
\begin{align}\label{first derivatives}
\begin{split}
\nabla_{(\mu,\phi)} \Psi_j^k(\mu, \phi, \nu, \theta) &= \begin{cases}
(-\sin \omega_{j,1}^k, -f^{1/2}(\mu, \phi) \cos \omega_{j,1}^k ), & k = 1\\
(-\sin \omega_{j,1}^k, -f^{1/2}(\mu, \phi) \cos \omega_{j,1}^k ), & k = 2\\
(\sin \omega_{j,1}^k, -f^{1/2}(\mu, \phi) \cos \omega_{j,1}^k ), & k = 3\\
(\sin \omega_{j,1}^k, -f^{1/2}(\mu, \phi) \cos \omega_{j,1}^k ), & k = 4
\end{cases}\\
\nabla_{(\nu,\theta)} \Psi_j^k(\mu, \phi, \nu, \theta) &= \begin{cases}
(-\sin \omega_{j,2}^k, f^{1/2}(\nu, \theta) \cos \omega_{j,2}^k ),& k = 1\\
(\sin \omega_{j,2}^k, f^{1/2}(\nu, \theta) \cos \omega_{j,2}^k ), & k = 2\\
(-\sin \omega_{j,2}^k, f^{1/2}(\nu, \theta) \cos \omega_{j,2}^k ),& k = 3\\
(\sin \omega_{j,2}^k, f^{1/2}(\nu, \theta) \cos \omega_{j,2}^k ), & k = 4.
\end{cases}
\end{split}
\end{align}
%asdf
%\begin{align}
%\begin{split}
%\Psi_\mu &= \pm  \sin \omega_{j,1}^k, \qquad \qquad \,\,\,\,\,\,\, \Psi_\nu = \pm \sin \omega_{j,2}^k,\\
%\Psi_\phi &= - f(\mu, \phi) \cos \omega_{j,1}^k, \qquad \Psi_\theta = f(\nu, \theta) \cos \omega_{j,2}^k,
%\end{split}
%\end{align}
%where $\pm$ in the equations for $\Psi_\mu$ and $\Psi_\nu$ are dependent on the configuration of the orbit. They are $-$ if the corresponding initial ($\mu,\phi$) or final ($\nu, \theta$) link is a $T$ link and $+$ if it is an $N$ link.
Using \eqref{first derivatives} to calculate the $(\mu,\phi,\nu,\theta)$ Hessian of $\Psi_j^k$, we find
%\begin{align}\label{second derivatives}
%\begin{split}
%\Psi_{ \theta \phi} &= f(\mu,\phi) \sin \omega_{j,1}^k \frac{\d\omega_{j,1}^k}{\d \theta}, \qquad 
%\Psi_{ \nu \phi} = f(\mu,\phi) \sin \omega_{j,1}^k \frac{\d\omega_{j,1}^k}{\d \nu},\\
%\Psi_{\theta \mu} &= \pm \cos \omega_{j,1}^k \frac{\d\omega_{j,1}^k}{\d \theta}, \qquad \qquad \, \Psi_{\nu \mu} = \pm  \cos \omega_{j,1}^k \frac{\d\omega_{j,1}^k}{\d \nu}.
%\end{split}
%\end{align}
\begin{align}\label{second derivatives}
\begin{split}
\frac{\d^2 \Psi_j^k}{\d \theta \d \phi}(\mu, \phi,\nu, \theta) &=  f^{1/2}(\mu,\phi) \sin \omega_{j,1}^k \frac{\d \omega_{j,1}^k }{\d \theta}, \qquad (1 \leq k \leq 4)\\
\frac{\d^2 \Psi_j^k}{\d \nu \d \phi}(\mu, \phi,\nu, \theta) &= f^{1/2}(\mu,\phi) \sin \omega_{j,1}^k \frac{\d \omega_{j,1}^k }{\d \nu}, \qquad (1 \leq k \leq 4)\\
\frac{\d^2 \Psi_j^k}{\d \theta \d \mu}(\mu, \phi,\nu, \theta) &= \begin{cases}
- \cos \omega_{j,1}^k \frac{\d \omega_{j,1}^k }{\d \theta}, & k = 1,2\\
\cos \omega_{j,1}^k \frac{\d \omega_{j,1}^k }{\d \theta}, & k = 3,4
\end{cases} \\
\frac{\d^2 \Psi_j^k}{\d \nu \d \mu}(\mu, \phi,\nu, \theta) &= \begin{cases}
- \cos \omega_{j,1}^k \frac{\d \omega_{j,1}^k }{\d \nu}, & k = 1,2\\
\cos \omega_{j,1}^k \frac{\d \omega_{j,1}^k }{\d \nu}, & k = 3,4.
\end{cases}
\end{split}
\end{align}
Inserting \eqref{second derivatives} into the expression \eqref{Aj} for $A_j^k$ in all possible configurations $(1 \leq k \leq 4)$, we find that on the boundary,
\begin{align}\label{new Aj w derivatives}
A_m^k(0, \phi, 0, \theta) =
\begin{cases}
-\cos \omega_{j,2}^k \frac{\d\omega_{j,1}^k}{\d \nu} - \sin \omega_{j,2}^k \frac{\d \omega_{j,1}^k}{\d \theta}, & k = 1\\
- \cos \omega_{j,2}^k \frac{\d\omega_{j,1}^k}{\d \nu} + \sin \omega_{j,2}^k \frac{\d \omega_{j,1}^k}{\d \theta},  & k = 2\\
+ \cos \omega_{j,2}^k \frac{\d \omega_{j,1}^k}{\d \nu} + \sin \omega_{j,2}^k \frac{\d \omega_{j,1}^k}{\d \theta}, & k = 3\\
+ \cos \omega_{j,2}^k \frac{\d\omega_{j,1}^k}{\d \nu} - \sin \omega_{j,2}^k \frac{\d \omega_{j,1}^k}{\d \theta}, & k = 4.
\end{cases}
\end{align}
Before evaluating this expression on the diagonal of the boundary, we differentiate $\omega_{j,1}^k$ in the direction $L = (y-q)/|y-q|$ of the last link to see that
\begin{align*}
0 &= \nabla_L \omega_{j,1}^k\\
&= \begin{cases}
\cos \omega_{j,2}^k \nabla_y^T \omega_{j,1}^k - \sin \omega_{j,2}^k \nabla_y^\perp \omega_{j,1}^k, & k = 1,3\\
\cos \omega_{j,2}^k \nabla_y^T \omega_{j,1}^k + \sin \omega_{j,2}^k \nabla_y^\perp \omega_{j,1}^k, & k = 2,4
\end{cases}\\
&= \begin{cases}
\frac{1}{f^{1/2}} \cos \omega_{j,2}^k  \frac{\d \omega_{j,1}^k}{\d \theta} - \sin \omega_{j,2}^k  \frac{\d \omega_{j,1}^k}{\d \nu}, & k = 1,3\\
\frac{1}{f^{1/2}} \cos \omega_{j,2}^k  \frac{\d \omega_{j,1}^k}{\d \theta} + \sin \omega_{j,2}^k  \frac{\d \omega_{j,1}^k}{\d \nu}, & k = 2,4.
\end{cases}
\end{align*}
%Here, $\pm$ correspond to whether the last link is a $T$ link $(-)$ or an $N$ link $(+)$.
This implies that
\begin{align}\label{d omega d nu}
\sqrt{f(\nu, \theta)} \frac{\d \omega_{j,1}^k}{\d \nu} = \begin{cases}
+\cot \omega_{j,2}^k \frac{\d \omega_{j,1}^k}{\d \theta}, & k =1,3\\
- \cot \omega_{j,2}^k \frac{\d \omega_{j,1}^k}{\d \theta}, & k = 2,4.
\end{cases}
\end{align}
Inserting the formula \eqref{d omega d nu} into \eqref{new Aj w derivatives} and evaluating on $\d \Omega$, we find that
\begin{align*}
|A_m^k(0, \phi, 0, \theta)| =  \left(\frac{\cos^2 \omega_{j,2}^k}{\sin \omega_{j,2}^k} + \frac{\sin^2 \omega_{j,2}^k}{\sin \omega_{j,2}^k}  \right) \left| \frac{\d \omega_{j,1}^k}{\d \theta} \right|  = \frac{1}{\sin \omega_{j,2}^k} \left| \frac{\d \omega_{j,1}^k}{\d \theta} \right|.
\end{align*}
Note that for $1 \leq k \leq 4$ and $x, y$ on the boundary, both $\omega_{j,1}^k$ and $\omega_{j,2}^k$ are independent of $k$. They coincide with the inital and final angles of the unique orbit connecting $x$ to $y$ in $j$ reflections and approximately one counterclockwise rotation. As the orbits corresponding to $5 \leq k \leq 8$ can be viewed as counterclockwise orbits in the reflected domain, we instead defined $\omega_{j,1}^k$ and $\omega_{j,2}^k$ to be the angles made between the initial and final links and the \textit{negatively} oriented tangent lines to the distance curves on which $x$ and $y$ respectively lie. With this convention, we see by symmetry that the roles of $\omega_{j,1}^k$ and $\omega_{j,2}^k$ are interchanged:
\begin{align}\label{first derivatives k58}
\begin{split}
\nabla_{(\mu,\phi)} \Psi_j^k(\mu, \phi, \nu, \theta) &= \begin{cases}
(-\sin \omega_{j,1}^k, f^{1/2}(\mu, \phi) \cos \omega_{j,1}^k ), & k = 5\\
(-\sin \omega_{j,1}^k, f^{1/2}(\mu, \phi) \cos \omega_{j,1}^k ), & k = 6\\
(\sin \omega_{j,1}^k, f^{1/2}(\mu, \phi) \cos \omega_{j,1}^k ), & k = 7\\
(\sin \omega_{j,1}^k, f^{1/2}(\mu, \phi) \cos \omega_{j,1}^k ), & k = 8
\end{cases}\\
\nabla_{(\nu,\theta)} \Psi_j^k(\mu, \phi, \nu, \theta) &= \begin{cases}
(-\sin \omega_{j,2}^k, -f^{1/2}(\nu, \theta) \cos \omega_{j,2}^k ),& k = 5\\
(\sin \omega_{j,2}^k, -f^{1/2}(\nu, \theta) \cos \omega_{j,2}^k ), & k = 6\\
(-\sin \omega_{j,2}^k, -f^{1/2}(\nu, \theta) \cos \omega_{j,2}^k ),& k = 7\\
(\sin \omega_{j,2}^k, -f^{1/2}(\nu, \theta) \cos \omega_{j,2}^k ), & k = 8.
\end{cases}
\end{split}
\end{align}
Replacing \eqref{first derivatives} by \eqref{first derivatives k58}, parallel computations to those above then show that for $5 \leq k \leq 8$,
\begin{align*}
|A_m^k(0,\phi,0,\phi)| = \frac{1}{\sin \omega_{j,2}^k} \Big|\frac{\d \omega_{j,1}}{\d \theta}\Big|.
\end{align*}
Note that for $5 \leq k \leq 8$, $\omega_{j,1}^{k-4} - \omega_{j,2}^k \to 0$ as $x, y \to \d \Omega$. Similarly, $\omega_{j,2}^{k-4} - \omega_{j,1}^k \to 0$ as $x, y\to \d \Omega$. Hence,
\begin{align}\label{k58}
|A_m^k(0,\phi,0,\phi)| = \frac{1}{\sin \omega_{j,1}} \Big| \frac{\d \omega_{j,2} }{ \d \phi} \Big|,
\end{align}
for $5 \leq k \leq 8$. Note that the righthand side of \eqref{k58} involves angles at boundary points and does not depend on $k$. Also observe that
\begin{align}\label{Clairaut}
\begin{split}
\frac{\d^2 \Psi_j^k}{\d \phi \d \theta}  &= \pm \sin \omega_{j,2}^k \frac{\d \omega_{j,2}^k}{\d \phi}\\
\frac{\d^2 \Psi_j^k}{ \d \theta \d \phi} &= \pm \sin \omega_{j,1}^k \frac{\d \omega_{j,1}^k}{\d \theta}.
\end{split}
\end{align}
Setting the two equations in \eqref{Clairaut} equal implies that
\begin{align*}
\frac{1}{\sin \omega_{j,1}} \Big| \frac{\d \omega_{j,2} }{ \d \phi} \Big| = \frac{1}{\sin \omega_{j,2}} \Big| \frac{\d \omega_{j,2} }{ \d \phi} \Big|,
\end{align*}
which combined with \eqref{k58}, concludes the proof of the lemma.
\end{proof}
\begin{rema}
In \cite{Vig18}, complete integrability of the ellipse actually implies that $\omega_{j,1} = \omega_{j,2}$ for elliptical billiards i.e. every geodesic loop is in fact a periodic orbit. In that case, the angular derivative $\d \omega / \d \theta$ appearing in Lemma \ref{Aj w derivatives} was explicitly calculated, using action angle coordinates and Jacobi elliptic functions. In general, it may be difficult to compute $\d \omega_1/\d \theta$ explicitly despite its relatively simple geometric interpretation.
\end{rema}

\subsection{Maslov factors}\label{Maslov factors}
To explicitly compute the Maslov factors $\sigma_j^\pm$ on $\Gamma_{\pm}^{j}$, we use an argument due to Keller (\cite{Keller}), following the presentation in Appendix B of \cite{HaHiFo18}. The free wave propagator $U(t) = e^{-i t \sqrt{-\Delta}}$ on $\R^2$ has an integral kernel given by
\begin{align}
U(t,x,y) = (2\pi)^{-2} \int_{\R_\xi^2} e^{i(\langle x-y, \xi \rangle - |\xi| t)}d\xi|dx \wedge dy|^{1/2},
\end{align}
considered as a distributional half density (cf. Section \ref{Fourier Integral Operators}). Let $e_1 = (y-x)/|y - x|$ and $e_2 = J e_1$, where $J$ is a $\pi/2$ counterclockwise rotation. With respect to this basis, we may write $\xi = \tau e_1 + \rho e_2$ and hence
$$
U(t,x,y) = (2\pi)^{-2} \int \int e^{i (|x - y| \tau - t \sqrt{\tau^2 + \rho^2})} d\tau d \rho |dx \wedge dy|^{1/2}.
$$
We see that stationary points of the phase occur precisely when $|x-y| = t$, $\tau > 0$ and $\rho = 0$. Applying the method of stationary phase in the variable $\rho$, we find that
\begin{align}\label{conormal}
U(t,x,y) = (2\pi)^{-3/2} \int e^{i (|x - y| - t)\tau } e^{-i\pi/4} \left(\frac{\tau}{t} \right)^{1/2} d\tau |dx \wedge dy|^{1/2}
\end{align}
to leading order. The reduction in number of phase variables in formula \ref{conormal} expresses $U(t,x,y)$ as a classical conormal distribution with principal symbol
\begin{align}\label{psym}
e^{-i\pi/4} \left(\frac{\tau}{t} \right)^{1/2} |dt \wedge d \tau \wedge ds \wedge dy|^{1/2} \in S^{1/2}(N^*\{|x-y| = t\}) \otimes \Omega_{1/2},
\end{align}
where $\Omega_{1/2}$ is the space of positive half densities on $\R \times \R^2 \times \R^2$, $s$ is an arclength coordinate on the hypersurface $|x - y| = t$ (a distance sphere) and $\tau$ is the symplectically dual coordinate to $|x - y| - t$.
\\
\\
Hence, the Maslov indices on $\Gamma_{\pm}^0 = N^*\{|x-y| = t: \pm \tau > 0 \}$ are given by $\sigma_0^{\pm} = \pm 1$. It is shown in Section 5 of \cite{GuMe79b} that after a reflection at the boundary, the Maslov factors remain unchanged. Hence, $\sigma_j^\pm = \pm 1$ for all $j \in \Z$. Both of the propagators corresponding to $\Gamma_{+}^j$ and $\Gamma_-^j$ contribute to the wave trace singularity near $[t_j,T_j]$, owing to the two modes of propagation:
\begin{align}\label{two branches}
S(t) = \frac{e^{i t \sqrt{-\Delta}} - e^{- i t \sqrt{-\Delta}}}{2 i \sqrt{-\Delta} }.
\end{align}
Hence, we multiply each $\pm$ branch of the Lagrangian distributions $S_{j, \pm}$ by $e^{\pm i \pi/ 4}$, which explains the real parts taken in Theorem \ref{Main theorem}, Theorem \ref{Rational caustic theorem} and Corollary \ref{Ellipse corollary}. Noting that the principal symbol of $\sqrt{-\Delta}$ is $|\xi|$ and $\tau = \pm |\eta| = \pm |\xi|$ on the canonical relations $\Gamma_{\pm}^0$, equations \eqref{psym}, \eqref{two branches} and the convention \ref{order of a lagrangian distribution} show that the principal symbol of $S(t)$ on $\Gamma^0$ has order
\begin{align*}
\frac{1}{2} + \frac{1}{2} + (2 + 3)/ 4 - 1 = - 5/4.
\end{align*}
Applying an elliptic parametrix for $(-\Delta_{\R^2})^{-1/2} \in \Psi^{-1}$ to the free sine wave $S(t)$ shows that $E(t) \in I^{-1/4}(\R^2 \times \R^2; N^*(\{|x - y| = t\})')$, corroborating the order appearing in Theorem \ref{HRP}.
%our parametrix is an integral over $\tau \in (-\infty, \infty)$ rather than $(0,\infty)$. Noting that the principal symbol in \eqref{D3} is even in $\xi$, we see that the resulting distribution is purely real, which explains the real part in Theorem \ref{Main theorem}.

%\subsection{Proof of Theorem \ref{Rational caustic theorem}}

%\subsection{Proof of Corollary \ref{Ellipse corollary}}

\subsection{Proofs of Theorem \ref{Main theorem}, Theorem \ref{Rational caustic theorem} and Corollary \ref{Ellipse corollary}}
Theorem \ref{Main theorem} readily follows from the formula \ref{parametrix with 8 Dirichlet}, Lemma \ref{Aj w derivatives} and the computations in Section \ref{Maslov factors}. We now show how Theorem \ref{Rational caustic theorem} and Corollary \ref{Ellipse corollary} follow directly from Theorem \ref{Main theorem}. Recall from Section \ref{Billiards} that a caustic is a smooth curve $\mathcal{C}$ in $\Omega$ such that every link tangent to $\mathcal{C}$ remains tangent to $\mathcal{C}$ after a reflection at the boundary. It is well known (see \cite{Katok}, \cite{Tabachnikov}) that if one periodic orbit is tangent to a caustic $\mathcal{C}$, then every orbit tangent to $\mathcal{C}$ is in fact periodic, with the same period (number of bounces) and winding number. In this case, we say $\mathcal{C}$ is a rational caustic and use $\omega(\mathcal{C})$ to denote the (rational) rotation number of any orbit $\gamma$ tangent to $\mathcal{C}$. Hence, rational caustics correspond to highly degenerate periodic orbits in the sense that they are not isolated and $1$ is an eigenvalue of the Poincar\'e map. Let $L_j$ be the length of a periodic orbit which is tangent to a caustic $\mathcal{C}_j$, making $j$ reflections at the boundary and one rotation. As all orbits tangent to $\mathcal{C}$ are periodic orbits, each $q \in \d \Omega$ is a critical point of the $j$-loop function, i.e. $\d_q \Psi_j(q,q) = 0$. In this case, the length function $\Psi_j(q,q)$ appearing in the phase of $\text{Tr} \cos t \sqrt{-\Delta}$ (as in Theorem \ref{Main theorem}) is the constant function $L_j$, i.e. every $j$ reflection loop is in fact a periodic orbit of length $L_j$. Assuming the noncoincidence condition \eqref{NCC} on $\Omega$, all periodic orbits of length $L_j$ arise in this way. Tangency to $\mathcal{C}$ also implies that the angles $\omega_1$ and $\omega_2$ in the amplitude are equal. Hence, the wave trace in Theroem \ref{Main theorem} is given  by
\begin{align}\label{almost caustic}
\Re \left\{(-1)^j e^{-i \pi / 4} {4} \int_{\d \Omega} \int_{0}^\infty e^{i \xi (t- L_j)} |\xi|^{1/2} \sin^{3/2} \omega_1 \left|\frac{\d \omega_1}{\d \theta}\right|^{1/2} X \cdot N d\xi dq\right\}.
\end{align}
As the $dq$ and $d\xi$ integrals can be separated, we obtain the Fourier transform of the homogeneous distribution
\begin{align*}
\chi_+^{3/2}(\xi) = \begin{cases}
\xi^{3/2} & \xi \geq 0\\
0 & \xi < 0.
\end{cases} 
\end{align*}
One can define $\chi_+^a$ similarly as an $L_\text{loc}^1$ function for $\Re a > -1$ and these distributions fact be analytically continued to a larger region of $a \in \C$. It is shown in \cite{Ho90} (Chapter 7) that the Fourier transform of $\chi_+^a$ (with dual variable $t$) is given by $e^{-i\pi (a+1)/2} (t - i 0)^{-a-1}$. The proof of Theorem \ref{Rational caustic theorem} is concluded by evaluating the Fourier transform in equation \eqref{almost caustic} at the point $t - L_j$.
\\
\\
In the special case of an ellipse, the billiard flow is known to be completely integrable and each confocal ellipse is in fact a caustic. Moreover, Poncelet's thoerem (see \cite{Poncelet1}, \cite{Poncelet2}) implies that all periodic orbits tangent to a given confocal ellipse have the same length. Hence, all periodic orbits in the ellipse correspond to rational caustics. It is also known (see \cite{GuMe79a}) that ellipses satisfy the noncoincidence condtion \ref{NCC} and hence, Theorem \ref{Rational caustic theorem} applies. Calculations from the author's previous work using action-angle coordinates and Jacobi elliptic function theory allow for the explicit computation of $\d \omega / \d \theta$ appearing in the integrand of \eqref{almost caustic} (see Section 5.5 of \cite{Vig18}). As the boundary of an ellipse
$$
\left\{(x,y): \frac{x^2}{a^2} + \frac{y^2}{b^2} \leq 1 \right\}
$$
is easily parametrized by $(a \cos \phi, b \sin \phi)$ for $\phi \in [0,2 \pi)$, the quantity $X(q) \cdot N(q) dq$ can be explicitly calculated. Combining these observations with Theorem \ref{Rational caustic theorem}, we obtain the formula appearing in Corollary \ref{Ellipse corollary}.

\section{An auxiliary check on the order of $S_j^k$}\label{Check on the order}
\noindent Given a discrepancy in the works \cite{MaMe82} and \cite{Popov1994}, we provide an additional check on the order of $a_0^j$ in Theorem \ref{Main theorem}. Assume $\Omega$ satisfies the noncoincidence condition \eqref{NCC} and let $\rho \in \mathcal{S}(\R)$ be a test function such that $\text{Supp} {\rho} \subset [t_j - \epsilon ,T_j + \epsilon]$, where $\epsilon$ is sufficiently small to ensure $\text{Supp} \rho \cap \text{LSP}(\Omega) = [t_j, T_j]$. Let $L_j \in [t_j, T_j]$ denote the length of a periodic orbit of rotation number $1/j$. We will compute the quantity
\begin{align}\label{SP1}
\mathcal{F}({{\rho}(t) \text{Tr} \cos t \sqrt{-\Delta}} ) = \int e^{-i\lambda t} \rho(t) \text{Tr} \cos t \sqrt{-\Delta} dt
\end{align}
in two different ways. Recalling our parametrix for $\cos t \sqrt{-\Delta}$, we see that \eqref{SP1} is given by
\begin{align}\label{SP2}
\sum_{\pm} \int_{\R_t} \int_{0}^\infty \int_{\d \Omega} e^{\pm i \tau (t - \Psi_j(q,q)) - \lambda t} \rho(t) \tau^m a^j(q) dq d \tau dt,
\end{align}
where $m$ is the purported order of $a^j(q)$. Changing variables by $\xi = \tau/\lambda$, we see that \eqref{SP2} becomes
\begin{align}\label{SP5}
\sum_{\pm} \int_{\R_t} \int_{0}^\infty \int_{\d \Omega} e^{\pm i \lambda \xi (t - \Psi_j(q,q)) - \lambda t} \rho(t) \lambda^{m+1} \xi^m a_0(q) dq d \xi dt.
\end{align}
To understand the $\lambda$ asymptotics of this oscillatory integral, we apply the method of stationary phase. First assume $L_j$ is a simple nondegenerate length. On the critical set, $ d_{t,\xi,q} (\xi(t - \Psi_j(q,q)) - t)  = 0$, which implies
\begin{align*}
\begin{cases}
t = \Psi_j(q,q)\\
\xi = 1\\
d_q \psi_j(q,q) = 0.
\end{cases}
\end{align*}
Hence, \eqref{SP1} is given by
%\begin{align}\label{SP3}
%2\pi \lambda^{m}  \int_{\d \Omega} e^{i \lambda \Psi_j(q,q)} \rho(\Psi_j(q,q)) a_0(q)dq  + o\left( \lambda^{m}\right).
%\end{align}
%Once again applying the method of stationary phase to the oscillatory integral above, we see that \eqref{SP3} is given by
\begin{align}\label{SP4}
(2\pi)^{3/2} \lambda^{m - 1/2} \sum_{q: d_q \Psi_j(q,q) = 0} \frac{e^{ i\pi/4\text{sgn Hess}(\Psi_j(q,q)) }\rho(\Psi_j(q,q))a_0(q)}{| \d_q^2(\Psi_j(q,q))|^{1/2}} + o(\lambda^{m-1/2}).
\end{align}
%\marginpar{\red{Worth making a commment that $\det(1 - P_\gamma) = -\det \text{Hess}\Psi_j^k (b_1 \cdots b_n)^{-1}$}}
Recall that periodic orbits of rotation number $1/j$ arise from critical points of the $j$-loop function, i.e. $d_q \Psi_j(q,q) = 0$ implies that the geodesic loop of $j$ reflections based at $q$ is actually a periodic orbit. Since $L_j$ was assumed to be simple, corresponding to a unique nondegenerate orbit, there are precisely $j$ such boundary points $q$ and the sum in \eqref{SP4} is actually finite. It is shown in Theorem 3 of \cite{KozTresh89} that nondegeneracy of $\gamma$ also implies $\d_q^2(\Psi_j(q,q)) \neq 0$.
%Denote by $P_\gamma$ the linearized Poincar\'e map associated to $\gamma$ and by $H$ the Hessian of the length functional in angular coordinates evaluated at $\gamma$. It is shown in \cite{KozTresh89},
%\begin{equation}
%\det(I - P_\lambda) = -\det(-H) \prod_{k = 1}^{j} \left( \frac{\d^2 |q(\phi_k+1) - q(\phi_k)|}{\d \phi_k \d \phi_k{k+1}}\right)^{-1}
%\end{equation}
Now, by the formulas in \cite{GuMe79b}, we know that for a simple length $L_j$ corresponding to a nondegenerate periodic orbit of $j$ reflections, the leading asymptotic of the wave trace modulo Maslov factors is given by
\begin{align}\label{GuMePoisson}
\sum \cos t \lambda_j = \frac{L_j}{|\det(I - P_\gamma)|^{1/2}} (t - L_j + i0)^{-1} \mod L_{loc}^1(\R),
\end{align}
where $P_\gamma$ is the Poincare map associated to the unique periodic orbit $\gamma$ of length $L_j$.  Formulas in \cite{Ho90} tell us that the Fourier transform of the righthand side of \eqref{GuMePoisson} is a constant multiple of the Heaviside function. Comparing degrees of homogeneity in \eqref{SP4} and \eqref{GuMePoisson} immediately implies that $m = 1/2$.
\\
\\
If there are infinitely many critical points of $\Psi_j(q,q)$ in the phase of \eqref{SP4}, the analysis is more subtle. For example, Poncelet's theorem for elliptical billiards actually implies that periodic orbits of a fixed length and rotation number come in one parameter families. For $j$ sufficiently large, every boundary point is the base point for a unique periodic orbit tangent to a single confocal ellipse, making $j$ reflections and a single rotation. The lengths of these orbits are independent of the base point, which implies $d_q \Psi_j(q,q)$ vanishes identically. In this case, $\Psi_j(q,q) = L_j$ and applying stationary phase to \eqref{SP5} yields
\begin{align}\label{SP3}
2\pi \lambda^{m} e^{i \lambda L_j}  \rho(L_j) \int_{\d \Omega}  a_0(q)dq  + o\left( \lambda^{m}\right).
\end{align}
While the expansion \eqref{GuMePoisson} is no longer valid for high length spectral multiplicity, the formulas in \cite{HeZe12} and \cite{Vig18} show that the \textit{variation} of the wave trace near such a period in the length spectrum is a distribution of the form $d_j \Re (t- L_j + i0)^{-5/2}$ for some constant $d_j$. %deforming the domain will change the length of such a periodic orbit. If we naively differentiate the leading order singularity in Theorem \eqref{Rational caustic theorem}, then modulo lower order terms, we obtain $\dot{L_j} c_j \Re(t- L_j + i0)^{-5/2}$, where $\dot{L_j}$ is the variation of $L_j$ in the length spectrum. 
Formally, the wave trace has one higher degree of regularity than its variation, which suggests that the wave trace is of the form $c_j (t - L_j +i 0)^{-3/2}$, in agreement with Theorem \ref{Rational caustic theorem} and Corollary \ref{Ellipse corollary}. Comparing this with the asymptotics in \eqref{SP3} and formulas in \cite{Ho90} for the Fourier transform of homogeneous distributions, we see again that $m = -1 - (-3/2) = 1/2$. The order of $a_0^j$ was also confirmed in the recent paper \cite{HeZe19}.% A more rigorous auxillary check on the order of the wave trace in Theorem \ref{Main theorem} is provided in Section \ref{Check on the order}, using the method of stationary phase and the well known trace formula in \cite{GuMe79b}.

%On the other hand, we know that the Fourier transform of the wave trace is given by the symmetric spectral measure for $\sqrt{-\Delta}$:
%\begin{align}
%\mathcal{F}(\text{Tr} \cos t \sqrt{-\Delta})(\lambda) = \frac{1}{2} \sum_{j, \pm} \delta(\lambda \pm \lambda_j).
%\end{align}
%In this setting, Weyl's law gives the spectral cluster estimate
%\begin{align}
%\begin{split}
%\# \{ j : \lambda - 1 \leq \lambda_j \leq \lambda +1 \}  = (2\pi)^{-2} \text{Vol}_{T^*\Omega} \{ \lambda - 1 \leq |\xi| \leq \lambda + 1 \} + o(1) \approx \frac{|\Omega|}{\pi} \lambda,
%\end{split}
%\end{align}
%which implies that $m = 1$.

\section{Acknowledgements} \noindent The author would like to thank Hamid Hezari for his many suggestions throughout this project and the anonymous referees for pointing out the Friedlander model and the commutator identity appearing Section 6 in place of Hadamard's variational formula.

\bibliographystyle{alpha}
\bibliography{WTBBreferences}

\end{document}